\documentclass[12pt,reqno]{amsart}
\usepackage{a4wide}
\usepackage{amsmath}
\usepackage{amsthm}
\usepackage{amssymb}
\usepackage{paralist}
\usepackage[misc]{ifsym}
\usepackage{epsfig}
\usepackage{epstopdf}
\usepackage[colorlinks=true]{hyperref}

\hypersetup{urlcolor=blue, citecolor=red}

\allowdisplaybreaks

\newtheorem{theorem}{Theorem}[section]

\newtheorem{lemma}[theorem]{Lemma}
\newtheorem{proposition}[theorem]{Proposition}

\theoremstyle{definition}

\newtheorem{remark}[theorem]{Remark}

\newcommand{\R}{{\mathbb R}}

\title[A critical fractional Schr\"{o}dinger equation]
{A new type of bubble solutions for a critical fractional Schr\"{o}dinger equation}

\author[Fan Du, Qiaoqiao Hua and Chunhua Wang]{}



\thanks{$^*$Corresponding author: Chunhua Wang}

\begin{document}
\maketitle
\centerline{\scshape
Fan Du$^{{\href{mailto:fdu@mails.ccnu.edu.cn}{\textrm{\Letter}}}1}$,
Qiaoqiao Hua$^{{\href{mailto:hqq@mails.ccnu.edu.cn}{\textrm{\Letter}}}1}$
and Chunhua Wang$^{{\href{mailto:chunhuawang@ccnu.edu.cn}{\textrm{\Letter}}}*1,2}$}

\medskip

{\footnotesize
 \centerline{$^1$School of Mathematics and Statistics, Central China Normal University}
 \centerline{ Wuhan, 430079, China}
}

\medskip

{\footnotesize
 \centerline{$^2$Hubei Key Lab--Math. Sci., Central China Normal University} \centerline{ Wuhan 430079, China, Wuhan, 430079, China}
}

\begin{abstract}
We consider the following critical fractional Schr\"{o}dinger equation
\begin{equation*}
  (-\Delta)^s u+V(|y'|,y'')u=u^{2_s^*-1},\quad u>0,\quad y =(y',y'') \in \R^3\times\R^{N-3},
\end{equation*}
where $N\geq 3,s\in(0,1)$, $2_s^*=\frac{2N}{N-2s}$ is the fractional critical Sobolev exponent and $V(|y'|,y'')$ is a bounded non-negative function in $\R^3\times\R^{N-3}$. If $r^{2s}V(r,y'')$ has a stable critical point $(r_0,y_0'')$ with $r_0>0$ and $V(r_0,y_0'')>0$, by using a finite-dimensional reduction method and various local Pohozaev identities, we prove that the problem above has a new type of infinitely many solutions which concentrate at points lying
on the top and the bottom of a cylinder. And the concentration points of the bubble solutions include saddle points of the function $r^{2s}V(r,y'')$. We have to overcome some difficulties caused by the non-localness of the fractional Laplacian.
\end{abstract}

\section{Introduction and the main results}\label{sec1}
\setcounter{equation}{0}

In this paper, we consider the following nonlinear elliptic problem
\begin{equation}\label{1.1}
  (-\Delta)^s u+V(y)u=u^{2_s^*-1},\quad u>0,\quad y  \in \R^N,
\end{equation}
where $s\in(0,1)$, $N\geq3$, $2_s^*=\frac{2N}{N-2s}$ is the fractional critical Sobolev exponent and $V(y)$ is non-negative and bounded potential. For any $s \in (0,1)$, $(-\Delta)^s$ is the fractional Laplacian, which is a nonlocal operator defined as
\begin{equation}\label{1.2}
  \begin{split}
  (-\Delta)^s u(y)&=c(N,s)P.V.\int_{\R^N}\frac{u(y)-u(x)}{|x-y|^{N+2s}}dx\\
  &=c(N,s)\lim\limits_{\varepsilon \to 0^+} \int_{\R^N \backslash B_\varepsilon(y)} \frac{u(y)-u(x)}{|x-y|^{N+2s}} dx,
  \end{split}
\end{equation}
where $P.V.$ is the Cauchy principal value and $c(N,s)$ is a positive dimensional constant that depends on $N$ and s.
This operator is well defined in $C_{loc}^{1,1}(\R^N)\cap \mathcal{L}_s(\R^N)$, where
\begin{equation*}
 \mathcal{L}_s(\R^N)= \bigg\{ u \in L_{loc}^1(\R^N): \int_{\R^N} \frac{|u(y)|}{1+|y|^{N+2s}}dy < \infty \bigg\}.
\end{equation*}
For more details on the fractional operator and fractional Sobolev space, we refer to \cite{ambrosio2021Nonlinear,chen2020fractional,di2012hitchhiker's} and the references therein.

In recent years, there has been a great deal of interest in fractional Laplacian. One of the main advantages of the fractional Laplacian is its ability to model anomalous diffusion, such as plasmas, flames propagation and chemical reactions in liquids. Additionally, the fractional Laplacian is used to model quasi-geostrophic flows, turbulence and water waves, molecular dynamics and relativistic quantum mechanics of stars (see \cite{bouchaud1990anomalous,caffarelli2010drift,tarasov2006fractional} and the references therein). In probability and finance, the fractional Laplacian plays an essential role in the theory of L\'{e}vy processes, and it can be understood as the infinitesimal generator of a stable L\'{e}vy diffusion process (see \cite{applebaum2009levy}). This connection with L\'{e}vy processes makes the fractional Laplacian a useful tool for modeling various financial products, such as American options (see \cite{di2012hitchhiker's}).

Solutions of problem \eqref{1.1} are related to the existence of standing wave solutions to the following fractional Schr\"odinger equation
\begin{equation}\label{0.1}
\left\{
  \begin{array}{ll}
   i\partial_t\Psi+ (-\Delta)^s\Psi=F(x,\Psi), \quad &\text{~in~} \R^N,\\
    \lim\limits_{|x|\to\infty}|\Psi(x,t)|=0, \quad &\text{~for~all~} t>0.
  \end{array}
\right.
\end{equation}
That is, solutions with the form $\Psi(x,t)=e^{-ict}u(x)$, where $c$ is a constant.

When $s=1$, Chen, Wei and Yan \cite{chen2012infinitely} considered the following nonlinear elliptic equation
\begin{equation}\label{0.A}
  -\Delta u+V(|y'|,y'')u=u^{\frac{N+2}{N-2}},\quad u>0,\quad u\in H^1(\R^N).
\end{equation}
They proved that \eqref{0.A} has infinitely many non-radial solutions if $N \geq 5$, $V (y)$ is radially symmetric and $r^2V(r)$ has a local maximum point, or a local minimum point $r_0 > 0$ with $V(r_0)>0$. Later, Peng, Wang and Wei \cite{peng2019constructing} constructed infinitely many solutions on a circle under a weak symmetric condition of $V(y)$, where they only required that $r^2V(r,y'')$ has a stable critical point $(r_0,y^{''}_0)$ with $r_0>0$ and $V(r_0,y^{''}_0)>0$. In \cite{guo2020solutions}, Guo, Liu and Nie proved that problem \eqref{1.1} had infinitely many solutions concentrated on a circle.
Recently, in \cite{duan2023doubling}, Duan, Musso and Wei constructed infinitely many solutions, where the bubbles were concentrated at points lying on the top and the bottom circles of a cylinder.

In this paper, under a class of weaker symmetric conditions for $V(y)$, we obtain a new type of infinitely many solutions for problem \eqref{1.1} combining a finite-dimensional reduction method and various local Pohozaev identities. Motivated by \cite{duan2023doubling,guo2020solutions,peng2018construction}, we consider the case $V(y)=V(|y'|, y'')=V(r, y'')$, where $y=(y', y'') \in \R^3 \times \R^{N-3}$.
We assume that $V(y)\geq0$ is bounded and satisfies

$(V)$ $r^{2s}V(r,y'')$ has a stable critical point $(r_0,y''_0)$ in the following sense: $r^{2s}V\linebreak (r,y'')$ has a  critical point $(r_0,y''_0)$ satisfying $r_0 > 0$, $V(r_0,y''_0) >0$ and
      $$\text{deg}\big(\nabla\big(r^{2s}V(r,y''),(r_0,y''_0)\big)\big)\neq 0.$$

For $5\leq N\leq 8$, we assume that $s$ satisfies
\begin{equation}\label{s}
\left\{
  \begin{array}{ll}
    \frac{N+3-\sqrt{N^2-2N+9}}{4}<s<\frac{3(N-1)-\sqrt{N^2-2N+9}}{8}, & N=5; \\
    \frac{N(\sqrt{8N-11}-1)}{8N-12}<s<1, & 6\leq N\leq 8.
  \end{array}
\right.
\end{equation}

The main result of this paper is the following:

\begin{theorem}\label{theo1.2}
Suppose that $V>0$ is bounded and belongs to $C^1$. If $V(|y'|,y'')$ satisfies $(V)$, $5\leq N\leq 8$ and $s$ satisfies \eqref{s}, then problem \eqref{1.1} has infinitely many non-radial solutions, whose energy can be made arbitrarily large.
\end{theorem}

Now, we outline the main idea in the proof of Theorem \ref{theo1.2} and discuss the main difficulties in the proof of desired results above.

We will prove Theorem \ref{theo1.2} by a finite-dimensional reduction method, combined with various local Pohozaev identities. The finite-dimensional reduction method has been extensively used to construct solutions for equations with critical growth, see \cite{cao2002scalar,chang1991perturbation,del2015intermediate,guo2013infinitely,guo2016infinitely,noussair2001scalar,yan2000concentration} and the references therein. In \cite{wei2010infinitely}, Wei and Yan used the number of the bubbles of the solutions as the parameter to construct the bubbling solutions for a class of prescribed scalar curvature problem. After that, there are a number of researches focusing on looking for infinitely many solutions for non-compact elliptic problems, see \cite{chen2012infinitely,del2015intermediate,guo2016infinitely,guo2017local,peng2019constructing,peng2018construction,wang2010neumann} and the references therein.

The main purpose of this paper is to construct a new type of bubble solutions like \cite{duan2023doubling}, where they constructed $2k$-bubble solutions concentrating at points lying on the upper and lower circles of a cylinder. In order to release the assumptions about the function $V(y)$, inspired by \cite{guo2020solutions,peng2018construction}, we try to construct solutions by using various local Pohozaev identities to find algebraic equations, which determine the location of the bubbles. Hence, we can construct $2k$-bubble solutions, symmetric with respect to the third coordinate, concentrating at even the saddle points of $V(y)$, which just has a weaker sense of symmetry. However, in the present paper, since the bubbles of our solutions may be prettily close when $\bar{h}$ tends to 0, we need to do more delicate computations and obtain more precise estimates. We will discuss this in more details later.

Before we close this introduction, let us outline the main idea in the proof of Theorem \ref{theo1.2}.

Denote by $D^{s,2}(\R^N)$ the completion of $C_c^{\infty}(\R^N)$ with respect to the Gagliardo semi-norm
\begin{equation*}
  \| (-\Delta)^{\frac{s}{2}}u\|_{L^2(\R^N)}
  =\Big(\int\int_{\R^{2N}}\frac{|u(x)-u(y)|^2}{|x-y|^{N+2s}}dxdy\Big)^{\frac{1}{2}},
\end{equation*}
or equivalently (see \cite{Pezzo2020Spectrum}), when $N>2s$,
\begin{equation*}
  D^{s,2}(\R^N)=\Big\{u \in L^{2^*_s}(\R^N):\| (-\Delta)^{\frac{s}{2}}u\|_{L^2(\R^N)}<\infty\Big\}.
\end{equation*}

We will construct solutions in the following space:
$$H^s(\R^N)=\Big\{u \in D^{s,2}(\R^N): \int_{\R^N} V(y)u^2 dy < +\infty\Big\},$$
with the norm:
$$\|u\|_{H^s(\R^N)}=\Big(\|(-\Delta)^{\frac{s}{2}}u\|_{L^2(\R^N)}^2+\int_{\R^N}V(y)u^2 dy\Big)^\frac{1}{2}.$$

We define the functional $I$ on $H^s(\R^N)$ by:
\begin{equation}\label{L.1}
  I(u)=\frac{1}{2}\int_{\R^N}\big(|(-\Delta)^{\frac{s}{2}}u|^2 + V(|y'|,y'')u^2\big) dy- \frac{1}{2_s^*}\int_{\R^N}(u)_+^{2_s^*} dy,
\end{equation}
where $(u)_+ = \max\{u,0\}$. Then solutions of problem \eqref{1.1} correspond to the critical points of the functional $I$.

It is well known that the following functions
\begin{equation}\label{1.04}
U_{x,\lambda}(y)=\frac{C_N\lambda^{\frac{N-2s}{2}}}{(1+\lambda^2|y-x|^2)^{\frac{N-2s}{2}}}, \quad \lambda >0, \quad x\in\R^N,
\end{equation}
where $C_N=(4^s\gamma_0)^{\frac{N-2s}{4s}}$ with $\gamma_{0}=\frac{\Gamma(\frac{N+2s}{2})}{\Gamma(\frac{N-2s}{2})}$, are the unique positive solutions, up to the translation and scaling, of the problem (see \cite{lieb1983sharp})
\begin{equation}\label{1.4}
  (-\Delta)^s u=u^{\frac{N+2s}{N-2s}}, \quad u>0 \quad \text{in}~ \R^N.
\end{equation}
Moreover, the functions
\begin{equation*}
  Z_0(y)=\frac{\partial U_{x,\lambda}(y)}{\partial \lambda}\Big|_{\lambda=1}, \quad Z_i(y)=\frac{\partial U_{x,\lambda}(y)}{\partial y_i}, \quad i=1,2,\cdots,N
\end{equation*}
solving the linearized problem
\begin{equation*}
  (-\Delta)^s \phi =(2_s^*-1)U_{x,\lambda}^{2_s^*-2}\phi,\quad \phi>0 \quad \text{in}~\R^N,
\end{equation*}
are the kernels of the linearized operator associated with problem \eqref{1.4}.

Define
\begin{equation*}
  \begin{split}
H_s=\Big\{u : &u \in H^s(\R^N), u \text{ is even in}~ y_2, y_3, \\ &u(r\cos\theta,r\sin\theta,y_3,y'')=u\Big(r\cos\big(\theta+\frac{2j\pi}{k}\big),r\sin\big(\theta+\frac{2j\pi}{k}\big),y_3,y''\Big)\Big\},
  \end{split}
\end{equation*}
where $r=\sqrt{y_1^2+y_2^2}$ and $\theta=\arctan\frac{y_2}{y_1}$.

Let
\begin{equation*}
  \left\{
  \begin{array}{ll}
    x_j^+=\big(\bar{r} \sqrt{1-\bar{h} ^2} \cos \frac{2(j-1)\pi}{k}, \bar{r} \sqrt{1-\bar{h}^2}\sin\frac{2(j-1)\pi}{k},\bar{r}\bar{h},\bar{y}''\big), & j=1,\cdots,k, \vspace{0.15cm}\\
    x_j^-=\big(\bar{r}\sqrt{1-\bar{h}^2}\cos\frac{2(j-1)\pi}{k},\bar{r}\sqrt{1-\bar{h}^2}\sin\frac{2(j-1)\pi}{k},-\bar{r}\bar{h},\bar{y}''\big), & j=1,\cdots,k,
  \end{array}
\right.
\end{equation*}
where $\bar{y}''$ is a vector in $\R^{N-3}$, $\bar{h}$ goes to 0 and $(\bar{r},\bar{y}'')$ is close to $(r_0,y_0'')$.

The idea of constructing solutions is to glue some $U_{x_j^\pm,\lambda}$ as an approximation solution. In order not only to deal with the slow decay of this function when $N$ is not big, but also to simplify some computations, we introduce the smooth cut off function $\eta(y)=\eta(|y'|,y'')$ satisfying $\eta=1$ if $|(r,y'')-(r_0,y_0'')|\leq\sigma$, $\eta=0$ if $|(r,y'')-(r_0,y_0'')|\geq 2\sigma$, $0\leq\eta\leq1$ and $|\nabla \eta|\leq C$, where $\sigma>0$ is a small constant such that $r^{2s}V(r,y'')>0$ if $|(r,y'')-(r_0,y_0'')|\leq10\sigma$. Moreover, we can assume that $\eta$ is even in $y_2$ and $y_3$.

Denote
\begin{equation*}
  Z_{x_j^{\pm},\lambda}=\eta U_{x_j^{\pm},\lambda},\quad Z_{\bar{r},\bar{h},\bar{y}'',\lambda}^*=\sum\limits_{j=1}^{k} U_{x_j^{+},\lambda}+\sum\limits_{j=1}^{k} U_{x_j^{-},\lambda},
\end{equation*}
and
\begin{equation*}
  Z_{\bar{r},\bar{h},\bar{y}'',\lambda}=\sum\limits_{j=1}^{k} \eta U_{x_j^{+},\lambda}+\sum\limits_{j=1}^{k} \eta U_{x_j^{-},\lambda}.
\end{equation*}
By the weak symmetry of $V(y)$, we observe that $V(x^\pm_j)=V(\bar{r},\bar{y}'')$ for $j=1,\cdots,k$.

In this paper, we assume that $k>0$ is a large integer, $\lambda \in [L_0k^{\frac{N-2s}{N-4s}},L_1k^{\frac{N-2s}{N-4s}}]$ for some constants $L_1>L_0>0$ and $(\bar{r},\bar{h},\bar{y}'')$ satisfies
\begin{equation}\label{1.7}
  |(\bar{r},\bar{y}'')-(r_0,y''_0)|\leq \theta, \quad \bar{h}\in[M_0k^{-\frac{N-2s-1}{N-2s+1}},M_1k^{-\frac{N-2s-1}{N-2s+1}}],
\end{equation}
for some constants $M_1>M_0>0$, and $\theta>0$ is a small constant.

In order to prove Theorem \ref{theo1.2}, we will show the following result.

\begin{theorem}\label{theo1.4}
Suppose that $5\leq N\leq 8$ and $s$ satisfies \eqref{s}, then there exists a positive integer $k_0>0$, such that for any integer $k>k_0$, problem \eqref{1.1} has a solution $u_k$ of the form
\begin{equation}\label{1.8}
  u_k=Z_{\bar{r}_k,\bar{h}_k,\bar{y}''_k,\lambda_k}+\varphi_k=\sum\limits_{j=1}^{k} \eta U_{x_j^{+},\lambda_k}+\sum\limits_{j=1}^{k} \eta U_{x_j^{-},\lambda_k}+\varphi_k,
\end{equation}
where $\lambda_k \in [L_0k^{\frac{N-2s}{N-4s}},L_1k^{\frac{N-2s}{N-4s}}]$, $\bar{h}_k\in[M_0k^{-\frac{N-2s-1}{N-2s+1}},M_1k^{-\frac{N-2s-1}{N-2s+1}}]$ and $\varphi_k\in H_s$. Moreover, as $k \to +\infty$, $|(\bar{r}_k,\bar{y}''_k)-(r_0,y''_0)|= o(1)$, and $\lambda_k^{-\frac{N-2s}{2}}\|\varphi_k\|_{L^{\infty}} \to 0$.
\end{theorem}

\begin{remark}\label{rema1.3.1}
If we denote $\tau=\frac{N-4s}{2(N-2s)}$, when $5\leq N\leq 8$, then the condition \eqref{s} implies that $\min\big\{\frac{N-2s}{2}-\tau,2s-\tau,\frac{2s}{N-2s}(\frac{N+2s}{2}-\tau)\big\}>\frac{2s+1}{2}$, which ensures
the critical point structure of $\frac{\partial F}{\partial \bar{h}}$ is not destroyed by $\varphi$. If we can improve the estimate of $\varphi$,
then we can get a wider range of $s$ and $N$.
\end{remark}

\begin{remark}\label{add-rem}
Since there involves a new parameter $\bar{h}$ in the points sequence of $\{x_{j}^{\pm}\}_{j=1}^{k},$
compared with \cite{guo2020solutions}, during the reduction process, we have to improve the estimate of $\varphi$(see Proposition \ref{prop2.5}).
\end{remark}

\begin{remark}\label{rema1.3}
Note that the bubble solutions $u_k$ concentrate at two kinds of points $\{x_{j}^{\pm}\}^{k}_{j=1}$ which are axisymmetric on the third coordinate and the number of bubbles can be made arbitrarily large. This blowing-up phenomenon of clustering cannot occur in the case that the potential function $V(y)$ just has isolated critical points. In particular, when taking $y_0''=0\in \R^{N-3}$, $u_k$ concentrates on a pair of symmetric points with respect to the region.
\end{remark}

By the reduction argument, we first use $Z_{\bar{r},\bar{h},\bar{y}'',\lambda}$ as an approximate solution to obtain a unique function $\varphi_{\bar{r},\bar{h},\bar{y}'',\lambda}$. Then the problem of finding critical points for $I(u)$ with the form \eqref{1.8} can be reduced to that of finding critical points of the following function
\begin{equation*}
F(\bar{r},\bar{h},\bar{y}'',\lambda):=I(Z_{\bar{r},\bar{h},\bar{y}'',\lambda}+\varphi_{\bar{r},\bar{h},\bar{y}'',\lambda}),
\end{equation*}
where $\bar{r},\bar{h},\bar{y}''$ and $\lambda$ satisfies those conditions in Theorem \ref{theo1.4}. In the second step, we solve the corresponding finite-dimensional problem to obtain a solution. However the estimate of $\varphi$ destroys the main terms in the expansions of $\frac{\partial F}{\partial \bar{r}}$ and $\frac{\partial F}{\partial \bar{y}''_j}$ (see Lemma \ref{lemB.3}). To overcome this difficulty, motivated by \cite{peng2018construction}, we will choose appropriate $(\bar{r},\bar{h},\bar{y}'')$ to satisfy the following local Pohozaev identities:
\begin{equation}\label{1.9}
  \begin{split}
    -\int_{\partial''{B_\rho^+}} t^{1-2s}\frac{\partial\tilde{u}}{\partial \nu}\frac{\partial\tilde{u}}{\partial y_i}+\frac{1}{2}\int_{\partial''{B_\rho^+}} t^{1-2s}|\nabla\tilde{u}|^2\nu_i
    =\int_{B_\rho}\big(-V(r,y'')u+(u)_+^{2_s^*-1}\big)\frac{\partial u}{\partial y_i}
  \end{split}
\end{equation}
for $i=4,\cdots,N$, and
\begin{equation}\label{1.10}
  \begin{split}
    &-\int_{\partial''{B_\rho^+}} t^{1-2s}\langle\nabla\tilde{u},Y\rangle\frac{\partial\tilde{u}}{\partial \nu}+\frac{1}{2}\int_{\partial''{B_\rho^+}}t^{1-2s}|\nabla\tilde{u}|^2\langle Y,\nu\rangle +\frac{2s-N}{2}\int_{\partial {B_\rho^+}}t^{1-2s}\frac{\partial\tilde{u}}{\partial \nu}\tilde{u}\\
    &=\int_{B_\rho}\big(-V(r,y'')u+(u)_+^{2_s^*-1}\big)\langle \nabla u,y\rangle,
  \end{split}
\end{equation}
where $u=Z_{\bar{r},\bar{h},\bar{y}'',\lambda}+\varphi_{\bar{r},\bar{h},\bar{y}'',\lambda}$, $\tilde{u}$ is the extension of $u$ (see below \eqref{3.A}) and
\begin{align*}
B_\rho&=\{y:|y-(r_0,y_0'')|\leq\rho\}\subseteq \R^N,\\[1mm]
B_\rho^+&=\{(y,t):|(y,t)-(r_0,y_0'',0)|\leq\rho,t>0\}\subseteq \R_+^{N+1},\\[1mm]
\partial'B_\rho^+&=\{(y,t):|y-(r_0,y_0'')|\leq\rho,t=0\}\subseteq \R^N,\\[1mm]
\partial''B_\rho^+&=\{(y,t):|(y,t)-(r_0,y_0'',0)|=\rho,t>0\}\subseteq \R_+^{N+1}.
\end{align*}
Moreover, $\partial B_\rho^+=\partial'B_\rho^+\cup\partial''B_\rho^+$.

For any $u \in D^{s,2}(\R^N)$, the s-harmonic extension $\tilde{u}$ is defined by
\begin{equation}\label{3.A}
  \tilde{u}(y,t)=\mathcal{P}_s[u]:=\int_{\R^N}P_s(y-\xi,t)u(\xi)d\xi, \quad (y,t)\in \R_+^{N+1}:=\R^N\times(0,\infty),
\end{equation}
where the Poisson kernel $P_s(x,t)$ is given by
\begin{equation*}
  P_s(x,t)=\beta(N,s)\frac{t^{2s}}{(|x|^2+t^2)^{\frac{N+2s}{2}}},
\end{equation*}
with constant $\beta(N,s)$ such that $\int_{\R^N}P_s(x,1)dx=1$. Then $\tilde{u} \in L^{2}(t^{1-2s},K)$ for any compact set $K$ in $\overline{\R_+^{N+1}}$, $\nabla\tilde{u} \in L^2(t^{1-2s},\R_+^{N+1})$ and $\tilde{u} \in C^\infty(\R_+^{N+1})$. Moreover, $\tilde{u}$ satisfies (see \cite{caffarelli2007extension})
\begin{equation}\label{3.B}
  \left\{
  \begin{array}{ll}
      \text{div}(t^{1-2s}\nabla\tilde{u})=0, &\quad \text{in}~ \R_+^{N+1},\vspace{0.2cm}\\
      \tilde{u}(y,0)=u(y),&\quad \text{on}~\R^N,\vspace{0.2cm}\\
    -\lim\limits_{t \to 0}t^{1-2s}\partial_t\tilde{u}(y,t)=\omega_s(-\Delta)^s u(y), &\quad \text{on}~\R^N,\vspace{0.2cm}\\
  \end{array}
\right.
\end{equation}
where $\omega_s=2^{1-2s}\Gamma(1-s)/\Gamma(s)$. It holds that
\begin{equation*}
  \|\tilde{u}\|_{L^2(t^{1-2s},\R^{N+1}_+)}=\omega_s\|u\|_{D^{s,2}(\R^N)}.
\end{equation*}
Without loss of generality, we may assume $\omega_s=1$.

We would like to point out that we obtain $(\bar{r},\bar{y}'')$ by using the reduction argument, rather than determine these parameters via computing the derivatives of the reduced function $F(\bar{r},\bar{h},\bar{y}'',\lambda)$ with respect to $\bar{r}$ and $\bar{y}''_k,k=4,\cdots,N$ directly.
We can still introduce the condition $\frac{\partial F}{\partial \bar{h}}=0$ to obtain
\begin{equation*}
  \bar{h}\sim O\Big(\frac{1}{k^{\frac{N-2s-1}{N-2s+1}}}\Big),
\end{equation*}
since $V(y)=V(|y'|,y'')$ is symmetric to the first three variables $y'=(y_1,y_2,y_3)$.

Noting that the maximum norm will not be affected by the number of the bubbles, we need to carry out the reduction procedure in a space with weighted maximum norm, similar weighted maximum norm has been used in \cite{del2003two,guo2020solutions,peng2019constructing,peng2018construction,wei2005arbitrary,wei2010infinitely}.

Our paper is organized as follows. In section 2, we perform a finite-dimensional reduction to get a finite dimensional setting. Then we prove some results for the finite dimensional problems and prove Theorem \ref{theo1.4} in section 3. In Appendix A, we give some essential estimates. The expansion of the energy for the approximate solutions are given in Appendix B and the proofs of the local Pohozaev identities for the fractional Laplacian will be proved in Appendix C.

\section{The finite-dimensional reduction}\label{sec2}

In this section, we perform a finite-\linebreak dimensional reduction by using $Z_{\bar{r},\bar{h},\bar{y}'',\lambda}$ as an approximation solution. For later calculations, we divide $\R^N$ into $k$ parts, for $j=1,\cdots,k$, \smallskip
\begin{equation*}
\begin{split}
  \Omega_j :=\Big\{y&=(y_1,y_2,y_3,y'')\in \R^3\times\R^{N-3}: \\
&\Big\langle \frac{(y_1,y_2)}{|(y_1,y_2)|},\big(\cos\frac{2(j-1)\pi}{k},\sin\frac{2(j-1)\pi}{k}\big)\Big\rangle_{\R^2}\geq\cos\frac{\pi}{k}\Big\},
\end{split}
\end{equation*}
where $\langle,\rangle_{\R^2}$ denote the dot product in $\R^2$. For $\Omega_j$, we further divide it into two separate parts:
$$\Omega_j^+=\big\{y:y=(y_1,y_2,y_3,y'') \in \Omega_j, y_3 \geq 0\big\},$$
$$\Omega_j^-=\big\{y:y=(y_1,y_2,y_3,y'') \in \Omega_j, y_3 < 0\big\}.$$
It is easy to verify that
\begin{equation*}
\R^N=\bigcup\limits_{j=1}^{k}\Omega_j,\quad \Omega_j=\Omega_j^+\cup\Omega_j^-,
\end{equation*}
and
\begin{equation*}
  \Omega_j \cap \Omega_i=\emptyset, \quad \Omega^+_j \cap \Omega^-_j=\emptyset, \quad \text{if}~i\neq j.
\end{equation*}

Let
\begin{equation*}
  \|u\|_*=\sup\limits_{y\in\R^N}\bigg(\sum\limits_{j=1}^{k}\Big(\frac{\lambda^{\frac{N-2s}{2}}}{(1+\lambda|y-x_j^+|)^{\frac{N-2s}{2}+\tau}}+\frac{\lambda^{\frac{N-2s}{2}}}{(1+\lambda|y-x_j^-|)^{\frac{N-2s}{2}+\tau}}\Big)\bigg)^{-1}|u(y)|,
\end{equation*}
and
\begin{equation*}
 \|f\|_{**}=\sup\limits_{y\in\R^N}\bigg(\sum\limits_{j=1}^{k}\Big(\frac{\lambda^{\frac{N+2s}{2}}}{(1+\lambda|y-x_j^+|)^{\frac{N+2s}{2}+\tau}}+\frac{\lambda^{\frac{N+2s}{2}}}{(1+\lambda|y-x_j^-|)^{\frac{N+2s}{2}+\tau}}\Big)\bigg)^{-1}|f(y)|,
\end{equation*}
where $\tau=\frac{N-4s}{2(N-2s)}$.

For $j=1,2,\cdots,k$, denote
\begin{equation*}
  Z_{j,1}^\pm=\frac{\partial Z_{x_j^\pm,\lambda}}{\partial\lambda},~ Z_{j,2}^\pm=\frac{\partial Z_{x_j^\pm,\lambda}}{\partial \bar{r}},~ Z_{j,3}^\pm=\frac{\partial Z_{x_j^\pm,\lambda}}{\partial \bar{h}},~ Z_{j,l}^\pm=\frac{\partial Z_{x_j^\pm,\lambda}}{\partial \bar{y}''_l},~ l=4,\cdots,N.
\end{equation*}

We also define the constrained space
\begin{equation*}
  \begin{split}
  \mathbb{H}=\Big\{v\in H_s:\int_{\R^N} Z_{x_j^+,\lambda}^{2_s^*-2}&Z_{j,l}^+ v=0 \quad \,\text{~and} \\
&\int_{\R^N} Z_{x_j^-,\lambda}^{2_s^*-2}Z_{j,l}^- v=0,~j=1,\cdots,k,~l=1,2,\cdots,N\Big\}.
  \end{split}
\end{equation*}

Now, we consider the following linearized problem
\begin{equation}\label{2.5}
  \begin{split}
  \left\{
  \begin{array}{ll}
    &(-\Delta)^s\varphi+V(r,y'')\varphi-(2_s^*-1)Z_{\bar{r},\bar{h},\bar{y}'',\lambda}^{2_s^*-2}\varphi\\
    &\quad=f+\sum\limits_{l=1}^N
c_l\sum\limits_{j=1}^k\Big(Z_{x_j^+,\lambda}^{2_s^*-2}Z_{j,l}^+ +Z_{x_j^-,\lambda}^{2_s^*-2}Z_{j,l}^-\Big), \quad\text{~in~} \R^N,\\
&\varphi \in \mathbb{H},
  \end{array}
\right.
  \end{split}
\end{equation}
for some constants $c_l$.

In the sequel of this section, we assume that $(\bar{r},\bar{y}'')$ and $\bar{h}$ satisfy \eqref{1.7}.

\begin{lemma}\label{lem2.1}
Suppose that $\varphi_k$ solves \eqref{2.5} for $f=f_k$. If $\|f_k\|_{**}$ goes to zero as $k$ goes to infinity, so does $\|\varphi_k\|_*$.
\end{lemma}
\begin{proof}
We argue by contradiction. Suppose that there exist $k \to \infty$, $f=f_k$, $\lambda=\lambda_k \in [L_0 k^{\frac{N-2s}{N-4s}},L_1k^{\frac{N-2s}{N-4s}}]$, $(\bar{r}_k,\bar{h}_k,\bar{y}''_k)$ satisfies \eqref{1.7}, $\varphi_k$ solving \eqref{2.5} for $f=f_k$, $\lambda=\lambda_k$, $\bar{r}=\bar{r}_k$, $\bar{h}=\bar{h}_k$, $\bar{y}''=\bar{y}''_k$ with $\|f_k\|_{**}\to 0$ and $\|\varphi_k\|_{*}\geq c> 0$. Without loss of generality, we may assume that $\|\varphi_k\|_{*}=1$. For simplicity, we drop the subscript $k$.

From equation \eqref{2.5}, we have
\begin{equation}\label{2.6}
\begin{split}
  |\varphi(y)|&\leq C\int_{\R^N}\frac{1}{|z-y|^{N-2s}}Z_{\bar{r},\bar{h},\bar{y}'',\lambda}^{2_s^*-2}|\varphi|dz +C\int_{\R^N}\frac{1}{|z-y|^{N-2s}}|f|dz \\
  &\quad +C\int_{\R^N}\frac{1}{|z-y|^{N-2s}}\Big|\sum\limits_{l=1}^N
c_l\sum\limits_{j=1}^k(Z_{x_j^+,\lambda}^{2_s^*-2}Z_{j,l}^+ +Z_{x_j^-,\lambda}^{2_s^*-2}Z_{j,l}^-)\Big|dz. \\
 \end{split}
\end{equation}

Using Lemma \ref{lemA.3}, we can deduce that
\begin{equation}\label{2.7}
\begin{split}
  &\int_{\R^N}\frac{1}{|z-y|^{N-2s}}Z_{\bar{r},\bar{h},\bar{y}'',\lambda}^{2_s^*-2}|\varphi|dz\\
   & \leq \|\varphi\|_*\int_{\R^N}\frac{Z_{\bar{r},\bar{h},\bar{y}'',\lambda}^{2_s^*-2}}{|z-y|^{N-2s}}\sum\limits_{j=1}^{k}\Big(\frac{\lambda^{\frac{N-2s}{2}}}{(1+\lambda|z-x_j^+|)^{\frac{N-2s}{2}+\tau}}+\frac{\lambda^{\frac{N-2s}{2}}}{(1+\lambda|z-x_j^-|)^{\frac{N-2s}{2}+\tau}}\Big)\\
   & \leq C\|\varphi\|_*\sum\limits_{j=1}^{k}\Big(\frac{\lambda^{\frac{N-2s}{2}}}{(1+\lambda|y-x_j^+|)^{\frac{N-2s}{2}+\tau+\theta}}+\frac{\lambda^{\frac{N-2s}{2}}}{(1+\lambda|y-x_j^-|)^{\frac{N-2s}{2}+\tau+\theta}}\Big),
 \end{split}
\end{equation}
where $\theta$ is a small constant.

It follows from Lemma \ref{lemA.2} that
\begin{equation}\label{2.8}
\begin{split}
&\int_{\R^N}\frac{1}{|z-y|^{N-2s}}|f|dz\\
 & \leq \|f\|_{**}\int_{\R^N}\frac{1}{|z-y|^{N-2s}}\sum\limits_{j=1}^{k}\Big(\frac{\lambda^{\frac{N+2s}{2}}}{(1+\lambda|z-x_j^+|)^{\frac{N+2s}{2}+\tau}}+\frac{\lambda^{\frac{N+2s}{2}}}{(1+\lambda|z-x_j^-|)^{\frac{N+2s}{2}+\tau}}\Big)\\
   & \leq C\|f\|_{**}\sum\limits_{j=1}^{k}\Big(\frac{\lambda^{\frac{N-2s}{2}}}{(1+\lambda|y-x_j^+|)^{\frac{N-2s}{2}+\tau}}+\frac{\lambda^{\frac{N-2s}{2}}}{(1+\lambda|y-x_j^-|)^{\frac{N-2s}{2}+\tau}}\Big).
 \end{split}
\end{equation}
From the definitions of $Z_{j,l}^\pm$, for $j=1,2,\cdots,k$, we have
\begin{equation}\label{2.9}
  |Z_{j,1}^\pm|\leq C\lambda^{-1}Z_{x_j^{\pm},\lambda},~|Z_{j,l}^\pm|\leq C\lambda Z_{x_j^{\pm},\lambda},~ l=2,\cdots,N.
\end{equation}
Combining estimates \eqref{2.9} and Lemma \ref{lemA.2}, we have
\begin{equation}\label{2.10}
\begin{split}
 &\int_{\R^N}\frac{1}{|z-y|^{N-2s}}\Big|\sum\limits_{j=1}^k(Z_{x_j^+,\lambda}^{2_s^*-2}Z_{j,l}^+ +Z_{x_j^-,\lambda}^{2_s^*-2}Z_{j,l}^-)\Big| dz\\
 & \leq C \lambda^{n_l} \int_{\R^N}\frac{1}{|z-y|^{N-2s}}\sum\limits_{j=1}^k\Big(\frac{\lambda^{\frac{N+2s}{2}}}{(1+\lambda|z-x_j^+|)^{N+2s}}+\frac{\lambda^{\frac{N+2s}{2}}}{(1+\lambda|z-x_j^-|)^{N+2s}}\Big)\\
 & \leq C \lambda^{n_l}\sum\limits_{j=1}^k\Big(\frac{\lambda^{\frac{N-2s}{2}}}{(1+\lambda|y-x_j^+|)^{\frac{N-2s}{2}+\tau}}+\frac{\lambda^{\frac{N-2s}{2}}}{(1+\lambda|y-x_j^-|)^{\frac{N-2s}{2}+\tau}}\Big),
 \end{split}
\end{equation}
where $n_1=-1$, $n_l=1$ for $l=2,\cdots,N$.

Next, we want to estimate $c_l$, $l=1,2,\cdots,N$. Multiplying equation \eqref{2.5} by $Z_{1,t}^+ (t=1,\cdots,N)$ and integrating, we see that $c_l$ satisfies
\begin{equation}\label{2.11}
\begin{split}
&\sum\limits_{l=1}^N c_l\sum\limits_{j=1}^k\int_{\R^N}\big(Z_{x_j^+,\lambda}^{2_s^*-2}Z_{j,l}^+ +Z_{x_j^-,\lambda}^{2_s^*-2}Z_{j,l}^-\big)Z_{1,t}^+\\
&=\big\langle(-\Delta)^s\varphi+V(r,y'')\varphi-(2_s^*-1)Z_{\bar{r},\bar{h},\bar{y}'',\lambda}^{2_s^*-2}\varphi,Z_{1,t}^+\big\rangle -\big\langle f,Z_{1,t}^+\big\rangle.
 \end{split}
\end{equation}
It is easy to check that
\begin{equation}\label{2.12}
  \sum\limits_{j=1}^k\int_{\R^N}\big(Z_{x_j^+,\lambda}^{2_s^*-2}Z_{j,l}^+ +Z_{x_j^-,\lambda}^{2_s^*-2}Z_{j,l}^-\big)Z_{1,t}^+
  \left\{
 \begin{array}{ll}
=(\bar{c}+o(1))\lambda^{2n_t}, & \hbox{$l=t$;}\vspace{0.15cm} \\
\leq\frac{\bar{c}\lambda^{n_l+n_t}}{\lambda^N}, & \hbox{$l\neq t$,}
\end{array}
\right.
\end{equation}
for some constant $\bar{c}>0$.

It follows from Lemma \ref{lemA.1} and \eqref{2.9} that
\begin{equation}\label{2.14}
\begin{split}
  &|\big\langle V(r,y'')\varphi,Z_{1,t}^+\big\rangle| \\
  &\leq C \lambda^{n_t}\|\varphi\|_*\int_{\R^N}\frac{\eta\lambda^{\frac{N-2s}{2}}}{(1+\lambda|y-x_1^+|)^{N-2s}} \\
  &\quad \times\sum\limits_{j=1}^{k}\bigg(\frac{\lambda^{\frac{N-2s}{2}}}{(1+\lambda|y-x_j^+|)^{\frac{N-2s}{2}+\tau}}+\frac{\lambda^{\frac{N-2s}{2}}}{(1+\lambda|y-x_j^-|)^{\frac{N-2s}{2}+\tau}}\bigg)\\
  &=C\lambda^{n_t}\|\varphi\|_* \int_{\R^N}\eta\bigg(\frac{\lambda^{N-2s}}{(1+\lambda|y-x_1^+|)^{\frac{3}{2}(N-2s)+\tau}}+\sum\limits_{j=2}^{k}\frac{\lambda^{N-2s}}{(1+\lambda|y-x_1^+|)^{N-2s}}\\
  &\quad \times\frac{1}{(1+\lambda|y-x_j^+|)^{\frac{N-2s}{2}+\tau}}
 +\sum\limits_{j=1}^{k}\frac{\lambda^{N-2s}}{(1+\lambda|y-x_1^+|)^{N-2s}}\frac{1}{(1+\lambda|y-x_j^-|)^{\frac{N-2s}{2}+\tau}}\bigg)\\
  &\leq C \lambda^{n_t}\|\varphi\|_* \bigg(\int_{\R^N}\frac{\eta\lambda^{N-2s}}{(1+\lambda|y-x_1^+|)^{\frac{3}{2}(N-2s)+\tau}}\\
  &\quad +\sum\limits_{j=2}^{k}\frac{1}{(\lambda|x_1^+-x_j^+|)^{\tau}}\int_{\R^N}\Big(\frac{\eta\lambda^{N-2s}}{(1+\lambda|y-x_1^+|)^{\frac{3}{2}(N-2s)}}+\frac{\eta\lambda^{N-2s}}{(1+\lambda|y-x_j^+|)^{\frac{3}{2}(N-2s)}}\Big)\\
  &\quad +\sum\limits_{j=1}^{k}\frac{1}{(\lambda|x_1^+-x_j^-|)^{\tau}}\int_{\R^N}\Big(\frac{\eta\lambda^{N-2s}}{(1+\lambda|y-x_1^+|)^{\frac{3}{2}(N-2s)}}+\frac{\eta\lambda^{N-2s}}{(1+\lambda|y-x_j^-|)^{\frac{3}{2}(N-2s)}}\Big)\bigg)\\
  &\leq C\|\varphi\|_* \int_{\R^N}\frac{\eta\lambda^{N-2s+n_t+\tau}}{(1+\lambda|y-x_1^+|)^{\frac{3}{2}(N-2s)}}\\
  &\leq \frac{C\lambda^{n_t}\|\varphi\|_*}{\lambda^{\min\{2s-\tau,\frac{N-2s}{2}-\tau\}}},
 \end{split}
\end{equation}
where we used the fact that, for $\gamma=\tau=\frac{N-4s}{2(N-2s)}<1$ in Lemma \ref{lemA.4} and $\bar{h}$ satisfying \eqref{1.7},
\begin{equation}\label{a2.14}
\sum\limits_{j=2}^{k}\frac{1}{(\lambda|x_1^+-x_j^+|)^{\tau}}
\leq\frac{Ck}{\lambda^{\tau}}\leq C\lambda^{\tau},
\end{equation}
and
\begin{equation}\label{a2.14.1}
\sum\limits_{j=1}^{k}\frac{1}{(\lambda|x_1^+-x_j^-|)^{\tau}}
\leq\frac{Ck}{\lambda^{\tau}}+\frac{C}{(\lambda\bar{h})^\tau}\leq C\lambda^{\tau}.
\end{equation}

Similarly, we have
\begin{equation}\label{2.15}
\begin{split}
&|\langle f,Z_{1,t}^+\rangle|\\
&\leq C\|f\|_{**}\int_{\R^N}\frac{\eta\lambda^{\frac{N-2s}{2}+n_t}}{(1+\lambda|z-x_1^+|)^{N-2s}}\\
&\quad \times\sum\limits_{j=1}^{k}\bigg(\frac{\lambda^{\frac{N+2s}{2}}}{(1+\lambda|z-x_j^+|)^{\frac{N+2s}{2}+\tau}}+\frac{\lambda^{\frac{N+2s}{2}}}{(1+\lambda|z-x_j^-|)^{\frac{N+2s}{2}+\tau}}\bigg)\\
&\leq C\lambda^{n_t}\|f\|_{**}\bigg(C'+\sum\limits_{j=2}^k\frac{1}{(\lambda|x_1^+-x_j^+|)^{\frac{N-2s}{2}+\tau-\theta_0}} + \sum\limits_{j=1}^k\frac{1}{(\lambda|x_1^+-x_j^-|)^{\frac{N-2s}{2}+\tau-\theta_0}}\bigg)\\
&\leq C \lambda^{n_t}\|f\|_{**},
 \end{split}
\end{equation}
where $\theta_0>0$ is a small constant and we used the fact that, for $\frac{N-4s}{N-2s}<\gamma<1$ in Lemma \ref{lemA.4} and $\bar{h}$ satisfying \eqref{1.7},
\begin{equation}\label{a02.14}
\sum\limits_{j=2}^{k}\frac{1}{(\lambda|x_1^+-x_j^+|)^{\gamma}}
\leq\frac{Ck}{\lambda^{\gamma}}\leq C,
\end{equation}
and
\begin{equation}\label{a02.14.1}
\sum\limits_{j=1}^{k}\frac{1}{(\lambda|x_1^+-x_j^-|)^{\gamma}}
\leq\frac{Ck}{\lambda^{\gamma}}+\frac{C}{(\lambda\bar{h})^\gamma}\leq C.
\end{equation}
On the other hand, direct calculation gives
\begin{equation}\label{2.13}
  \big\langle(-\Delta)^s\varphi-(2_s^*-1)Z_{\bar{r},\bar{h},\bar{y}'',\lambda}^{2_s^*-2}\varphi,Z_{1,t}^+\big\rangle =O\Big(\frac{\lambda^{n_t}\|\varphi\|_*}{\lambda^{s+\epsilon}}\Big).
\end{equation}
Hence,
\begin{equation}\label{2.16}
 \big\langle(-\Delta)^s\varphi+V(r,y'')\varphi-(2_s^*-1)Z_{\bar{r},\bar{h},\bar{y}'',\lambda}^{2_s^*-2}\varphi,Z_{1,t}^+\big\rangle -\big\langle f,Z_{1,t}^+\big\rangle = O\bigg(\lambda^{n_t}\Big(\frac{\|\varphi\|_*}{\lambda^{s+\epsilon}}+\|f\|_{**}\Big)\bigg),
\end{equation}
which, together with \eqref{2.11} and \eqref{2.12}, yields that
\begin{equation*}
  c_l=\frac{1}{\lambda^{n_l}}\big(o(\|\varphi\|_*)+O(\|f\|_{**})\big).
\end{equation*}
So,
\begin{equation}\label{2.17}
  \|\varphi\|_* \leq C \Bigg(\|f\|_{**}+\frac{\sum\limits_{j=1}^k\big(\frac{1}{(1+\lambda|y-x_j^+|)^{\frac{N-2s}{2}+\tau+\theta}}+\frac{1}{(1+\lambda|y-x_j^-|)^{\frac{N-2s}{2}+\tau+\theta}}\big)}{\sum\limits_{j=1}^k\big(\frac{1}{(1+\lambda|y-x_j^+|)^{\frac{N-2s}{2}+\tau}}+\frac{1}{(1+\lambda|y-x_j^-|)^{\frac{N-2s}{2}+\tau}}\big)}\Bigg),
\end{equation}
which, together with $\|\varphi\|_*=1$, yields that there is $R>0$ such that
\begin{equation}\label{2.18}
  \|\lambda^{-\frac{N-2s}{2}}\varphi(y)\|_{L^\infty(B_{R/\lambda}(x_j^*))}\geq a>0,
\end{equation}
for some $j$ with $x_j^*=x_j^+$ or $x_j^-$. But $\tilde{\varphi}(y)=\lambda^{-\frac{N-2s}{2}}\varphi(\lambda^{-1}y+x_j^*)$ converges uniformly in any compact set to a solution $u$ of
\begin{equation}\label{2.19}
  (-\Delta)^s u -(2_s^*-1)U_{0,\Lambda}^{2_s^*-2}u=0, \quad \text{in}~\R^N,
\end{equation}
for some $\Lambda \in [\Lambda_1,\Lambda_2]$ and $u$ is perpendicular to the kernel of \eqref{2.19}. As a consequence, $u=0$, which is a contradiction to \eqref{2.18}.
\end{proof}

From Lemma \ref{lem2.1}, using the same argument as in the proof of Proposition 4.1 in \cite{del2003two}, we can prove the following result.

\begin{lemma}\label{lem2.2}
There exist $k_0>0$ and a constant $C>0$, independent of $k$, such that for $k\geq k_0$ and all $f \in L^\infty(\R^N)$, problem \eqref{2.5} has a unique solution $\varphi=\mathcal{L}_k(f)$. Moreover,
\begin{equation}\label{2.20}
  \|\mathcal{L}_k(f)\|_* \leq C \|f\|_{**}, \quad |c_l|\leq \frac{C}{\lambda^{n_l}}\|f\|_{**},
\end{equation}
where $n_1=-1$, $n_l=1$ for $l=2,\cdots,N$.
\end{lemma}

Next, we consider the following problem
\begin{equation}\label{2.21}
\left\{
  \begin{array}{ll}
    \begin{aligned}
    &(-\Delta)^s(Z_{\bar{r},\bar{h},\bar{y}'',\lambda}+\varphi)+V(r,y'')(Z_{\bar{r},\bar{h},\bar{y}'',\lambda}+\varphi) \\
    &= (Z_{\bar{r},\bar{h},\bar{y}'',\lambda}+\varphi)_+^{2_s^*-1}+\sum\limits_{l=1}^N c_l\sum\limits_{j=1}^k\Big(Z_{x_j^+,\lambda}^{2_s^*-2}Z_{j,l}^+ +Z_{x_j^-,\lambda}^{2_s^*-2}Z_{j,l}^-\Big),\quad \text{in~} \R^N,
  \end{aligned}\\
  \varphi \in \mathbb{H}.
  \end{array}
\right.
\end{equation}

First, we give the main result of this section.
\begin{proposition}\label{prop2.5}
There exists a positive large integer $k_0$, such that for all $k\geq k_0$ and $\lambda \in [L_0 k^{\frac{N-2s}{N-4s}},L_1 k^{\frac{N-2s}{N-4s}}]$, $(\bar{r},\bar{h},\bar{y}'')$ satisfies \eqref{1.7}, problem \eqref{2.21} has a unique solution $\varphi=\varphi_{\bar{r},\bar{h},\bar{y}'',\lambda}$ satisfying
\begin{equation}\label{2.28}
  \|\varphi\|_{*} \leq \frac{C}{\lambda^{{\frac{2s+1}{2}}+\epsilon}}, \quad |c_l|\leq \frac{C}{\lambda^{{\frac{2s+1}{2}}+n_l+\epsilon}},
\end{equation}
where $\epsilon>0$ is a small constant.
\end{proposition}
In order to prove Proposition \ref{prop2.5}, we need several lemmas.

Rewrite \eqref{2.21} as
\begin{equation}\label{2.22}
  \left\{
  \begin{array}{ll}
  \begin{aligned}
    &(-\Delta)^s\varphi + V(r,y'')\varphi- (2_s^*-1)Z_{\bar{r},\bar{h},\bar{y}'',\lambda}^{2_s^*-2}\varphi\\
    &= \mathcal{F}(\varphi)+l_k+\sum\limits_{l=1}^N c_l\sum\limits_{j=1}^k\Big(Z_{x_j^+,\lambda}^{2_s^*-2}Z_{j,l}^+ +Z_{x_j^-,\lambda}^{2_s^*-2}Z_{j,l}^-\Big),\quad \text{in~} \R^N,
  \end{aligned}\\
\varphi \in \mathbb{H},
  \end{array}
\right.
\end{equation}
where
\begin{equation}\label{2.23}
  \mathcal{F}(\varphi)=(Z_{\bar{r},\bar{h},\bar{y}'',\lambda}+\varphi)_+^{2_s^*-1}-Z_{\bar{r},\bar{h},\bar{y}'',\lambda}^{2_s^*-1}-(2_s^*-1)Z_{\bar{r},\bar{h},\bar{y}'',\lambda}^{2_s^*-2}\varphi,
\end{equation}
and
\begin{equation}\label{2.24}
  \begin{split}
  l_k(x)&=\Big(Z_{\bar{r},\bar{h},\bar{y}'',\lambda}^{2_s^*-1}-\sum\limits_{j=1}^{k} \big(\eta U_{x_j^+,\lambda}^{2_s^*-1} +\eta U_{x_j^-,\lambda}^{2_s^*-1}\big)\Big)-V(r,y'')Z_{\bar{r},\bar{h},\bar{y}'',\lambda} -\sum\limits_{j=1}^{k}c(N,s)\\
        &\quad \times\lim\limits_{\varepsilon \to 0^+} \int_{\R^N \backslash B_\varepsilon(x)} \bigg(\frac{\big(\eta(x)-\eta(y)\big)U_{x_j^+,\lambda}(y)}{|x-y|^{N+2s}}+\frac{\big(\eta(x)-\eta(y)\big)U_{x_j^-,\lambda}(y)}{|x-y|^{N+2s}}\bigg) dy\\
        &=:J_1+J_2+J_3.
  \end{split}
\end{equation}

In order to use the contraction mapping theorem to prove that \eqref{2.22} is uniquely solvable, we need to estimate $\mathcal{F}(\varphi)$ and $l_k$ respectively.

\begin{lemma}\label{lem2.3}
If $N>4s+1$, then
\begin{equation*}
  \|\mathcal{F}(\varphi)\|_{**}\leq C\lambda^{\frac{4s}{N-2s}\tau} \|\varphi\|_{*}^{\min\{2_s^*-1,2\}}.
\end{equation*}
\end{lemma}

\begin{proof}
If $2_s^*\leq3$, we have
\begin{equation*}
  |\mathcal{F}(\varphi)|\leq C |\varphi|^{2_s^*-1},
\end{equation*}
which, together with H\"older inequality, yields that
\begin{equation}\label{b2.23}
  \begin{split}
    &|\mathcal{F}(\varphi)|\\
    &\leq C \|\varphi\|_{*}^{2_s^*-1} \bigg(\sum\limits_{j=1}^{k}\Big(\frac{\lambda^{\frac{N-2s}{2}}}{(1+\lambda|y-x_j^+|)^{\frac{N-2s}{2}+\tau}}+\frac{\lambda^{\frac{N-2s}{2}}}{(1+\lambda|y-x_j^-|)^{\frac{N-2s}{2}+\tau}}\Big)\bigg)^{2_s^*-1}\\
    &\leq C \|\varphi\|_{*}^{2_s^*-1} \sum\limits_{j=1}^{k}\Big(\frac{\lambda^{\frac{N+2s}{2}}}{(1+\lambda|y-x_j^+|)^{\frac{N+2s}{2}+\tau}}+\frac{\lambda^{\frac{N+2s}{2}}}{(1+\lambda|y-x_j^-|)^{\frac{N+2s}{2}+\tau}}\Big)\\
    &\quad \times\bigg(\sum\limits_{j=1}^{k}\Big(\frac{1}{(1+\lambda|y-x_j^+|)^{\tau}}+\frac{1}{(1+\lambda|y-x_j^-|)^{\tau}}\Big)\bigg)^\frac{4s}{N-2s}\\
 &\leq  C\lambda^{\frac{4s}{N-2s}\tau} \|\varphi\|_{*}^{2_s^*-1} \sum\limits_{j=1}^{k}\Big(\frac{\lambda^{\frac{N+2s}{2}}}{(1+\lambda|y-x_j^+|)^{\frac{N+2s}{2}+\tau}}+\frac{\lambda^{\frac{N+2s}{2}}}{(1+\lambda|y-x_j^-|)^{\frac{N+2s}{2}+\tau}}\Big). \end{split}
\end{equation}
Therefore,
\begin{equation*}
   \|\mathcal{F}(\varphi)\|_{**}\leq C \lambda^{\frac{4s}{N-2s}\tau}\|\varphi\|_{*}^{2_s^*-1}.
\end{equation*}

Similarly, if $2_s^*>3$, then we have
\begin{equation*}
  \begin{aligned}
   &|\mathcal{F}(\varphi)|\\
   &\leq C |Z_{\bar{r},\bar{h},\bar{y}'',\lambda}|^{2_s^*-3}|\varphi|^2+C |\varphi|^{2_s^*-1}\\
    &\leq C\|\varphi\|_*^2 \bigg(\sum\limits_{j=1}^{k}\Big(\frac{\lambda^{\frac{N-2s}{2}}}{(1+\lambda|y-x_j^+|)^{\frac{N-2s}{2}+\tau}}+\frac{\lambda^{\frac{N-2s}{2}}}{(1+\lambda|y-x_j^-|)^{\frac{N-2s}{2}+\tau}}\Big)\bigg)^{2_s^*-1}\\
    &\leq C\lambda^{\frac{4s}{N-2s}\tau} \|\varphi\|_*^2 \sum\limits_{j=1}^{k}\Big(\frac{\lambda^{\frac{N+2s}{2}}}{(1+\lambda|y-x_j^+|)^{\frac{N+2s}{2}+\tau}}+\frac{\lambda^{\frac{N+2s}{2}}}{(1+\lambda|y-x_j^-|)^{\frac{N+2s}{2}+\tau}}\Big). \end{aligned}
\end{equation*}
Hence, we obtain
\begin{equation*}
  \|\mathcal{F}(\varphi)\|_{**}\leq C\lambda^{\frac{4s}{N-2s}\tau} \|\varphi\|_{*}^{\min\{2_s^*-1,2\}}.
\end{equation*}
\vspace{-8pt}\end{proof}

Next, we estimate $l_k$.
\begin{lemma}\label{lem2.4}
If $5\leq N\leq 8$ and $s$ satisfies \eqref{s}, then there is a small constant $\epsilon>0$, such that
\begin{equation*}
  \|l_k\|_{**}\leq \frac{C}{\lambda^{\frac{2s+1}{2}+\epsilon}}.
\end{equation*}
\end{lemma}

\begin{proof}
Recall, for $j=1,2,\cdots,k$,
\begin{equation*}
 \begin{aligned}
  \Omega_j^+:=\Big\{y=(y_1,y_2,&y_3,y'')\in \R^3\times\R^{N-3}:y_3\geq 0~\text{and}~ \\
  &\Big\langle \frac{(y_1,y_2)}{|(y_1,y_2)|},\big(\cos\frac{2(j-1)\pi}{k},\sin\frac{2(j-1)\pi}{k}\big)\Big\rangle_{\R^2}\geq \cos\frac{\pi}{k}\Big\}.
  \end{aligned}
\end{equation*}
By symmetry, we can assume that $y \in \Omega_1^+$. Then it follows
\begin{equation}\label{b2.24}
  |y-x_1^+|\leq |y-x_1^-|, \quad |y-x_1^+|\leq |y-x_j^+|\leq |y-x_j^-|,\quad j=2,3,\cdots,k.
\end{equation}

Firstly, we estimate $J_1$.
\begin{align*}
  |J_1|&=\Big|Z_{\bar{r},\bar{h},\bar{y}'',\lambda}^{2_s^*-1}-\sum\limits_{j=1}^{k} \big(\eta U_{x_j^+,\lambda}^{2_s^*-1}+\eta U_{x_j^-,\lambda}^{2_s^*-1}\big)\Big|\\
  &= \Big|\Big(\sum\limits_{j=1}^k (\eta U_{x_j^+,\lambda}+\eta U_{x_j^-,\lambda})\Big)^{2_s^*-1}-\sum\limits_{j=1}^k (\eta U_{x_j^+,\lambda}^{2_s^*-1}+\eta U_{x_j^-,\lambda}^{2_s^*-1})\Big|\\
  &\leq C \Big(U_{x_1^+,\lambda}^{2_s^*-2}\big(\sum\limits_{j=2}^k U_{x_j^+,\lambda}+\sum\limits_{j=1}^k U_{x_j^-,\lambda}\big)+\big(\sum\limits_{j=2}^kU_{x_j^+,\lambda}+\sum\limits_{j=1}^k U_{x_j^-,\lambda}\big)^{2_s^*-1}\Big)\\
  &\leq C \frac{\lambda^{\frac{N+2s}{2}}}{(1+\lambda|y-x_1^+|)^{4s}} \Big(\sum\limits_{j=2}^{k}\frac{1}{(1+\lambda|y-x_j^+|)^{N-2s}}+\sum\limits_{j=1}^{k}\frac{1}{(1+\lambda|y-x_j^-|)^{N-2s}}\Big)\\
  &\quad +C \Big(\sum\limits_{j=2}^{k}\frac{\lambda^{\frac{N-2s}{2}}}{(1+\lambda|y-x_j^+|)^{N-2s}}\Big)^{2_s^*-1}\\
  &=:J_{11}+J_{12}.
\end{align*}
For the term $J_{11}$, if $N-2s\geq\frac{N+2s}{2}-\tau$, using Lemma \ref{lemA.4.6}, then we have for $y \in \Omega_1^+$,
\begin{align*}
  &\frac{\lambda^{\frac{N+2s}{2}}}{(1+\lambda|y-x_1^+|)^{4s}} \Big(\sum\limits_{j=2}^{k}\frac{1}{(1+\lambda|y-x_j^+|)^{N-2s}}+\sum\limits_{j=1}^{k}\frac{1}{(1+\lambda|y-x_j^-|)^{N-2s}}\Big)\\
  &\leq \frac{C\lambda^{\frac{N+2s}{2}}}{(1+\lambda|y-x_1^+|)^{\frac{N+2s}{2}+\tau}}
  \Big(\sum\limits_{j=2}^{k}\frac{1}{(\lambda|x_1^+-x_j^+|)^{\frac{N+2s}{2}-\tau}}+\sum\limits_{j=1}^{k}\frac{1}{(\lambda|x_1^+-x_j^-|)^{\frac{N+2s}{2}-\tau}}\Big)\\
  &\leq \frac{C}{\lambda^{\frac{2s}{N-2s}(\frac{N+2s}{2}-\tau)}}\frac{\lambda^{\frac{N+2s}{2}}}{(1+\lambda|y-x_1^+|)^{\frac{N+2s}{2}+\tau}},
\end{align*}
where we used Lemma \ref{lemA.4} and $\bar{h}$ satisfying \eqref{1.7}.

Similarly, if $N-2s<\frac{N+2s}{2}-\tau$, which implies that $4s>\frac{N+2s}{2}+\tau$, then we obtain that
\begin{align*}
   &\frac{\lambda^{\frac{N+2s}{2}}}{(1+\lambda|y-x_1^+|)^{4s}} \Big(\sum\limits_{j=2}^{k}\frac{1}{(1+\lambda|y-x_j^+|)^{N-2s}}+\sum\limits_{j=1}^{k}\frac{1}{(1+\lambda|y-x_j^-|)^{N-2s}}\Big)\\
  &\leq \frac{C\lambda^{\frac{N+2s}{2}}}{(1+\lambda|y-x_1^+|)^{\frac{N+2s}{2}+\tau}}\Big(\sum\limits_{j=2}^{k}\frac{1}{(\lambda|x_1^+-x_j^+|)^{N-2s}}+\sum\limits_{j=1}^{k}\frac{1}{(\lambda|x_1^+-x_j^-|)^{N-2s}}\Big)\\
  &\leq \frac{C}{\lambda^{2s}}\frac{\lambda^{\frac{N+2s}{2}}}{(1+\lambda|y-x_1^+|)^{\frac{N+2s}{2}+\tau}}.
\end{align*}

Hence, if $s$ satisfies \eqref{s}, then we have
\begin{equation}\label{a2.24}
\|J_{11}\|_{**}\leq \frac{C}{\lambda^{\frac{2s+1}{2}+\epsilon}}.
\end{equation}

As for $J_{12}$, using the H\"older inequality, we have
\begin{align*}
  &\Big(\sum\limits_{j=2}^{k}\frac{\lambda^{\frac{N-2s}{2}}}{(1+\lambda|y-x_j^+|)^{N-2s}}\Big)^{2_s^*-1}\\
  &\leq
\sum\limits_{j=2}^{k}\frac{\lambda^{\frac{N+2s}{2}}}{(1+\lambda|y-x_j^+|)^{\frac{N+2s}{2}+\tau}}\Big(\sum\limits_{j=2}^{k}\frac{1}{(1+\lambda|y-x_j^+|)^{\frac{N+2s}{4s}(\frac{N-2s}{2}-\frac{N-2s}{N+2s}\tau)}}\Big)^{\frac{4s}{N-2s}}\\
&\leq
\sum\limits_{j=2}^{k}\frac{C\lambda^{\frac{N+2s}{2}}}{(1+\lambda|y-x_j^+|)^{\frac{N+2s}{2}+\tau}}\Big(\sum\limits_{j=2}^{k}\frac{1}{(\lambda|x_1^+-x_j^+|)^{\frac{N+2s}{4s}(\frac{N-2s}{2}-\frac{N-2s}{N+2s}\tau)}}\Big)^{\frac{4s}{N-2s}}\\
&\leq C\sum\limits_{j=2}^{k}\frac{\lambda^{\frac{N+2s}{2}}}{(1+\lambda|y-x_j^+|)^{\frac{N+2s}{2}+\tau}}\Big(\frac{k}{\lambda}\Big)^{\frac{N+2s}{N-2s}(\frac{N-2s}{2}-\frac{N-2s}{N+2s}\tau)}\\
&\leq C\sum\limits_{j=2}^{k}\frac{\lambda^{\frac{N+2s}{2}}}{(1+\lambda|y-x_j^+|)^{\frac{N+2s}{2}+\tau}}\frac{1}{\lambda^{\frac{2s}{N-2s}(\frac{N+2s}{2}-\tau)}},
\end{align*}
since $\frac{N+2s}{4s}(\frac{N-2s}{2}-\frac{N-2s}{N+2s}\tau)>1$.\\
Thus,
\begin{equation*}
 \|J_{1}\|_{**}\leq \frac{C}{\lambda^{\frac{2s+1}{2}+\epsilon}}.
\end{equation*}

Secondly, we estimate $J_2$. Note that $\eta=0$ when $|(r,y'')-(r_0,y_0'')|\geq 2\sigma$ and $\frac{1}{\lambda}\leq \frac{C}{1+\lambda|y-x_j^\pm|}$ when $|(r,y'')-(r_0,y_0'')| < 2\sigma$. If $N-2s\geq\frac{N+2s}{2}-\tau$, then we have
\begin{align*}
  |J_2|&\leq \frac{C}{\lambda^{2s}} \sum\limits_{j=1}^{k}\Big(\frac{\eta\lambda^{\frac{N+2s}{2}}}{(1+\lambda|y-x_j^+|)^{N-2s}}+\frac{\eta\lambda^{\frac{N+2s}{2}}}{(1+\lambda|y-x_j^-|)^{N-2s}}\Big)\\
        &\leq \frac{C}{\lambda^{{2s}}} \sum\limits_{j=1}^{k}\Big(\frac{\lambda^{\frac{N+2s}{2}}}{(1+\lambda|y-x_j^+|)^{\frac{N+2s}{2}+\tau}}
        +\frac{\lambda^{\frac{N+2s}{2}}}{(1+\lambda|y-x_j^-|)^{\frac{N+2s}{2}+\tau}}\Big).\\
\end{align*}
If $N-2s<\frac{N+2s}{2}-\tau$, then we have
\begin{align*}
  |J_2|&\leq \frac{C}{\lambda^{2s}} \sum\limits_{j=1}^{k}\Big(\frac{\eta\lambda^{\frac{N+2s}{2}}}{(1+\lambda|y-x_j^+|)^{N-2s}}+\frac{\eta\lambda^{\frac{N+2s}{2}}}{(1+\lambda|y-x_j^-|)^{N-2s}}\Big)\\
        &\leq \frac{C}{\lambda^{\frac{N-2s}{2}-\tau}} \sum\limits_{j=1}^{k}\Big(\frac{\lambda^{\frac{N+2s}{2}}}{(1+\lambda|y-x_j^+|)^{\frac{N+2s}{2}+\tau}}
        +\frac{\lambda^{\frac{N+2s}{2}}}{(1+\lambda|y-x_j^-|)^{\frac{N+2s}{2}+\tau}}\Big).\\
\end{align*}
So,
\begin{equation*}
  \|J_{2}\|_{**}\leq \frac{C}{\lambda^{\frac{2s+1}{2}+\epsilon}}.
\end{equation*}

Finally, we estimate $J_3$. We have
\begin{align*}
    J_3 &=\sum\limits_{j=1}^{k}c(N,s)\bigg(\lim\limits_{\varepsilon \to 0^+} \int_{B_{\frac{\sigma}{4}}(y) \backslash B_\varepsilon(y)}\frac{\big(\eta(y)-\eta(x)\big)\big(U_{x_j^+,\lambda}(x)+U_{x_j^-,\lambda}(x)\big)}{|x-y|^{N+2s}} dx \\
    &\quad + \int_{\R^N \backslash B_{\frac{\sigma}{4}}(y)} \frac{\big(\eta(y)-\eta(x)\big)\big(U_{x_j^+,\lambda}(x)+U_{x_j^-,\lambda}(x)\big)}{|x-y|^{N+2s}} dx\bigg)\\
    &=:\sum\limits_{j=1}^{k}c(N,s)(J_{31}+J_{32}).
\end{align*}

We first estimate the term $J_{31}$. From the definition of the function $\eta$, we have $\eta(y)-\eta(x)=0$ when $x,y\in B_\sigma((r_0,y''_0))$ or $x,y\in\R^N\backslash B_{2\sigma}((r_0,y''_0))$. So, when $|y-(r_0,y''_0)|<\frac{3\sigma}{4}$ or $|y-(r_0,y''_0)|>\frac{9\sigma}{4}$, we have $\eta(y)-\eta(x)\equiv0$ for $x \in B_{\frac{\sigma}{4}}(y)$, hence $J_{31}\neq 0$ only if $B_{\frac{\sigma}{4}}(y)\subset B_{\frac{5}{2}\sigma}((r_0,y''_0))\backslash B_{\frac{1}{2}\sigma}((r_0,y''_0))$. Since $x_j^{\pm}\to(r_0,y''_0)$, we consider the case $B_{\frac{\sigma}{4}}(y)\subset B_{\frac{11}{4}\sigma}(x^{\pm}_j)\backslash B_{\frac{1}{4}\sigma}(x^{\pm}_j)$, hence
\begin{equation*}
  \frac{\sigma}{2} \leq|y-x_j^{\pm}|\leq \frac{5\sigma}{2} \quad \text{~and~} \quad \frac{\sigma}{4} \leq|x-x_j^{\pm}|\leq \frac{11\sigma}{4}, \quad \forall~ x\in B_{\frac{\sigma}{4}}(y).
\end{equation*}
Then we have $\frac{1}{10}|y-x_j^{\pm}|\leq|x-x_j^{\pm}|\leq\frac{11}{2}|y-x_j^{\pm}|$.

Furthermore,
\begin{align*}
  &\lim\limits_{\varepsilon \to 0^+} \int_{B_{\frac{\sigma}{4}}(y) \backslash B_\varepsilon(y)}\frac{\big(\eta(y)-\eta(x)\big)U_{x_j^{\pm},\lambda}(y)}{|x-y|^{N+2s}} dx\\
  &=\lim\limits_{\varepsilon \to 0^+} \int_{B_{\frac{\sigma}{4}}(y) \backslash B_\varepsilon(y)}\frac{\langle\nabla \eta(y),y-x\rangle U_{x_j^{\pm},\lambda}(x)}{|x-y|^{N+2s}} dx\\
  &\quad +O\Big(\lim\limits_{\varepsilon \to 0^+} \int_{B_{\frac{\sigma}{4}}(y) \backslash B_\varepsilon(y)}\frac{U_{x_j^{\pm},\lambda}(x)}{|x-y|^{N+2s-2}} dx\Big).
\end{align*}
Note that $B_{\frac{\sigma}{4}}(y) \backslash B_\varepsilon(y)$ is a symmetric set. By the mean value theorem, we obtain that
\begin{align*}
  &\Big|\lim\limits_{\varepsilon \to 0^+} \int_{B_{\frac{\sigma}{4}}(y) \backslash B_\varepsilon(y)}\frac{\langle\nabla \eta(y),y-x\rangle U_{x_j^{\pm},\lambda}(x)}{|x-y|^{N+2s}}\Big|\\
  &=\Big|\lim\limits_{\varepsilon \to 0^+} \int_{B_{\frac{\sigma}{4}}(0) \backslash B_\varepsilon(0)}\frac{\langle\nabla \eta(y),z\rangle}{|z|^{N+2s}} \frac{C_N\lambda^{\frac{N-2s}{2}}}{(1+\lambda^{2}|z+y-x^{\pm}_j|^{2})^{\frac{N-2s}{2}}}\Big|\\
  &=\Big|\lim\limits_{\varepsilon \to 0^+} \int_{B_{\frac{\sigma}{4}}(0) \backslash B_\varepsilon(0)}\frac{\langle\nabla \eta(y),z\rangle}{2|z|^{N+2s}}
  \bigg(\frac{C_N\lambda^{\frac{N-2s}{2}}}{(1+\lambda^{2}|z+y-x^{\pm}_j|^{2})^{\frac{N-2s}{2}}}\\
  &\quad -\frac{C_N\lambda^{\frac{N-2s}{2}}}{(1+\lambda^{2}|-z+y-x^{\pm}_j|^{2})^{\frac{N-2s}{2}}}\bigg)\Big|\\
  &\leq C\lambda^{\frac{N-2s}{2}+1}\int_{B_{\frac{\sigma}{4}}(0)}\frac{|\nabla \eta(y)|}{|z|^{N+2s-2}} \frac{1}{(1+\lambda|(2\vartheta-1)z+y-x^{\pm}_j|)^{N-2s+1}}\\
  &\leq \frac{C}{\lambda^{{\frac{2s+1}{2}}+\epsilon}}\frac{\lambda^{\frac{N+2s}{2}}}{(1+\lambda|y-x^{\pm}_j|)^{\frac{N+2s}{2}+\tau}},
\end{align*}
for $0<\vartheta<1$ and since $|(2\vartheta-1)z+y-x^+_j|\geq|y-x^{\pm}_j|-|(2\vartheta-1)z|\geq\frac{1}{10}|y-x^{\pm}_j|$ for $z\in B_{\frac{\sigma}{4}(0)}$.

On the other hand, we can obtain
\begin{align*}
&\Big|\lim\limits_{\varepsilon \to 0^+} \int_{B_{\frac{\sigma}{4}}(y) \backslash B_\varepsilon(y)}\frac{U_{x_j^{\pm},\lambda}(x)}{|x-y|^{N+2s-2}}\Big|\\
&\leq C\int_{B_{\frac{\sigma}{4}}(y)}\frac{1}{|x-y|^{N+2s-2}}\frac{\lambda^{\frac{N-2s}{2}}}{(1+\lambda|x-x^{\pm}_j|)^{N-2s}}\\
&\leq \frac{C}{\lambda^{{2s}}}\int_{B_{\frac{\sigma}{4}}(y)}
\frac{1}{|x-y|^{N+2s-2}}\frac{\lambda^{\frac{N+2s}{2}}}{(1+\lambda|x-x^{\pm}_j|)^{N-2s}}\\
&\leq \frac{C}{\lambda^{\frac{2s+1}{2}+\epsilon}}\frac{\lambda^{\frac{N+2s}{2}}}{(1+\lambda|y-x^{\pm}_j|)^{\frac{N+2s}{2}+\tau}}.
\end{align*}
Hence,
\begin{align*}
  |J_{31}|&=\Bigg|\lim\limits_{\varepsilon \to 0^+} \int_{B_{\frac{\sigma}{4}}(y) \backslash B_\varepsilon(y)}\frac{\big(\eta(y)-\eta(x)\big)\big(U_{x_j^+,\lambda}(x)+U_{x_j^-,\lambda}(x)\big)}{|x-y|^{N+2s}}\Bigg|\\
  &\leq \frac{C}{\lambda^{\frac{2s+1}{2}+\epsilon}}\bigg(\frac{\lambda^{\frac{N+2s}{2}}}{(1+\lambda|y-x^+_j|)^{\frac{N+2s}{2}+\tau}}+\frac{\lambda^{\frac{N+2s}{2}}}{(1+\lambda|y-x^-_j|)^{\frac{N+2s}{2}+\tau}}\bigg).
\end{align*}

As for the term $J_{32}$, we divide it into the following three cases:

\vspace{4pt}\noindent
\textbf{Case 1.} If $|y-x^+_j|\leq\sigma$, together with the definition of the function $\eta$, then we have
\begin{align*}
&\Big|\int_{\R^N \backslash B_{\frac{\sigma}{4}}(y)}\frac{\big(\eta(y)-\eta(x)\big)U_{x_j^+,\lambda}(x)}{|x-y|^{N+2s}}\Big|\\
&\leq C\int_{\R^N \backslash \big(B_{\frac{\sigma}{4}}(y)\cup B_\sigma((x_j^+))\big)} \frac{1}{|x-y|^{N+2s}}\frac{\lambda^{\frac{N-2s}{2}}}{(1+\lambda|x-x^+_j|)^{N-2s}}\\
&\leq \frac{C}{\lambda^{2s}}\frac{\lambda^{\frac{N+2s}{2}}}{(1+\lambda|y-x^+_j|)^{N-2s}}
\int_{\R^N \backslash B_{\frac{\sigma}{4}}(y)}\frac{1}{|x-y|^{N+2s}}\\
&\leq \frac{C}{\lambda^{\frac{2s+1}{2}+\epsilon}}\frac{\lambda^{\frac{N+2s}{2}}}{(1+\lambda|y-x^+_j|)^{\frac{N+2s}{2}+\tau}},
\end{align*}
where we used $\frac{1}{\lambda}\leq\frac{C}{1+\lambda|y-x_j^+|}$ for $|y -x^+_j|\leq\sigma$.

\vspace{4pt}\noindent
\textbf{Case 2.} If $\sigma<|y-x^+_j|\leq3\sigma$, using Lemma \ref{lemC.1}, then there holds
\begin{align*}
&\Big|\int_{\R^N \backslash B_{\frac{\sigma}{4}}(y)} \frac{\big(\eta(y)-\eta(x)\big)U_{x_j^+,\lambda}(x)}{|x-y|^{N+2s}}\Big|\\
&\leq C\int_{\R^N \backslash B_{\frac{\sigma}{4}}(y)} \frac{1}{|x-y|^{N+2s}} \frac{\lambda^{\frac{N-2s}{2}}}{(1+\lambda|x-x^+_j|)^{N-2s}}\\
&\leq C \lambda^{\frac{N+2s}{2}}\int_{\R^N \backslash B_{\frac{\sigma\lambda}{4}}(\lambda y)} \frac{1}{|z-\lambda y|^{N+2s}} \frac{1}{(1+|z-\lambda x^+_j|)^{N-2s}}\\
&\leq C\lambda^{\frac{N+2s}{2}}\Big(\frac{1}{(\lambda|y-x^+_j|)^N}+\frac{1}{\lambda^{2s}}\frac{1}{(\lambda|y-x^+_j|)^{N-2s}}\Big)\\
&\leq \frac{C}{\lambda^{\frac{2s+1}{2}+\epsilon}}\frac{\lambda^{\frac{N+2s}{2}}}{(1+\lambda|y-x^+_j|)^{\frac{N+2s}{2}+\tau}},
\end{align*}
where we used $\frac{C_1}{1+\lambda|y-x_j^+|}\leq\frac{1}{\lambda}\leq\frac{C_2}{1+\lambda|y-x_j^+|}$ for $\sigma\leq|y -x^+_j|\leq3\sigma$.

\vspace{4pt}\noindent
\textbf{Case 3.} Suppose that $|y-x^+_j|>3\sigma$. Note that $|x-y|\geq|y-x^+_j|-|x-x^+_j|\geq\frac{1}{3}|y-x^+_j|$ when $|y-x^+_j|>3\sigma$ and $|x-x^+_j|\leq2\sigma$. Then we have
\begin{align*}
&\Big|\int_{\R^N \backslash B_{\frac{\sigma}{4}}(y)} \frac{\big(\eta(y)-\eta(x)\big)U_{x_j^+,\lambda}(x)}{|x-y|^{N+2s}}\Big|\\
&\leq \int_{B_{2\sigma(x_j^+)}} \frac{1}{|x-y|^{N+2s}} \frac{\lambda^{\frac{N-2s}{2}}}{(1+\lambda|x-x^+_j|)^{N-2s}}\\
&\leq \frac{C}{\lambda^{\frac{N-2s}{2}}} \int_{B_{2\sigma(x_j^+)}} \frac{1}{|x-y|^{N+2s}} \frac{1}{(|x- x^+_j|)^{N-2s}}\\
&\leq \frac{C}{\lambda^{\frac{2s+1}{2}+\epsilon}}\frac{\lambda^{\frac{N+2s}{2}}}{(1+\lambda|y-x^+_j|)^{\frac{N+2s}{2}+\tau}},
\end{align*}
where we used $\frac{1}{|x-y|}\leq\frac{C\lambda}{1+\lambda|y-x^+_j|}$.

Hence, we have
\begin{align*}
  \Big|\int_{\R^N \backslash B_{\frac{\sigma}{4}}(y)} \frac{\big(\eta(y)-\eta(x)\big)U_{x_j^+,\lambda}(x)}{|x-y|^{N+2s}} dx\Big|
  \leq \frac{C}{\lambda^{\frac{2s+1}{2}+\epsilon}}\frac{\lambda^{\frac{N+2s}{2}}}{(1+\lambda|y-x^+_j|)^{\frac{N+2s}{2}+\tau}}.
\end{align*}

Similarly, we can deduce that
\begin{align*}
  \Big|\int_{\R^N \backslash B_{\frac{\sigma}{4}}(y)} \frac{\big(\eta(y)-\eta(x)\big)U_{x_j^-,\lambda}(x)}{|x-y|^{N+2s}} dx\Big|
  \leq \frac{C}{\lambda^{\frac{2s+1}{2}+\epsilon}}\frac{\lambda^{\frac{N+2s}{2}}}{(1+\lambda|y-x^-_j|)^{\frac{N+2s}{2}+\tau}}.
\end{align*}
So, we obtain
\begin{equation}\label{a2.28}
  \begin{aligned}
  \|J_{3}\|_{**}\leq \frac{C}{\lambda^{\frac{2s+1}{2}+\epsilon}}.
  \end{aligned}
\end{equation}

As a result, we have proved that
\begin{equation*}
  \|l_k\|_{**}\leq \frac{C}{\lambda^{\frac{2s+1}{2}+\epsilon}}.
\end{equation*}
\end{proof}

Now, we are in a position to prove Proposition \ref{prop2.5}.
\begin{proof}[\textbf{Proof of Proposition \ref{prop2.5}.}]
We first denote
\begin{equation}\label{2.29}
  \mathbb{E}=\Big\{u:u \in C(\R^N)\cap\mathbb{H},~\|u\|_*\leq \frac{1}{\lambda^{\frac{2s+1}{2}}}\Big\}.
\end{equation}
By Lemma \ref{lem2.2}, the existence and properties of the solution $\varphi$ to problem \eqref{2.22} is simplified to find a fixed point for
\begin{equation}\label{2.30}
  \varphi=\mathcal{A}(\varphi)=:\mathcal{L}_k(\mathcal{F}(\varphi))+\mathcal{L}_k(l_k),
\end{equation}
where $\mathcal{L}_k$ is the linear bounded operator defined in Lemma \ref{lem2.2}.

Next, we will prove that $\mathcal{A}$ is a contraction map from $\mathbb{E}$ to $\mathbb{E}$. In fact, if $\varphi \in L^\infty(\R^N)$, then by Proposition 2.9 in \cite{silvestre2007regularity}, we can obtain $\varphi \in C(\R^N)$. For any $\varphi \in \mathbb{E}$, by Lemma \ref{lem2.2}, Lemma \ref{lem2.3} and Lemma \ref{lem2.4}, we have
\begin{align*}
    \|\mathcal{A}(\varphi)\|_* &\leq \|\mathcal{L}_k(\mathcal{F}(\varphi))\|_*+\|\mathcal{L}_k(l_k)\|_*\\
    &\leq C\|\mathcal{F}(\varphi)\|_{**}+C\|l_k\|_{**}\\
    &\leq C\lambda^{\frac{4s}{N-2s}\tau} \|\varphi\|_{*}^{\min\{2_s^*-1,2\}}+\frac{C}{\lambda^{{\frac{2s+1}{2}}+\epsilon}}
    \leq \frac{1}{\lambda^{{\frac{2s+1}{2}}}}.
\end{align*}
This shows that $\mathcal{A}$ maps $\mathbb{E}$ to $\mathbb{E}$.

On the other hand, for all $\varphi_1,\varphi_2 \in \mathbb{E}$, we have
\begin{equation}\label{2.31}
  \|\mathcal{A}(\varphi_1)-\mathcal{A}(\varphi_2)\|_*=\|\mathcal{L}_k(\mathcal{F}(\varphi_1))-\mathcal{L}_k(\mathcal{F}(\varphi_2))\|_* \leq C\|\mathcal{F}(\varphi_1)-\mathcal{F}(\varphi_2)\|_{**}.
\end{equation}
If $2_s^*\leq3$, using H\"older inequality like \eqref{b2.23}, then we have
\begin{align*}
    &|\mathcal{F}(\varphi_1)-\mathcal{F}(\varphi_2)| \\
    &\leq C(|\varphi_1|^{2_s^*-2}+|\varphi_2|^{2_s^*-2})|\varphi_1-\varphi_2|\\
    &\leq C(\|\varphi_1\|_*^{2_s^*-2}+\|\varphi_2\|_*^{2_s^*-2})\|\varphi_1-\varphi_2\|_*\\
    &\quad \times\Big(\sum\limits_{j=1}^{k}\big(\frac{\lambda^{\frac{N-2s}{2}}}{(1+\lambda|y-x_j^+|)^{\frac{N-2s}{2}+\tau}}+\frac{\lambda^{\frac{N-2s}{2}}}{(1+\lambda|y-x_j^-|)^{\frac{N-2s}{2}+\tau}}\big)\Big)^{2_s^*-1}\\
  &\leq C(\|\varphi_1\|_*^{2_s^*-2}+\|\varphi_2\|_*^{2_s^*-2})\|\varphi_1-\varphi_2\|_*\\
  &\quad \times\Big(\sum\limits_{j=1}^{k}\big(\frac{\lambda^{\frac{N+2s}{2}}}{(1+\lambda|y-x_j^+|)^{\frac{N+2s}{2}+\tau}}+\frac{\lambda^{\frac{N+2s}{2}}}{(1+\lambda|y-x_j^-|)^{\frac{N+2s}{2}+\tau}}\big)\Big).
\end{align*}
Thus
\begin{equation}\label{2.32}
  \|\mathcal{A}(\varphi_1)-\mathcal{A}(\varphi_2)\|_* \leq C(\|\varphi_1\|_*^{2_s^*-2}+\|\varphi_2\|_*^{2_s^*-2})\|\varphi_1-\varphi_2\|_*\leq \frac{1}{2}\|\varphi_1-\varphi_2\|_*.
\end{equation}
Therefore, $\mathcal{A}$ is a contraction map from $\mathbb{E}$ to $\mathbb{E}$. The case $2_s^*>3$ can be discussed in a similar way.

By the contraction mapping theorem, there exists a unique $\varphi=\varphi_{\bar{r},\bar{h},\bar{y}'',\lambda} \in \mathbb{E}$ such that \eqref{2.30} holds. Furthermore, according to Lemma \ref{lem2.2}, Lemma \ref{lem2.3} and Lemma \ref{lem2.4}, we deduce
\begin{equation*}
  \|\varphi\|_* \leq \|\mathcal{L}_k(\mathcal{F}(\varphi))\|_*+\|\mathcal{L}_k(l_k)\|_* \leq C\|\mathcal{F}(\varphi)\|_{**}+C\|l_k\|_{**}\leq \frac{C}{\lambda^{{\frac{2s+1}{2}}+\epsilon}},
\end{equation*}
and
\begin{equation*}
  |c_l|\leq\frac{C}{\lambda^{n_l}}\|\mathcal{F}(\varphi)+l_k\|_{**} \leq \frac{C}{\lambda^{{\frac{2s+1}{2}}+n_l+\epsilon}},
\end{equation*}
for $l=1,2,\cdots,N$.
\end{proof}

\section{Proof of Theorem \ref{theo1.2}}\label{sec3}

Let $\varphi=\varphi_{\bar{r},\bar{h},\bar{y}'',\lambda}$ be the function obtained in Proposition \ref{prop2.5} and $u_k=Z_{\bar{r},\bar{h},\bar{y}'',\lambda}+\varphi$. Following the idea in \cite{peng2018construction}, in order to use local Pohozaev identities, we quote the extension of $u_k$, that is
\begin{equation*}
  \tilde{u}_k=\tilde{Z}_{\bar{r},\bar{h},\bar{y}'',\lambda}+\tilde{\varphi},
\end{equation*}
where $\tilde{Z}_{\bar{r},\bar{h},\bar{y}'',\lambda}$ and $\tilde{\varphi}$ are the extensions of $Z_{\bar{r},\bar{h},\bar{y}'',\lambda}$ and $\varphi$, respectively. Then we have
\begin{equation}\label{3.1}
  \left\{
  \begin{array}{ll}
    \text{div}(t^{1-2s}\nabla\tilde{u}_k)=0,&\text{in}\quad \R_+^{N+1},\\
    \begin{aligned}
    -\lim\limits_{t \to 0^+}t^{1-2s}\partial_t\tilde{u}_k
    &=\omega_s\Big(-V(r,y'')u_k+(u_k)_+^{2_s^*-1}\\
    &\quad +\sum\limits_{l=1}^N c_l\sum\limits_{j=1}^k\big(Z_{x_j^+,\lambda}^{2_s^*-2}Z_{j,l}^+
    +Z_{x_j^-,\lambda}^{2_s^*-2}Z_{j,l}^-\big)\Big),
    \end{aligned}
&\text{on} \quad \R^N.
  \end{array}
\right.
\end{equation}
Without loss of generality, we may assume $\omega_s=1$.

Multiplying equation \eqref{3.1} by $\frac{\partial\tilde{u}_k}{\partial y_i}~(i=4,\cdots,N)$ and $\langle\nabla\tilde{u}_k,Y\rangle$, then integrating by parts, we have the following two Pohozaev identities
\begin{equation}\label{3.2}
  \begin{split}
    &\int_{\partial''{B_\rho^+}} \Big(-t^{1-2s}\frac{\partial\tilde{u}_k}{\partial \nu}\frac{\partial\tilde{u}_k}{\partial y_i}+\frac{1}{2} t^{1-2s}|\nabla\tilde{u}_k|^2\nu_i\Big)\\
    &=\int_{B_\rho}\Big(-V(r,y'')u_k+(u_k)_+^{2_s^*-1}+\sum\limits_{l=1}^N
c_l\sum\limits_{j=1}^k\big(Z_{x_j^+,\lambda}^{2_s^*-2}Z_{j,l}^+ +Z_{x_j^-,\lambda}^{2_s^*-2}Z_{j,l}^-\big)\Big)\frac{\partial u_k}{\partial y_i},
  \end{split}
\end{equation}
for $i=4,\cdots,N$ and
\begin{equation}\label{3.3}
  \begin{split}
    &\int_{\partial''{B_\rho^+}} \Big(-t^{1-2s}\langle\nabla\tilde{u}_k,Y\rangle\frac{\partial\tilde{u}_k}{\partial \nu}+\frac{1}{2}t^{1-2s}|\nabla\tilde{u}_k|^2\langle Y,\nu\rangle\Big) +\frac{2s-N}{2}\int_{\partial{B_\rho^+}}t^{1-2s}\frac{\partial\tilde{u}_k}{\partial \nu}\tilde{u}_k\\
    &=\int_{{B_\rho}}\Big(-V(r,y'')u_k+(u_k)_+^{2_s^*-1}+\sum\limits_{l=1}^N
c_l\sum\limits_{j=1}^k\big(Z_{x_j^+,\lambda}^{2_s^*-2}Z_{j,l}^+ +Z_{x_j^-,\lambda}^{2_s^*-2}Z_{j,l}^-\big)\Big)\langle \nabla u_k,y\rangle.
  \end{split}
\end{equation}
The proof is standard, which will be given in Appendix C.

In this section, we will choose suitable $(\bar{r},\bar{h},\bar{y}'',\lambda)$ such that $u_k=Z_{\bar{r},\bar{h},\bar{y}'',\lambda}+\varphi_{\bar{r},\bar{h},\bar{y}'',\lambda}$ is a solution of \eqref{1.1}. For this purpose, we need the following result.

\begin{proposition}\label{prop3.1}
Suppose that $(\bar{r},\bar{h},\bar{y}'',\lambda)$ satisfies
\begin{equation}\label{3.4}
  \begin{split}
\int_{\partial''{B_\rho^+}} \Big(-t^{1-2s}\frac{\partial\tilde{u}_k}{\partial \nu}\frac{\partial\tilde{u}_k}{\partial y_i}+\frac{1}{2} t^{1-2s}|\nabla\tilde{u}_k|^2\nu_i\Big)
    =\int_{B_\rho}\Big(-V(r,y'')u_k+(u_k)_+^{2_s^*-1}\Big)\frac{\partial u_k}{\partial y_i}
  \end{split}
\end{equation}
for $i=4,\cdots,N$,
\begin{equation}\label{3.5}
  \begin{split}
    &\int_{\partial''{B_\rho^+}} \Big(-t^{1-2s}\langle\nabla\tilde{u}_k,Y\rangle\frac{\partial\tilde{u}_k}{\partial \nu}+\frac{1}{2}t^{1-2s}|\nabla\tilde{u}_k|^2\langle Y,\nu\rangle\Big) +\frac{2s-N}{2}\int_{\partial {B_\rho^+}}t^{1-2s}\frac{\partial\tilde{u}_k}{\partial \nu}\tilde{u}_k\\
    &=\int_{B_\rho}\Big(-V(r,y'')u_k+(u_k)_+^{2_s^*-1}\Big)\langle \nabla u_k,y\rangle,
  \end{split}
\end{equation}
\begin{equation}\label{3.6a}
\int_{\R^N}\Big((-\Delta)^s u_k+V(r,y'')u_k-(u_k)_+^{2_s^*-1}\Big)\frac{\partial Z_{\bar{r},\bar{h},\bar{y}'',\lambda}}{\partial \bar{h}}=0,
\end{equation}
and
\begin{equation}\label{3.6}
\int_{\R^N}\Big((-\Delta)^s u_k+V(r,y'')u_k-(u_k)_+^{2_s^*-1}\Big)\frac{\partial Z_{\bar{r},\bar{h},\bar{y}'',\lambda}}{\partial \lambda}=0,
\end{equation}
where $u_k=Z_{\bar{r},\bar{h},\bar{y}'',\lambda}+\varphi_{\bar{r},\bar{h},\bar{y}'',\lambda}$ and $B_\rho=\{y:|y-(r_0,y''_0)|\leq \rho\}$ with $\rho \in (2\sigma,5\sigma)$. Then $c_l=0$, $l=1,2,\cdots,N$.
\end{proposition}

\begin{proof}
If \eqref{3.4} and \eqref{3.5} hold, then it follows from \eqref{3.2} and \eqref{3.3} that
\begin{equation}\label{a3.7}
  \sum\limits_{l=1}^N c_l\sum\limits_{j=1}^k\int_{B_\rho}\big(Z_{x_j^+,\lambda}^{2_s^*-2}Z_{j,l}^+ +Z_{x_j^-,\lambda}^{2_s^*-2}Z_{j,l}^-\big)\frac{\partial u_k}{\partial y_i}=0,\quad i=4,\cdots,N,
\end{equation}
and
\begin{equation}\label{a3.8}
  \sum\limits_{l=1}^N c_l\sum\limits_{j=1}^k\int_{B_\rho}\big(Z_{x_j^+,\lambda}^{2_s^*-2}Z_{j,l}^+ +Z_{x_j^-,\lambda}^{2_s^*-2}Z_{j,l}^-\big)\langle\nabla u_k,y\rangle=0.
\end{equation}
Note that $Z_{\bar{r},\bar{h},\bar{y}'',\lambda}=0$ in $\R^N\backslash B_\rho$, which together with  (\ref{3.6a})-(\ref{a3.8}), yields that
\begin{equation}\label{3.7}
  \sum\limits_{l=1}^N c_l\sum\limits_{j=1}^k\int_{\R^N}\big(Z_{x_j^+,\lambda}^{2_s^*-2}Z_{j,l}^+ +Z_{x_j^-,\lambda}^{2_s^*-2}Z_{j,l}^-\big)v=0
\end{equation}
holds for $v=\frac{\partial u_k}{\partial y_i}$, $v=\langle\nabla u_k,y\rangle$, $v=\frac{\partial Z_{\bar{r},\bar{h},\bar{y}'',\lambda}}{\partial \bar{h}}$ and $v=\frac{\partial Z_{\bar{r},\bar{h},\bar{y}'',\lambda}}{\partial \lambda}$.

Using Lemma \ref{lemA.1}, we can compute, for $i=4,\cdots,N$,
\begin{align*}
  &\int_{\R^N}Z_{x_1^+,\lambda}^{2_s^*-2}Z^+_{1,i}\frac{\partial Z_{\bar{r},\bar{h},\bar{y}'',\lambda}}{\partial y_i}\\
  &\leq\int_{\R^N} \frac{C\lambda^{N+2}}{(1+\lambda|y-x_1^+|)^{N+2s}} \sum\limits_{j=1}^{k}\Big(\frac{1}{(1+\lambda|y-x_j^+|)^{N-2s}}+\frac{1}{(1+\lambda|y-x_j^-|)^{N-2s}}\Big)\\
  &\leq \int_{\R^N}\frac{C\lambda^{N+2}}{(1+\lambda|y-x_1^+|)^{2N}}\\
  &\quad +\sum\limits_{j=2}^{k}\frac{C\lambda^{N+2}}{(\lambda|x_1^+-x_j^+|)^{N-2s}}\int_{\R^N}\Big(\frac{1}{(1+\lambda|y-x_1^+|)^{N+2s}} + \frac{1}{(1+\lambda|y-x_j^+|)^{N+2s}}\Big)\\
  &\quad +\sum\limits_{j=1}^{k}\frac{C\lambda^{N+2}}{(\lambda|x_1^+-x_j^-|)^{N-2s}}\int_{\R^N}\Big(\frac{1}{(1+\lambda|y-x_1^+|)^{N+2s}} + \frac{1}{(1+\lambda|y-x_j^-|)^{N+2s}}\Big)\\
  &=C\lambda^2\Big(\int_{\R^N} \frac{1}{(1+|y|)^{2N}}+o(1)\Big),
\end{align*}
where we used Lemma \ref{lemA.4}.

Then, we can check that
\begin{equation}\label{3.8}
  \sum\limits_{j=1}^k\int_{\R^N}\big(Z_{x_j^+,\lambda}^{2_s^*-2}Z_{j,i}^+ + Z_{x_j^-,\lambda}^{2_s^*-2}Z_{j,i}^-\big)\frac{\partial Z_{\bar{r},\bar{h},\bar{y}'',\lambda}}{\partial y_i}= 2k\lambda^2\big(a_1+o(1)\big),
\end{equation}
for $i=4,\cdots,N$. Similarly we can compute that
\begin{equation}\label{3.9}
  \sum\limits_{j=1}^k\int_{\R^N}\big(Z_{x_j^+,\lambda}^{2_s^*-2}Z_{j,2}^+ + Z_{x_j^-,\lambda}^{2_s^*-2}Z_{j,2}^-\big)\langle y',\nabla_{y'} Z_{\bar{r},\bar{h},\bar{y}'',\lambda}\rangle= 2k\lambda^2\big(a_2+o(1)\big),
\end{equation}
\begin{equation}\label{3.10a}
  \sum\limits_{j=1}^k\int_{\R^N}\big(Z_{x_j^+,\lambda}^{2_s^*-2}Z_{j,3}^+ + Z_{x_j^-,\lambda}^{2_s^*-2}Z_{j,3}^-\big)\frac{\partial Z_{\bar{r},\bar{h},\bar{y}'',\lambda}}{\partial \bar{h}}= 2k\lambda^{2}\big(a_3+o(1)\big),
\end{equation}
and
\begin{equation}\label{3.10}
  \sum\limits_{j=1}^k\int_{\R^N}\big(Z_{x_j^+,\lambda}^{2_s^*-2}Z_{j,1}^+ + Z_{x_j^-,\lambda}^{2_s^*-2}Z_{j,1}^-\big)\frac{\partial Z_{\bar{r},\bar{h},\bar{y}'',\lambda}}{\partial \lambda}= \frac{2k}{\lambda^{2}}\big(a_4+o(1)\big),
\end{equation}
for some constants $a_1, a_2, a_3, a_4>0$.

Since $\varphi$ is a solution to \eqref{2.22}, by fractional elliptical equation estimates (see, for example Proposition 2.9 in \cite{caffarelli2007extension} and Theorem 12.2.1 in \cite{chen2020fractional}), we can obtain $\varphi \in C^1(B_\rho)$. Using integrating by parts and \eqref{2.28}, we have
\begin{equation}\label{3.11}
 \sum\limits_{l=1}^N c_l \sum\limits_{j=1}^k\int_{\R^N}\big(Z_{x_j^+,\lambda}^{2_s^*-2}Z_{j,l}^+ + Z_{x_j^-,\lambda}^{2_s^*-2}Z_{j,l}^-\big)v=o(k\lambda^2)\sum\limits_{l=2}^{N}|c_l|+o(k|c_1|)
\end{equation}
holds for $v=\langle y,\nabla \varphi_{\bar{r},\bar{h},\bar{y}'',\lambda}\rangle$, $v=\frac{\partial \varphi_{\bar{r},\bar{h},\bar{y}'',\lambda}}{\partial \bar{h}}$ and $v=\frac{\partial \varphi_{\bar{r},\bar{h},\bar{y}'',\lambda}}{\partial y_i}$.

It follows from (\ref{3.7}) that
\begin{equation}\label{3.12}
 \sum\limits_{l=1}^N c_l \sum\limits_{j=1}^k\int_{\R^N}\big(Z_{x_j^+,\lambda}^{2_s^*-2}Z_{j,l}^+ + Z_{x_j^-,\lambda}^{2_s^*-2}Z_{j,l}^-\big)v=o(k\lambda^2)\sum\limits_{l=2}^{N}|c_l|+o(k|c_1|)
\end{equation}
holds for $v=\langle y, \nabla Z_{\bar{r},\bar{h},\bar{y}'',\lambda}\rangle$, $v=\frac{\partial Z_{\bar{r},\bar{h},\bar{y}'',\lambda}}{\partial \bar{h}}$ and $v=\frac{\partial Z_{\bar{r},\bar{h},\bar{y}'',\lambda}}{\partial y_i}$.

Since
\begin{equation*}
  \langle y, \nabla Z_{\bar{r},\bar{h},\bar{y}'',\lambda}\rangle = \langle y', \nabla_{y'} Z_{\bar{r},\bar{h},\bar{y}'',\lambda}\rangle + \langle y'', \nabla_{y''} Z_{\bar{r},\bar{h},\bar{y}'',\lambda}\rangle,
\end{equation*}
we obtain
\begin{equation}\label{3.13}
  \begin{split}
 &\sum\limits_{l=1}^N c_l \sum\limits_{j=1}^k\int_{\R^N}\big(Z_{x_j^+,\lambda}^{2_s^*-2}Z_{j,l}^+ + Z_{x_j^-,\lambda}^{2_s^*-2}Z_{j,l}^-\big)\langle y, \nabla Z_{\bar{r},\bar{h},\bar{y}'',\lambda}\rangle\\
 &= c_2\sum\limits_{j=1}^k\int_{\R^N}\big(Z_{x_j^+,\lambda}^{2_s^*-2}Z_{j,2}^+ + Z_{x_j^-,\lambda}^{2_s^*-2}Z_{j,2}^-\big)\langle y', \nabla_{y'} Z_{\bar{r},\bar{h},\bar{y}'',\lambda}\rangle + o(k\lambda^2)\sum\limits_{l=3}^{N}|c_l|+o(k|c_1|).\\
  \end{split}
\end{equation}
Similarly,
\begin{equation}\label{3.14a}
  \begin{split}
  &\sum\limits_{l=1}^N c_l \sum\limits_{j=1}^k\int_{\R^N}\big(Z_{x_j^+,\lambda}^{2_s^*-2}Z_{j,l}^+ + Z_{x_j^-,\lambda}^{2_s^*-2}Z_{j,l}^-\big)\frac{\partial Z_{\bar{r},\bar{h},\bar{y}'',\lambda}}{\partial \bar{h}}\\
 &= c_3\int_{\R^N}\big(Z_{x_j^+,\lambda}^{2_s^*-2}Z_{j,3}^+ + Z_{x_j^-,\lambda}^{2_s^*-2}Z_{j,3}^-\big)\frac{\partial Z_{\bar{r},\bar{h},\bar{y}'',\lambda}}{\partial \bar{h}} + o(k\lambda^2)\sum\limits_{l\neq1,3}|c_l|+o(k|c_1|), \\
  \end{split}
\end{equation}
and for $i=4,\cdots,N$,
\begin{equation}\label{3.14}
  \begin{split}
  &\sum\limits_{l=1}^N c_l \sum\limits_{j=1}^k\int_{\R^N}\big(Z_{x_j^+,\lambda}^{2_s^*-2}Z_{j,l}^+ + Z_{x_j^-,\lambda}^{2_s^*-2}Z_{j,l}^-\big)\frac{\partial Z_{\bar{r},\bar{h},\bar{y}'',\lambda}}{\partial y_i}\\
 &= c_i\sum\limits_{j=1}^k\int_{\R^N}\big(Z_{x_j^+,\lambda}^{2_s^*-2}Z_{j,i}^+ + Z_{x_j^-,\lambda}^{2_s^*-2}Z_{j,i}^-\big)\frac{\partial Z_{\bar{r},\bar{h},\bar{y}'',\lambda}}{\partial y_i} + o(k\lambda^2)\sum\limits_{l\neq1,i}|c_l|+o(k|c_1|). \\
  \end{split}
\end{equation}

According to \eqref{3.12}-\eqref{3.14}, we have
\begin{equation}\label{3.15}
  c_2\sum\limits_{j=1}^k\int_{\R^N}\big(Z_{x_j^+,\lambda}^{2_s^*-2}Z_{j,2}^+ + Z_{x_j^-,\lambda}^{2_s^*-2}Z_{j,2}^-\big)\langle y', \nabla_{y'} Z_{\bar{r},\bar{h},\bar{y}'',\lambda}\rangle = o(k\lambda^2)\sum\limits_{l=3}^{N}|c_l|+o(k|c_1|),
\end{equation}
\begin{equation}\label{3.15b}
  c_3\sum\limits_{j=1}^k\int_{\R^N}\big(Z_{x_j^+,\lambda}^{2_s^*-2}Z_{j,3}^+ + Z_{x_j^-,\lambda}^{2_s^*-2}Z_{j,3}^-\big)\frac{\partial Z_{\bar{r},\bar{h},\bar{y}'',\lambda}}{\partial \bar{h}} = o(k\lambda^2)\sum\limits_{l\neq1,3}|c_l|+o(k|c_1|),
\end{equation}
and for $i=4,\cdots,N$,
\begin{equation}\label{3.16}
  c_i\sum\limits_{j=1}^k\int_{\R^N}\big(Z_{x_j^+,\lambda}^{2_s^*-2}Z_{j,i}^+ + Z_{x_j^-,\lambda}^{2_s^*-2}Z_{j,i}^-\big)\frac{\partial Z_{\bar{r},\bar{h},\bar{y}'',\lambda}}{\partial y_i} = o(k\lambda^2)\sum\limits_{l\neq1,i}|c_l|+o(k|c_1|),
\end{equation}
which, together with \eqref{3.8}-\eqref{3.10a}, yields that
\begin{equation}\label{3.17}
  c_l=o\Big(\frac{1}{\lambda^{2}}|c_1|\Big), \quad l=2,\cdots,N.
\end{equation}
On the other hand, there holds
\begin{equation}\label{3.18}
  \begin{split}
    0 &=\sum\limits_{l=1}^N c_l \sum\limits_{j=1}^k\int_{\R^N}\big(Z_{x_j^+,\lambda}^{2_s^*-2}Z_{j,l}^+ + Z_{x_j^-,\lambda}^{2_s^*-2}Z_{j,l}^-\big)\frac{\partial Z_{\bar{r},\bar{h},\bar{y}'',\lambda}}{\partial \lambda}\notag\\
      &=c_1 \sum\limits_{j=1}^k\int_{\R^N}\big(Z_{x_j^+,\lambda}^{2_s^*-2}Z_{j,1}^+ + Z_{x_j^-,\lambda}^{2_s^*-2}Z_{j,1}^-\big)\frac{\partial Z_{\bar{r},\bar{h},\bar{y}'',\lambda}}{\partial \lambda} + o\Big(\frac{k}{\lambda^{2}}\Big)c_1\notag\\
      &=\frac{2k}{\lambda^{2}}\big(a_4+o(1)\big)c_1+o\Big(\frac{k}{\lambda^{2}}\Big)c_1,
  \end{split}
\end{equation}
which implies that $c_1=0$. The proof is completed.
\end{proof}

Next, we will estimate \eqref{3.4} and \eqref{3.5}.
A direct computation gives
\begin{align*}
    &\int_{B_\rho}\Big(-V(r,y'')u_k+(u_k)_+^{2_s^*-1}\Big)\frac{\partial u_k}{\partial y_i}\\
    &=-\frac{1}{2}\int_{\partial B_\rho}V(r,y'')u_k^2 \nu_i + \frac{1}{2}\int_{B_\rho}\frac{\partial V(r,y'')}{\partial y_i}u_k^2
    + \frac{1}{2_s^*}\int_{\partial B_\rho}(u_k)_+^{2_s^*} \nu_i.
\end{align*}
Hence, we have that \eqref{3.4} is equivalent to
\begin{equation}\label{3.20}
  \begin{split}
    \frac{1}{2}\int_{B_\rho}\frac{\partial V(r,y'')}{\partial y_i}u_k^2 &= -\int_{\partial''{B_\rho^+}} t^{1-2s}\frac{\partial\tilde{u}_k}{\partial \nu}\frac{\partial\tilde{u}_k}{\partial y_i} +\frac{1}{2}\int_{\partial''{B_\rho^+}} t^{1-2s}|\nabla\tilde{u}_k|^2\nu_i\\
    &\quad  + \frac{1}{2}\int_{\partial B_\rho}V(r,y'')u_k^2 \nu_i - \frac{1}{2_s^*}\int_{\partial B_\rho}(u_k)_+^{2_s^*} \nu_i,
  \end{split}
\end{equation}
for $i=4,\cdots,N$.

Similarly, from \eqref{3.1} we can obtain
\begin{align*}
    &\int_{\partial{B_\rho^+}} t^{1-2s}\frac{\partial\tilde{u}_k}{\partial \nu}\tilde{u}_k  =\int_{\partial''{B_\rho^+}}t^{1-2s}\frac{\partial\tilde{u}_k}{\partial \nu}\tilde{u}_k \\
    &\qquad + \int_{B_\rho}\Big(-V(r,y'')u_k^2 +(u_k)_+^{2_s^*}+\sum\limits_{l=1}^N
c_l\sum\limits_{j=1}^k\big(Z_{x_j^+,\lambda}^{2_s^*-2}Z_{j,l}^+ +Z_{x_j^-,\lambda}^{2_s^*-2}Z_{j,l}^-\big)\Big)u_k.
\end{align*}
A direct computation gives
\begin{align*}
    &\int_{B_\rho}\Big(-V(r,y'')u_k+(u_k)_+^{2_s^*-1}\Big)\langle \nabla u_k, y\rangle\\
    &=\int_{B_\rho} \Big(-\frac{1}{2}V(r,y'')\langle \nabla u_k^2, y\rangle + \frac{1}{2_s^*}\langle \nabla (u_k)_+^{2_s^*}, y\rangle\Big)\\
    &=-\frac{1}{2}\int_{\partial B_\rho}V(r,y'')u_k^2 \langle y,\nu\rangle + \frac{1}{2}\int_{B_\rho}\Big(NV(r,y'')+\langle \nabla V(r,y''), y\rangle\Big)u_k^2 \\
    &\qquad + \frac{1}{2_s^*}\int_{\partial B_\rho}(u_k)_+^{2_s^*} \langle y,\nu\rangle + \frac{2s-N}{2}\int_{B_\rho}(u_k)_+^{2_s^*},
\end{align*}
which, together with
\begin{equation*}
\sum\limits_{l=1}^N c_l \sum\limits_{j=1}^k\int_{\R^N}\big(Z_{x_j^+,\lambda}^{2_s^*-2}Z_{j,l}^+ + Z_{x_j^-,\lambda}^{2_s^*-2}Z_{j,l}^-\big)\varphi=0,
\end{equation*}
yields that \eqref{3.5} is equivalent to
\begin{equation}\label{3.19}
  \begin{split}
    &\int_{B_\rho}\Big(sV(r,y'')+\frac{1}{2}\langle \nabla V(r,y''), y\rangle\Big)u_k^2 \\
    &= -\int_{\partial''{B_\rho^+}}t^{1-2s}\langle\nabla\tilde{u}_k,Y\rangle\frac{\partial\tilde{u}_k}{\partial \nu}+\frac{1}{2}\int_{\partial''{B_\rho^+}}t^{1-2s}|\nabla\tilde{u}_k|^2\langle Y,\nu\rangle \\
    &\qquad +\frac{2s-N}{2}\int_{\partial''{B_\rho^+}}t^{1-2s}\frac{\partial\tilde{u}_k}{\partial \nu}\tilde{u}_k
     + \frac{1}{2}\int_{\partial B_\rho}V(r,y'')u_k^2 \langle y,\nu\rangle
- \frac{1}{2_s^*}\int_{\partial B_\rho}(u_k)_+^{2_s^*} \langle y,\nu\rangle\\
     &\qquad + \frac{2s-N}{2}\sum\limits_{l=1}^N c_l \int_{B_\rho}\Big(\sum\limits_{j=1}^k\big(Z_{x_j^+,\lambda}^{2_s^*-2}Z_{j,l}^+ +Z_{x_j^-,\lambda}^{2_s^*-2}Z_{j,l}^-\big)\Big)Z_{\bar{r},\bar{h},\bar{y}'',\lambda}\\
     &:=-S_1+S_2+S_3+S_4-S_5+S_6.
  \end{split}
\end{equation}

\begin{lemma}\label{lem3.2}
Relations \eqref{3.19} and \eqref{3.20} are, respectively, equivalent to
\begin{equation}\label{3.21}
  \int_{B_\rho}\Big(sV(r,y'')+\frac{1}{2}\langle \nabla V(r,y''), y\rangle\Big)u_k^2 = O\Big(\frac{k}{\lambda^{2s+\epsilon}}\Big),
\end{equation}
and
\begin{equation}\label{3.22}
  \frac{1}{2}\int_{B_\rho}\frac{\partial V(r,y'')}{\partial y_i}u_k^2 = O\Big(\frac{k}{\lambda^{2s+\epsilon}}\Big), \quad i=4,\cdots,N.
\end{equation}
\end{lemma}

\begin{proof}
Here, we only give the proof for \eqref{3.21}, the proof of \eqref{3.22} is similar. We will deal with the terms in the right-hand side of \eqref{3.19} one by one.

For the first term $S_1$, noting that $\tilde{u}_k=\tilde{Z}_{\bar{r},\bar{h},\bar{y}'',\lambda}+\tilde{\varphi}$, we have
\begin{equation}\label{a3.22}
  \begin{split}
    &\int_{\partial''{B_\rho^+}}t^{1-2s}\langle\nabla\tilde{u}_k,Y\rangle\frac{\partial\tilde{u}_k}{\partial \nu} \\ &=\int_{\partial''{B_\rho^+}}t^{1-2s}\langle\nabla\tilde{Z}_{\bar{r},\bar{h},\bar{y}'',\lambda},Y\rangle\frac{\partial\tilde{Z}_{\bar{r},\bar{h},\bar{y}''}}{\partial \nu} +\int_{\partial''{B_\rho^+}}t^{1-2s}\langle\nabla\tilde{\varphi},Y\rangle\frac{\partial\tilde{\varphi}}{\partial \nu} \\ &\qquad +\int_{\partial''{B_\rho^+}}t^{1-2s}\langle\nabla\tilde{Z}_{\bar{r},\bar{h},\bar{y}'',\lambda},Y\rangle\frac{\partial\tilde{\varphi}}{\partial \nu} +\int_{\partial''{B_\rho^+}}t^{1-2s}\langle\nabla\tilde{\varphi},Y\rangle\frac{\partial\tilde{Z}_{\bar{r},\bar{h},\bar{y}'',\lambda}}{\partial \nu}.
  \end{split}
\end{equation}
Next, we will estimate the terms in \eqref{a3.22} one by one.

Using Lemma \ref{lemA.7}, we obtain
\begin{align}\label{3.39A}
    &\Big|\int_{\partial''{B_\rho^+}}t^{1-2s}\langle\nabla\tilde{Z}_{\bar{r},\bar{h},\bar{y}'',\lambda},Y\rangle\frac{\partial\tilde{Z}_{\bar{r},\bar{h},\bar{y}''}}{\partial \nu}\Big|\notag\\
    &\leq\int_{\partial''{B_\rho^+}}t^{1-2s}|\nabla\tilde{Z}_{\bar{r},\bar{h},\bar{y}'',\lambda}|^2\notag\\
    &\leq\frac{C}{\lambda^{N-2s}}\int_{\partial''{B_\rho^+}}t^{1-2s}\Big(\sum\limits_{j=1}^{k}\frac{1}{(1+|y-x^+_j|)^{N-2s+1}}+\sum\limits_{j=1}^{k}\frac{1}{(1+|y-x^-_j|)^{N-2s+1}}\Big)^2\notag\\
    &\leq\frac{Ck^2}{\lambda^{N-2s}}\int_{\partial''{B_\rho^+}}\frac{t^{1-2s}}{(1+|y-x^+_1|)^{2N-4s+2}}
    \leq\frac{Ck^2}{\lambda^{N-2s}}\leq\frac{Ck}{\lambda^{2s+\epsilon}},
\end{align}
since $\frac{N-4s}{N-2s}<N-4s$.
By Lemma \ref{lemA.8}, there holds
\begin{equation}\label{3.39B}
  \begin{split}
    &\Big|\int_{\partial''{B_\rho^+}}t^{1-2s}\langle\nabla\tilde{\varphi},Y\rangle\frac{\partial\tilde{\varphi}}{\partial \nu}\Big|
    \leq C\int_{\partial''{B_\rho^+}}t^{1-2s}|\nabla\tilde{\varphi}|^2
    \leq \frac{Ck\|\varphi\|^2_*}{\lambda^\tau}\leq\frac{Ck}{\lambda^{2s+\epsilon}}.
  \end{split}
\end{equation}

By the process of the proofs of \eqref{3.39A} and \eqref{3.39B}, together with H\"older inequality, we have
\begin{equation}\label{3.39C}
  \begin{split}
  &\Big|\int_{\partial''{B_\rho^+}}t^{1-2s}\langle\nabla\tilde{Z}_{\bar{r},\bar{h},\bar{y}'',\lambda},Y\rangle\frac{\partial\tilde{\varphi}}{\partial \nu}\Big|
  \leq C \int_{\partial''{B_\rho^+}}t^{1-2s}|\nabla\tilde{Z}_{\bar{r},\bar{h},\bar{y}'',\lambda}||\nabla\tilde{\varphi}|\\
  &\leq C \Big(\int_{\partial''{B_\rho^+}}t^{1-2s}|\nabla\tilde{Z}_{\bar{r},\bar{h},\bar{y}'',\lambda}|^2\Big)^{\frac{1}{2}}
  \Big(\int_{\partial''{B_\rho^+}}t^{1-2s}|\nabla\tilde{\varphi}|^2\Big)^{\frac{1}{2}}
  \leq\frac{Ck}{\lambda^{2s+\epsilon}}.
    \end{split}
\end{equation}
Similar to \eqref{3.39C}, we have
\begin{equation}\label{3.39D}
  \Big|\int_{\partial''{B_\rho^+}}t^{1-2s}\langle\nabla \tilde{\varphi} ,Y\rangle\frac{\partial\tilde{Z}_{\bar{r},\bar{h},\bar{y}'',\lambda}}{\partial \nu}\Big|
  \leq C \int_{\partial''{B_\rho^+}}t^{1-2s}|\nabla\tilde{\varphi}||\nabla\tilde{Z}_{\bar{r},\bar{h},\bar{y}'',\lambda}|
  \leq\frac{Ck}{\lambda^{2s+\epsilon}}.
\end{equation}
Hence, we have proved that
\begin{equation}\label{3.39E}
  |S_1|=\Big|\int_{\partial''{B_\rho^+}}t^{1-2s}\langle\nabla\tilde{u}_k,Y\rangle\frac{\partial\tilde{u}_k}{\partial \nu}\Big|\leq \frac{Ck}{\lambda^{2s+\epsilon}}.
\end{equation}

By the same argument as the calculation of \eqref{3.39E}, we can prove
\begin{equation}\label{3.39F}
  |S_2|=\Big|\frac{1}{2}\int_{\partial''{B_\rho^+}}t^{1-2s}|\nabla\tilde{u}_k|^2\langle Y,\nu\rangle\Big|\leq \frac{Ck}{\lambda^{2s+\epsilon}}.
\end{equation}

Next, we estimate the term $S_3$.
\begin{equation*}
  \begin{split}
    \int_{\partial''{B_\rho^+}}t^{1-2s}\frac{\partial\tilde{u}_k}{\partial \nu}\tilde{u}_k
    &=\int_{\partial''{B_\rho^+}}t^{1-2s}\frac{\partial\tilde{Z}_{\bar{r},\bar{h},\bar{y}'',\lambda}}{\partial \nu}\tilde{Z}_{\bar{r},\bar{h},\bar{y}'',\lambda} +\int_{\partial''{B_\rho^+}}t^{1-2s}\frac{\partial\tilde{\varphi}}{\partial \nu}\tilde{\varphi}\\
    &\qquad +\int_{\partial''{B_\rho^+}}t^{1-2s}\frac{\partial\tilde{Z}_{\bar{r},\bar{h},\bar{y}'',\lambda}}{\partial \nu}\tilde{\varphi}+\int_{\partial''{B_\rho^+}}t^{1-2s}\frac{\partial\tilde{\varphi}}{\partial \nu}\tilde{Z}_{\bar{r},\bar{h},\bar{y}'',\lambda}.
  \end{split}
\end{equation*}
By Lemma \ref{lemA.7}, we have
\begin{align*}
   &\Big|\int_{\partial''{B_\rho^+}}t^{1-2s}\frac{\partial\tilde{Z}_{\bar{r},\bar{h},\bar{y}'',\lambda}}{\partial \nu}\tilde{Z}_{\bar{r},\bar{h},\bar{y}'',\lambda} \Big|\\
   &\leq\frac{C}{\lambda^{N-2s}}\int_{\partial''{B_\rho^+}}t^{1-2s}\Big(\sum\limits_{j=1}^{k}\frac{1}{(1+|y-x^+_j|)^{N-2s+1}}+\sum\limits_{j=1}^{k}\frac{1}{(1+|y-x^-_j|)^{N-2s+1}}\Big)\\
   &\quad\times\Big(\sum\limits_{j=1}^{k}\frac{1}{(1+|y-x^+_j|)^{N-2s}}+\sum\limits_{j=1}^{k}\frac{1}{(1+|y-x^-_j|)^{N-2s}}\Big)\\
   &\leq\frac{Ck^2}{\lambda^{N-2s}}\int_{\partial''{B_\rho^+}}\frac{t^{1-2s}}{(1+|y-x^+_1|)^{2N-4s+1}}\\
   &\leq\frac{Ck^2}{\lambda^{N-2s}}\leq\frac{Ck}{\lambda^{2s+\epsilon}}.
\end{align*}

It follows from Lemma \ref{lemA.8} that
\begin{align*}
  \Big|\int_{\partial''{B_\rho^+}}t^{1-2s}\frac{\partial\tilde{\varphi}}{\partial \nu}\tilde{\varphi}\Big|
  &\leq\Big(\int_{\partial''{B_\rho^+}}t^{1-2s}|\nabla\tilde{\varphi}|^2\Big)^{\frac{1}{2}}\Big(\int_{\partial''{B_\rho^+}}t^{1-2s}|\tilde{\varphi}|^2\Big)^{\frac{1}{2}}\\
  &\leq\frac{Ck\|\varphi\|^2_*}{\lambda^\tau}\leq\frac{Ck}{\lambda^{2s+\epsilon}}.
\end{align*}
Similarly, we have
\begin{equation*}
 \Big| \int_{\partial''{B_\rho^+}}t^{1-2s}\frac{\partial\tilde{Z}_{\bar{r},\bar{h},\bar{y}'',\lambda}}{\partial\nu}\tilde{\varphi}\Big|
 \leq\frac{Ck}{\lambda^{2s+\epsilon}},
\end{equation*}
and
\begin{equation*}
 \Big|\int_{\partial''{B_\rho^+}}t^{1-2s}\frac{\partial\tilde{\varphi}}{\partial\nu}\tilde{Z}_{\bar{r},\bar{h},\bar{y}'',\lambda}\Big|
 \leq\frac{Ck}{\lambda^{2s+\epsilon}}.
\end{equation*}
So, we can obtain that
\begin{equation*}
  |S_3|=\Big|\frac{2s-N}{2}\int_{\partial''{B_\rho^+}}t^{1-2s}\frac{\partial\tilde{u}_k}{\partial \nu}\tilde{u}_k\Big|\leq \frac{Ck}{\lambda^{2s+\epsilon}}.
\end{equation*}

Next, we estimate the terms $S_4$ and $S_5$. Since $u_k|_{\partial B_\rho}=\varphi$, we deduce that
\begin{equation*}
  \begin{split}
  &|S_4|=\Big|\frac{1}{2}\int_{\partial B_\rho}V(r,y'')u_k^2 \langle y,\nu\rangle\Big|\\
  &\leq C\|\varphi\|^2_* \int_{\partial B_\rho}\Big(\sum\limits_{j=1}^{k}\big(\frac{\lambda^{\frac{N-2s}{2}}}{(1+\lambda|y-x_j^+|)^{\frac{N-2s}{2}+\tau}}+\frac{\lambda^{\frac{N-2s}{2}}}{(1+\lambda|y-x_j^-|)^{\frac{N-2s}{2}+\tau}}\big)\Big)^2\\
  &\leq \frac{Ck^2\|\varphi\|^2_*}{\lambda^{2\tau}} \leq \frac{Ck}{\lambda^{2s+\epsilon}}.
  \end{split}
\end{equation*}

Similarly, we have
\begin{equation*}
  |S_5|=\Big|\frac{1}{2_s^*}\int_{\partial B_\rho}(u_k)_+^{2_s^*} \langle y,\nu\rangle\Big|
\leq \frac{Ck^{2_s^*}\|\varphi\|^{2_s^*}_*}{\lambda^{{2_s^*}\tau}} \leq \frac{Ck}{\lambda^{2s+\epsilon}}.
\end{equation*}

Finally, we estimate the term $S_6$. Note that
\begin{equation*}
  \begin{split}
&\sum\limits_{j=1}^k\int_{B_\rho}Z_{x_j^+,\lambda}^{2_s^*-2}Z_{j,l}^+ Z_{\bar{r},\bar{h},\bar{y}'',\lambda}\\
&=\sum\limits_{j=1}^k\int_{B_\rho}Z_{x_j^+,\lambda}^{2_s^*-1}Z_{j,l}^+ +\sum\limits_{j=1}^k\int_{B_\rho}\sum\limits_{i\neq j}Z_{x_j^+,\lambda}^{2_s^*-2}Z_{j,l}^+ Z_{x_i^+,\lambda}\\
&\quad +\sum\limits_{j=1}^k\sum\limits_{i=1}^k\int_{B_\rho}Z_{x_j^+,\lambda}^{2_s^*-2}Z_{j,l}^+ Z_{x_i^-,\lambda}\\
&=O\Big(\frac{k\lambda^{n_l}}{\lambda^{s}}\Big).
  \end{split}
\end{equation*}
Similarly, we have
\begin{equation*}
\sum\limits_{j=1}^k\int_{B_\rho}Z_{x_j^-,\lambda}^{2_s^*-2}Z_{j,l}^- Z_{\bar{r},\bar{h},\bar{y}'',\lambda}
=O\Big(\frac{k\lambda^{n_l}}{\lambda^{s}}\Big),
\end{equation*}
which, together with Proposition \ref{prop2.5}, yields that
\begin{equation*}
  |S_6|=\Big|\frac{2s-N}{2}\sum\limits_{l=1}^N c_l \int_{B_\rho}\Big(\sum\limits_{j=1}^k\big(Z_{x_j^+,\lambda}^{2_s^*-2}Z_{j,l}^+ +Z_{x_j^-,\lambda}^{2_s^*-2}Z_{j,l}^-\big)\Big)Z_{\bar{r},\bar{h},\bar{y}'',\lambda}\Big|\leq \frac{Ck}{\lambda^{2s+\epsilon}}.
\end{equation*}
Combining the above estimates, we have that \eqref{3.19} is equivalent to
\begin{equation*}
  \int_{B_\rho}\Big(sV(r,y'')+\frac{1}{2}\langle \nabla V(r,y''), y\rangle\Big)u_k^2 = O\Big(\frac{k}{\lambda^{2s+\epsilon}}\Big).
\end{equation*}\vspace{-10pt}
\end{proof}

\begin{lemma}\label{lemB.2}
We have
\begin{equation}\label{B.2}
  \begin{split}
  &\int_{\R^N}\Big((-\Delta)^s u_k+V(r,y'')u_k-(u_k)_+^{2_s^*-1}\Big)\frac{\partial Z_{\bar{r},\bar{h},\bar{y}'',\lambda}}{\partial \lambda}\\
&=k\bigg(-\frac{2sB_1V(\bar{r},\bar{y}'')}{\lambda^{2s+1}}+\frac{(N-2s)B_2k^{N-2s}}{\lambda^{N-2s+1}(\sqrt{1-\bar{h}^2})^{N-2s}}
  +\frac{(N-2s)B_3k}{\lambda^{N-2s+1}\bar{h}^{N-2s-1}\sqrt{1-\bar{h}^2}}\\
  &\quad
  +O\Big(\frac{1}{\lambda^{2s+1+\epsilon}}\Big)\bigg),
  \end{split}
\end{equation}
where $B_1, B_2$ and $B_3$ are the same positive constants in Lemma \ref{lemB.0}.
\end{lemma}

\begin{proof}
A direct computation gives
\begin{equation*}
  \begin{split}
&\int_{\R^N}\big((-\Delta)^{s}u_k +V(r,y'')u_k -(u_k)_+^{2_s^*-1}\big)\frac{\partial Z_{\bar{r},\bar{h},\bar{y}'',\lambda}} {\partial\lambda}\\
&=\Big\langle I'(Z_{\bar{r},\bar{h},\bar{y}'',\lambda}),\frac{\partial Z_{\bar{r},\bar{h},\bar{y}'',\lambda}} {\partial\lambda}\Big\rangle \\
&\quad+ 2k\Big\langle(-\Delta)^s\varphi+V(r,y'')\varphi-(2_s^*-1)Z_{\bar{r},\bar{h},\bar{y}'',\lambda}^{2_s^*-2}\varphi,\frac{\partial Z_{x_1^+,\lambda}} {\partial\lambda}\Big\rangle\\
&\quad-\int_{\R^N}\big((Z_{\bar{r},\bar{h},\bar{y}'',\lambda}+\varphi)_+^{2_s^*-1}-Z_{\bar{r},\bar{h},\bar{y}'',\lambda}^{2_s^*-1}-(2_s^*-1)Z_{\bar{r},\bar{h},\bar{y}'',\lambda}^{2_s^*-2}\varphi\big)\frac{\partial Z_{\bar{r},\bar{h},\bar{y}'',\lambda}} {\partial\lambda}\\
&=:\Big\langle I'(Z_{\bar{r},\bar{h},\bar{y}'',\lambda}),\frac{\partial Z_{\bar{r},\bar{h},\bar{y}'',\lambda}} {\partial\lambda}\Big\rangle +2kH_1-H_2.
  \end{split}
\end{equation*}

Similar to \eqref{2.14} and \eqref{2.13}, by direct computations we have
\begin{equation}\label{b2.1}
  \begin{split}
|H_1|=O\Big(\frac{\|\varphi\|_*}{\lambda^{s+1}}\Big)
=O\Big(\frac{1}{\lambda^{2s+1+\epsilon}}\Big).
  \end{split}
\end{equation}

Note that
\begin{equation*}
 |(1+t)_+^p-1-pt|\leq
\left\{
 \begin{array}{ll}
     Ct^2, & \hbox{$1<p\leq2$;} \vspace{0.15cm}\\
     C(t^2+|t|^p), & \hbox{$p>2$.}
   \end{array}
 \right.
\end{equation*}
If $2_s^*\leq3$, then we have
\begin{align*}
|H_2|&=\Big|\int_{\R^N}\big((Z_{\bar{r},\bar{h},\bar{y}'',\lambda}+\varphi)_+^{2_s^*-1}-Z_{\bar{r},\bar{h},\bar{y}'',\lambda}^{2_s^*-1}-(2_s^*-1)Z_{\bar{r},\bar{h},\bar{y}'',\lambda}^{2_s^*-2}\varphi\big)\frac{\partial Z_{\bar{r},\bar{h},\bar{y}'',\lambda}} {\partial\lambda}\Big|\\
&\leq C\Big|\int_{\R^N}Z_{\bar{r},\bar{h},\bar{y}'',\lambda}^{2_s^*-3}\varphi^2\frac{\partial Z_{\bar{r},\bar{h},\bar{y}'',\lambda}} {\partial\lambda}\Big|\\
&\leq\frac{C\|\varphi\|_*^2}{\lambda}\int_{\R^N}\bigg(\sum\limits_{j=1}^{k}\Big(\frac{\lambda^{\frac{N-2s}{2}}}{(1+\lambda|y-x_j^+|)^{N-2s}}+\frac{\lambda^{\frac{N-2s}{2}}}{(1+\lambda|y-x_j^-|)^{N-2s}}\Big)\bigg)^{2_s^*-2}\\
&\quad\times\bigg(\sum\limits_{j=1}^{k}\Big(\frac{\lambda^{\frac{N-2s}{2}}}{(1+\lambda|y-x_j^+|)^{\frac{N-2s}{2}+\tau}}+\frac{\lambda^{\frac{N-2s}{2}}}{(1+\lambda|y-x_j^-|)^{\frac{N-2s}{2}+\tau}}\Big)\bigg)^2\\
&\leq\frac{C\|\varphi\|_*^2}{\lambda}\int_{\R^N}\lambda^N\bigg(\sum\limits_{j=1}^{k}\Big(\frac{1}{(1+\lambda|y-x_j^+|)^{4s}}+\frac{1}{(1+\lambda|y-x_j^-|)^{4s}}\Big)\bigg)\\
&\quad\times\bigg(\sum\limits_{j=1}^{k}\Big(\frac{1}{(1+\lambda|y-x_j^+|)^{N-2s+2\tau}}+\frac{1}{(1+\lambda|y-x_j^-|)^{N-2s+2\tau}}\Big)\bigg)\\
&\leq\frac{Ck\|\varphi\|_*^2}{\lambda}\Big(C'+\sum\limits_{j=2}^{k}
\frac{1}{(\lambda|x_1^+-x_j^+|)^{2s+\tau}}+\sum\limits_{j=1}^{k}\frac{1}{(\lambda|x_1^+-x_j^-|)^{2s+\tau}}\Big)\\
&=O\Big(\frac{k}{\lambda^{2s+1+\epsilon}}\Big).
  \end{align*}
Similarly, if $2_s^*>3$, we have,
\begin{align*}
|H_2|\leq C\Big|\int_{\R^N}\big(Z_{\bar{r},\bar{h},\bar{y}'',\lambda}^{2_s^*-3}\varphi^2+\varphi^{2_s^*-1}\big)\frac{\partial Z_{\bar{r},\bar{h},\bar{y}'',\lambda}} {\partial\lambda}\Big|
=O\Big(\frac{k}{\lambda^{2s+1+\epsilon}}\Big).
\end{align*}
As a consequence, we have
\begin{align*}
\Big\langle I'(Z_{\bar{r},\bar{h},\bar{y}'',\lambda}+\varphi),\frac{\partial Z_{\bar{r},\bar{h},\bar{y}'',\lambda}} {\partial\lambda}\Big\rangle
=\Big\langle I'(Z_{\bar{r},\bar{h},\bar{y}'',\lambda}),\frac{\partial Z_{\bar{r},\bar{h},\bar{y}'',\lambda}} {\partial\lambda}\Big\rangle +O\Big(\frac{k}{\lambda^{2s+1+\epsilon}}\Big),
\end{align*}
which, together with Lemma \ref{lemB.1}, yields that \eqref{B.2}. The proof is completed.
\end{proof}

\begin{lemma}\label{lemB.2a}
We have
\begin{equation}\label{B.2a}
  \begin{split}
&\int_{\R^N}\Big((-\Delta)^s u_k+V(r,y'')u_k-(u_k)_+^{2_s^*-1}\Big)\frac{\partial Z_{\bar{r},\bar{h},\bar{y}'',\lambda}}{\partial \bar{h}}\\
&=k\bigg(-\frac{(N-2s)B_2\bar{h}k^{N-2s}}{\lambda^{N-2s}(\sqrt{1-\bar{h}^2})^{N-2s+2}}
    +\frac{(N-2s-1)B_3k}{\lambda^{N-2s}\bar{h}^{N-2s}\sqrt{1-\bar{h}^2}}
    + O\Big(\frac{\bar{h}}{\lambda^{2s+\epsilon}}\Big)\bigg),
  \end{split}
\end{equation}
where $B_2, B_3>0$ are the same positive constants in Lemma \ref{lemB.0}.
\end{lemma}

\begin{proof}
A direct computation gives
\begin{equation*}
  \begin{split}
&\int_{\R^N}\big((-\Delta)^{s}u_k +V(r,y'')u_k -(u_k)_+^{2_s^*-1}\big)\frac{\partial Z_{\bar{r},\bar{h},\bar{y}'',\lambda}} {\partial \bar{h}}\\
&=\Big\langle I'(Z_{\bar{r},\bar{h},\bar{y}'',\lambda}),\frac{\partial Z_{\bar{r},\bar{h},\bar{y}'',\lambda}} {\partial\bar{h}}\Big\rangle  \\
&\quad+\Big\langle(-\Delta)^s\varphi+V(r,y'')\varphi-(2_s^*-1)Z_{\bar{r},\bar{h},\bar{y}'',\lambda}^{2_s^*-2}\varphi,\frac{\partial Z_{\bar{r},\bar{h},\bar{y}'',\lambda}} {\partial\bar{h}}\Big\rangle\\
&\quad-\int_{\R^N}\big((Z_{\bar{r},\bar{h},\bar{y}'',\lambda}+\varphi)_+^{2_s^*-1}-Z_{\bar{r},\bar{h},\bar{y}'',\lambda}^{2_s^*-1}-(2_s^*-1)Z_{\bar{r},\bar{h},\bar{y}'',\lambda}^{2_s^*-2}\varphi\big)\frac{\partial Z_{\bar{r},\bar{h},\bar{y}'',\lambda}} {\partial\bar{h}}.
  \end{split}
\end{equation*}

Since $\varphi$ and $Z_{\bar{r},\bar{h},\bar{y}'',\lambda}$ are even in $y_2,y_3$, similar to \eqref{2.14} and \eqref{2.13}, by direct computations we have
\begin{equation}\label{3.36a}
  \begin{split}
&\Big|\Big\langle(-\Delta)^s\varphi+V(r,y'')\varphi-(2_s^*-1)Z_{\bar{r},\bar{h},\bar{y}'',\lambda}^{2_s^*-2}\varphi,\frac{\partial Z_{\bar{r},\bar{h},\bar{y}'',\lambda}} {\partial\bar{h}}\Big\rangle\Big|\\
&=\Big|\langle(-\Delta)^s\varphi+V(r,y'')\varphi-(2_s^*-1)(Z^*_{\bar{r},\bar{h},\bar{y}'',\lambda})^{2_s^*-2}\varphi,\frac{\partial Z^*_{\bar{r},\bar{h},\bar{y}'',\lambda}} {\partial\bar{h}}\Big\rangle\Big|
+O\Big(\frac{k\|\varphi\|_*}{\lambda^{\frac{2s+1}{2}}}\Big)\\
&=O\Big(\frac{k\bar{h}\|\varphi\|_*}{\lambda^{\frac{2s-1}{2}}}\Big)
=O\Big(\frac{k\bar{h}}{\lambda^{2s+\epsilon}}\Big).
  \end{split}
\end{equation}

On the other hand, if $2_s^*\leq3$, then we have
\begin{equation}\label{3.36a.1}
  \begin{split}
&\Big|\int_{\R^N}\big((Z_{\bar{r},\bar{h},\bar{y}'',\lambda}+\varphi)_+^{2_s^*-1}-Z_{\bar{r},\bar{h},\bar{y}'',\lambda}^{2_s^*-1}-(2_s^*-1)Z_{\bar{r},\bar{h},\bar{y}'',\lambda}^{2_s^*-2}\varphi\big)\frac{\partial Z_{\bar{r},\bar{h},\bar{y}'',\lambda}} {\partial\bar{h}}\Big|\notag\\
&=\Bigg|\int_{\R^N}\big((Z^*_{\bar{r},\bar{h},\bar{y}'',\lambda}+\varphi)_+^{2_s^*-1}-(Z^*_{\bar{r},\bar{h},\bar{y}'',\lambda})^{2_s^*-1}-(2_s^*-1)(Z^*_{\bar{r},\bar{h},\bar{y}'',\lambda})^{2_s^*-2}\varphi\big)\\
&\quad\times\frac{\partial Z^*_{\bar{r},\bar{h},\bar{y}'',\lambda}} {\partial\bar{h}}\Bigg|+O\Big(\frac{k\|\varphi\|_*}{\lambda^{\frac{2s+1}{2}}}\Big)\notag\\
&\leq C\bar{h} \Big|\int_{\R^N}(Z^*_{\bar{r},\bar{h},\bar{y}'',\lambda})^{2_s^*-3}\varphi^2\frac{\partial Z^*_{\bar{r},\bar{h},\bar{y}'',\lambda}} {\partial y_1}\Big|+O\Big(\frac{k\|\varphi\|_*}{\lambda^{\frac{2s+1}{2}}}\Big)\notag\\
&\leq C\lambda\bar{h}\|\varphi\|_*^2\int_{\R^N}\bigg(\sum\limits_{j=1}^{k}\Big(\frac{\lambda^{\frac{N-2s}{2}}}{(1+\lambda|y-x_j^+|)^{N-2s}}+\frac{\lambda^{\frac{N-2s}{2}}}{(1+\lambda|y-x_j^-|)^{N-2s}}\Big)\bigg)^{2_s^*-2}\notag\\
&\quad\times\bigg(\sum\limits_{j=1}^{k}\Big(\frac{\lambda^{\frac{N-2s}{2}}}{(1+\lambda|y-x_j^+|)^{\frac{N-2s}{2}+\tau}}+\frac{\lambda^{\frac{N-2s}{2}}}{(1+\lambda|y-x_j^-|)^{\frac{N-2s}{2}+\tau}}\Big)\bigg)^2+O\Big(\frac{k\|\varphi\|_*}{\lambda^{\frac{2s+1}{2}}}\Big)\\
&\leq C\lambda\bar{h}\|\varphi\|_*^2\int_{\R^N}\lambda^N\bigg(\sum\limits_{j=1}^{k}\Big(\frac{1}{(1+\lambda|y-x_j^+|)^{4s}}+\frac{1}{(1+\lambda|y-x_j^-|)^{4s}}\Big)\bigg)\notag\\
&\quad\times\bigg(\sum\limits_{j=1}^{k}\Big(\frac{1}{(1+\lambda|y-x_j^+|)^{N-2s+2\tau}}+\frac{1}{(1+\lambda|y-x_j^-|)^{N-2s+2\tau}}\Big)\bigg)+O\Big(\frac{k\|\varphi\|_*}{\lambda^{\frac{2s+1}{2}}}\Big)\notag\\
&\leq Ck\lambda\bar{h}\|\varphi\|_*^2+O\Big(\frac{k\|\varphi\|_*}{\lambda^{\frac{2s+1}{2}}}\Big)\notag\\
&=O\Big(\frac{k\bar{h}}{\lambda^{2s+\epsilon}}\Big)\notag.
  \end{split}
\end{equation}
Similarly, if $2_s^*>3$, then we have,
\begin{align*}
&\int_{\R^N}\big((Z_{\bar{r},\bar{h},\bar{y}'',\lambda}+\varphi)_+^{2_s^*-1}-Z_{\bar{r},\bar{h},\bar{y}'',\lambda}^{2_s^*-1}-(2_s^*-1)Z_{\bar{r},\bar{h},\bar{y}'',\lambda}^{2_s^*-2}\varphi\big)\frac{\partial Z_{\bar{r},\bar{h},\bar{y}'',\lambda}} {\partial\bar{h}}\\
&\leq C\bar{h}\Big|\int_{\R^N}\big((Z^*_{\bar{r},\bar{h},\bar{y}'',\lambda})^{2_s^*-3}\varphi^2+\varphi^{2_s^*-1}\big)\frac{\partial Z^*_{\bar{r},\bar{h},\bar{y}'',\lambda}} {\partial y_1}\Big|+O\Big(\frac{k\|\varphi\|_*}{\lambda^{\frac{2s+1}{2}}}\Big)\\
&=O\Big(\frac{k\bar{h}}{\lambda^{2s+\epsilon}}\Big).
\end{align*}
As a consequence, we have
\begin{align*}
\Big\langle I'(Z_{\bar{r},\bar{h},\bar{y}'',\lambda}+\varphi),\frac{\partial Z_{\bar{r},\bar{h},\bar{y}'',\lambda}} {\partial\bar{h}}\Big\rangle
=\Big\langle I'(Z_{\bar{r},\bar{h},\bar{y}'',\lambda}),\frac{\partial Z_{\bar{r},\bar{h},\bar{y}'',\lambda}} {\partial \bar{h}}\Big\rangle
+O\Big(\frac{k\bar{h}}{\lambda^{2s+\epsilon}}\Big).
\end{align*}
which, together with Lemma \ref{lemB.1a}, yields that \eqref{B.2a}. The proof is completed.
\end{proof}

\begin{lemma}\label{lem3.3}
For any function $g(r,y'') \in C^1(\R^N)$, there holds
\begin{equation*}
  \int_{B_\rho}g(r,y'')u_k^2 = 2k\Big(\frac{1}{\lambda^{2s}}g(\bar{r},\bar{y}'')\int_{\R^N}U_{0,1}^2+o\big(\frac{1}{\lambda^{2s}}\big)\Big).
\end{equation*}
\end{lemma}
\begin{proof}
The proof is similar to Lemma 3.3 in \cite{guo2020solutions}, here we omit it.
\end{proof}

By \eqref{3.21} and \eqref{3.22}, it is easy to see that
\begin{equation*}
    \int_{B_\rho}\Big(sV(\bar{r},\bar{y}'')+\frac{1}{2}\bar{r}\frac{\partial V(\bar{r},\bar{y}'')}{\partial \bar{r}}\Big)u_k^2 = O\Big(\frac{k}{\lambda^{2s+\epsilon}}\Big).
\end{equation*}
That is
\begin{equation}\label{3.43}
  \int_{B_\rho} \frac{1}{\bar{r}^{2s-1}}\frac{\partial\big(\bar{r}^{2s}V(\bar{r},\bar{y}'')\big)}{\partial \bar{r}} u_k^2 = O\Big(\frac{k}{\lambda^{2s+\epsilon}}\Big).
\end{equation}
Applying Lemma \ref{lem3.3} to \eqref{3.22} and \eqref{3.43}, we obtain
\begin{equation*}
 2k\bigg(\frac{1}{\lambda^{2s}}\frac{\partial V(\bar{r},\bar{y}'')}{\partial \bar{y}_j}\int_{\R^N}U_{0,1}^2+o\Big(\frac{1}{\lambda^{2s}}\Big)\bigg)=O\Big(\frac{k}{\lambda^{2s+\epsilon}}\Big),
\end{equation*}
for $j=4,\cdots,N$ and
\begin{equation*}
 2k\bigg(\frac{1}{\lambda^{2s}}\frac{1}{\bar{r}^{2s-1}}\frac{\partial\big(\bar{r}^{2s}V(\bar{r},\bar{y}'')\big)}{\partial \bar{r}}\int_{\R^N}U_{0,1}^2+o\Big(\frac{1}{\lambda^{2s}}\Big)\bigg)=O\Big(\frac{k}{\lambda^{2s+\epsilon}}\Big).
\end{equation*}
Therefore, the equations to determine $(\bar{r},\bar{y}'')$ are
\begin{equation}\label{3.45}
 \frac{\partial \big(\bar{r}^{2s}V(\bar{r},\bar{y}'')\big)}{\partial \bar{y}''_j}=o(1),\quad j=4,\cdots,N,
\end{equation}
and
\begin{equation}\label{3.44}
 \frac{\partial \big(\bar{r}^{2s}V(\bar{r},\bar{y}'')\big)}{\partial \bar{r}}=o(1).
\end{equation}
Now we will complete the proof of Theorem \ref{theo1.4}.
\begin{proof}[\textbf{Proof of Theorem \ref{theo1.4}.}]
We have proved that \eqref{3.4}-\eqref{3.6} are equivalent to
\begin{equation}\label{3.46}
 \frac{\partial \big(\bar{r}^{2s}V(\bar{r},\bar{y}'')\big)}{\partial \bar{r}}=o(1),
\end{equation}
\begin{equation}\label{3.47}
 \frac{\partial \big(\bar{r}^{2s}V(\bar{r},\bar{y}'')\big)}{\partial \bar{y}''_j}=o(1), \quad j=4,\cdots,N,
\end{equation}
\begin{equation}\label{3.48}
-\frac{(N-2s)B_2\bar{h}k^{N-2s}}{\lambda^{N-2s}(\sqrt{1-\bar{h}^2})^{N-2s+2}}
+\frac{(N-2s-1)B_3k}{\lambda^{N-2s}\bar{h}^{N-2s}\sqrt{1-\bar{h}^2}}\\
= O\Big(\frac{\bar{h}}{\lambda^{2s+\epsilon}}\Big),
\end{equation}
and
\begin{equation}\label{3.49}
-\frac{2sB_1V(\bar{r},\bar{y}'')}{\lambda^{2s+1}}
+\frac{(N-2s)B_2k^{N-2s}}{\lambda^{N-2s+1}(\sqrt{1-\bar{h}^2})^{N-2s}}
=O\Big(\frac{1}{\lambda^{2s+1+\epsilon}}\Big),
\end{equation}
where $B_1,B_2,B_3$ are the same positive constants in Lemma \ref{lemB.0}.

Letting $\bar{h}=t_1k^{-\frac{N-2s-1}{N-2s+1}}$ and $\lambda=t_2k^{\frac{N-2s}{N-4s}}$, then $(t_1,t_2)\in[L_0,L_1]\times[M_0,M_1]$. It follows from \eqref{3.48} and \eqref{3.49} that
\begin{equation}\label{3.50}
  -t_1+\frac{D_1}{t_1^{N-2s}}=o(1), \quad t_1\in[L_0,L_1]
\end{equation}
and
\begin{equation}\label{3.51}
  -\frac{V(\bar{r},\bar{y}'')}{t_2^{2s+1}}+\frac{D_2}{t_2^{N-2s+1}}=o(1),\quad t_1\in[M_0,M_1]
\end{equation}
where $D_1=\frac{(N-2s-1)B_3}{(N-2s)B_2}$, $D_2=\frac{(N-2s)B_2}{2sB_1}$ are some positive constants.

Let
\begin{align*}
  F(t_1,t_2,\bar{r},\bar{y}'')
  =\bigg(\nabla_{\bar{r},\bar{y}''}(\bar{r}^{2s}V(\bar{r},\bar{y}'')),-t_1+\frac{D_1}{t_1^{N-2s}},  -\frac{V(\bar{r},\bar{y}'')}{t_2^{2s+1}}+\frac{D_2}{t_2^{N-2s+1}}\bigg).
\end{align*}
Then, it holds
\begin{align*}
  &\text{deg}\Big(F(t_1,t_2,\bar{r},\bar{y}''),[L_0,L_1]\times[M_0,M_1] \times B_\theta\big((r_0,y''_0)\big)\Big)\\
  &=\text{deg}\Big(\nabla_{\bar{r},\bar{y}''}\big(\bar{r}^{2s}V(\bar{r},\bar{y}'')\big),B_\theta\big((r_0,y''_0)\big)\Big)\neq0.
\end{align*}
Hence, \eqref{3.46}, \eqref{3.47}, \eqref{3.50} and \eqref{3.51} have a solution $t_1^k\times t_2^k\in[L_0,L_1]\times[M_0,M_1]$, $(\bar{r}_k,\bar{y}''_k)\in B_\theta\big((r_0,y''_0)\big)$.
\end{proof}

\appendix
\setcounter{equation}{0}
\renewcommand{\theequation}{A.\arabic{equation}}
\section{Basic estimates}\label{secA}
In this section, we give some essential estimates. For $x_i,x_j,y \in \R^N$, define
\begin{equation*}
  g_{i,j}(y)=\frac{1}{(1+|y-x_i|)^{\kappa_1}}\frac{1}{(1+|y-x_j|)^{\kappa_2}},
\end{equation*}
where $x_i\neq x_j$, $\kappa_1\geq1$ and $\kappa_2\geq1$ are two constants.

\begin{lemma}(Lemma B.1, \cite{wei2010infinitely})\label{lemA.1}
For any constant $0<\varsigma\leq \min\{\kappa_1,\kappa_2\}$, there is a constant $C>0$, such that
\begin{equation*}
  g_{i,j}(y)\leq\frac{C}{|x_i-x_j|^\varsigma}\Big(\frac{1}{(1+|y-x_i|)^{\kappa_1+\kappa_2-\varsigma}}+\frac{1}{(1+|y-x_j|)^{\kappa_1+\kappa_2-\varsigma}}\Big).
\end{equation*}
\end{lemma}

\begin{lemma}(Lemma 2.1, \cite{guo2016infinitely})\label{lemA.2}
For any constant $0<\theta<N-2s$, there exist a constant $C>0$, such that
\begin{equation*}
  \int_{\R^N}\frac{1}{|y-z|^{N-2s}}\frac{1}{(1+|z|)^{2s+\theta}}dz \leq \frac{C}{(1+|y|)^{\theta}}.
\end{equation*}
\end{lemma}

Recall that for $j=1,\cdots,k$,
\begin{equation}\label{A.0}
x_j^\pm=\Big(\bar{r} \sqrt{1-\bar{h} ^2} \cos \frac{2(j-1)\pi}{k}, \bar{r}\sqrt{1-\bar{h}^2}\sin\frac{2(j-1)\pi}{k},\pm\bar{r}\bar{h},\bar{y}''\Big).
\end{equation}

\begin{lemma}\label{lemA.4}
For any $\gamma>0$, we have the following estimates for $k\rightarrow +\infty$,
\begin{equation}\label{A.1}
  \sum\limits_{j=2}^{k}\frac{1}{|x_j^+-x_1^+|^\gamma}
=\left\{
      \begin{array}{ll}
      O\Big(\frac{k^\gamma}{(\bar{r}\sqrt{1-\bar{h}^2})^\gamma}\Big), & \gamma>1; \vspace{0.12cm}\\
      O\Big( \frac{k \text{ln}k}{(\bar{r}\sqrt{1-\bar{h}^2})}\Big), & \gamma=1;\vspace{0.12cm} \\
      O\Big(\frac{k}{(\bar{r}\sqrt{1-\bar{h}^2})^\gamma}\Big) , & \gamma<1,
      \end{array}
    \right.
\end{equation}
and if $\bar{h}=o(1)$, $\frac{1}{\bar{h}k}=o(1)$, then for any $\gamma>1$
\begin{equation}\label{A.2}
  \sum\limits_{j=1}^{k}\frac{1}{|x_j^--x_1^+|^\gamma} = O\Big(\frac{k\bar{h}}{(\bar{r}\bar{h})^\gamma}\Big).
\end{equation}
\end{lemma}

\begin{proof}
From \eqref{A.0}, we have $|x_j^+-x_1^+|=2\bar{r} \sqrt{1-\bar{h} ^2}\sin\frac{(j-1)\pi}{k}$ for $j=2,\cdots,k$. Without loss of generality, we can assume $k$ is odd. For any $\gamma>0$, it holds
\begin{align*}
  \sum\limits_{j=2}^{k}\frac{1}{|x_j^+-x_1^+|^\gamma}
&= \frac{1}{(2\bar{r} \sqrt{1-\bar{h} ^2})^\gamma}\sum\limits_{j=2}^{k}\frac{1}{(\sin\frac{(j-1)\pi}{k})^\gamma}\\
&= \frac{2}{(2\bar{r} \sqrt{1-\bar{h} ^2})^\gamma}\sum\limits_{j=2}^{\frac{k+1}{2}}\frac{1}{(\sin\frac{(j-1)\pi}{k})^\gamma}.
\end{align*}
Since there exist two constants $c_1, c_2>0$, such that
\begin{equation*}
  0<c_1\frac{(j-1)\pi}{k}\leq\sin\frac{(j-1)\pi}{k}\leq c_2\frac{(j-1)\pi}{k},\quad \text{for~} j=2,\cdots,\frac{k+1}{2},
\end{equation*}
we obtain \eqref{A.1}.

Similarly, we have $|x_j^--x_1^+|^2=4\bar{r}^2 \Big((1-\bar{h} ^2)\sin^2\frac{(j-1)\pi}{k}+\bar{h}^2\Big)$ for $j=1,\cdots,k$. For any $\gamma>1$, it holds
\begin{align*}
  \sum\limits_{j=1}^{k}\frac{1}{|x_j^--x_1^+|^\gamma}
  &= \frac{2}{(2\bar{r}\bar{h})^\gamma}\sum\limits_{j=1}^{\frac{k+1}{2}}\frac{1}{(\frac{1-\bar{h}^2}{\bar{h}^2}\sin^2\frac{(j-1)\pi}{k}+1)^{\frac{\gamma}{2}}}\\
  &\leq\frac{2}{(2\bar{r}\bar{h})^\gamma}\frac{\bar{h}k}{\pi\sqrt{1-\bar{h}^2}}\int_0^{+\infty}\frac{1}{(1+z^2)^\frac{\gamma}{2}}dz   =O\Big(\frac{k\bar{h}}{(\bar{r}\bar{h})^\gamma}\Big).
\end{align*}
Then we have \eqref{A.2}.
\end{proof}

In particular, we have

\begin{lemma}\label{lemA.4.4.5}
Assume that $N>4s+1$ and $\bar{h},\bar{r}$ satisfies \eqref{1.7}. Then as $k\to \infty$
\begin{equation}\label{a4.3.1}
\sum\limits_{j=2}^{k}\frac{1}{|x_j^+-x_1^+|^{N-2s}}=\frac{A_1k^{N-2s}}{(\bar{r}\sqrt{1-\bar{h}^2})^{N-2s}}\Big(1+O\big(\frac{\ln k}{k^2}\big)\Big),
\end{equation}
and if $\frac{1}{k\bar{h}}=o(1)$, then
\begin{equation}\label{a4.3.a2}
  \begin{split}
  &\sum\limits_{j=1}^{k}\frac{1}{|x_j^--x_1^+|^{N-2s}}\\
  &=\frac{A_2k}{\bar{r}^{N-2s}\bar{h}^{N-2s-1}\sqrt{1-\bar{h}^2}}\Big(1+O\big(\frac{1}{\bar{h}k}\big)\Big)\Big)+O\Big(\frac{k^{N-2s-2}\ln k}{(\bar{r}\sqrt{1-\bar{h}^2})^{N-2s}}\Big),
  \end{split}
\end{equation}
and
\begin{equation}\label{a4.3.2}
  \begin{split}
  &\sum\limits_{j=1}^{k}\frac{1}{|x_j^--x_1^+|^{N-2s+2}}\\
  &=\frac{A_3k}{\bar{r}^{N-2s+2}\bar{h}^{N-2s+1}\sqrt{1-\bar{h}^2}}\Big(1+O\big(\frac{1}{\bar{h}k}\big)\Big)+O\Big(\frac{k^{N-2s}\ln k}{(\bar{r}\sqrt{1-\bar{h}^2})^{N-2s+2}}\Big),
  \end{split}
\end{equation}
where
\begin{equation*}
  A_1=\frac{2}{(2\pi)^{N-2s}}\sum\limits_{j=1}^{+\infty}\frac{1}{j^{N-2s}},~ A_2=\frac{2}{2^{N-2s}\pi}\int_0^{+\infty}\frac{1}{(1+t^2)^{\frac{N-2s}{2}}}dt,
\end{equation*}
and $A_3=\frac{(N-2s-1)}{4(N-2s)}A_2$ are some positive constants.
\end{lemma}

\begin{proof}
Since the proofs of \eqref{a4.3.1} and \eqref{a4.3.a2} are similar to Lemma A.1 in \cite{duan2023doubling}, we omit it.
Note that for $\gamma\geq 1,$
\begin{equation*}
    \int_0^{+\infty} \frac{1}{(1+z^2)^\frac{\gamma}{2}} \text{d} z =\frac{\Gamma(\frac{1}{2})\Gamma(\frac{\gamma-1}{2})}{\Gamma(\frac{\gamma}{2})},
    \end{equation*}
    where $\Gamma(x)$ is the Gamma function. Then, we have
\begin{align*}
&\sum_{j=1}^k  \frac{1}{|x_{1}^{+}-x_{j}^{-}|^{N-2s+2}}\\
&=\frac{2}{\left(2\bar{r}\bar{h}\right)^{N-2s+2}}\displaystyle\sum_{j=1}^{\left[\frac{k-1}{2}\right]} \frac{1}{\left(\frac{(1-\bar{h}^2)}{\bar{h}^2}\left(\sin\frac{(j-1)\pi}{k}\right)^2+1\right)^{\frac{^{N-2s+2}}{2}}}\\
&=\frac{2}{(2\bar{r}\bar{h})^{N-2s+2}}\frac{\bar{h}k}{\pi\sqrt{1-\bar{h}^2}}\int_0^{+\infty}\frac{1}{(1+z^2)^\frac{{N-2s+2}}{2}}dz\Big(1+O\big(\frac{1}{\bar{h}k}\big)\Big)\\
&\quad +O\Big(\frac{k^{N-2s}\ln k}{(\bar{r}\sqrt{1-\bar{h}^2})^{N-2s+2}}\Big)\\
&=\frac{A_3k}{\bar{r}^{N-2s+2} \bar{h}^{N-2s+1}(\sqrt{1-h^2})}\Big(1+O\big(\frac{1}{\bar{h}k}\big)\Big)+O\Big(\frac{k^{N-2s}\ln k}{(\bar{r}\sqrt{1-\bar{h}^2})^{N-2s+2}}\Big).
\end{align*}
\end{proof}

Similarly, we have
\begin{equation}\label{a4.3.2a}
  \begin{split}
&\sum\limits_{j=1}^{k}\frac{\sin^2(\frac{(j-1)\pi}{k})}{|x_j^--x_1^+|^{N-2s+2}}\\
&=\frac{A_4k}{\bar{r}^{N-2s+2}\bar{h}^{N-2s-1}(\sqrt{1-\bar{h}^2})^{3}}\Big(1+O\big(\frac{1}{\bar{h}k}\big)\Big)+O\Big(\frac{k^{N-2s-2}\ln k}{(\bar{r}\sqrt{1-\bar{h}^2})^{N-2s+2}}\Big),
  \end{split}
\end{equation}
where $A_4=\frac{A_2}{4(N-2s)}$.

\begin{lemma}\label{lemA.4.5}
We have the expansion, for $k \to \infty$, $j=2,\cdots,N$,
\begin{equation}\label{a4.4.1}
  \int_{\R^N}U_{x_1^+,\lambda}^{2_s^*-1}U_{x_j^+,\lambda}=\frac{A_5}{(\lambda|x_j^+ - x_1^+|)^{N-2s}} + O\Big(\frac{1}{(\lambda|x_j^+ - x_1^+|)^{N-\epsilon_0}}\Big),
\end{equation}
and for $j=1,\cdots,N$,
\begin{equation}\label{a4.4.2}
  \int_{\R^N}U_{x_1^+,\lambda}^{2_s^*-1}U_{x_j^-,\lambda}=\frac{A_5}{(\lambda|x_j^+ - x_1^-|)^{N-2s}} + O\Big(\frac{1}{(\lambda|x_j^+ - x_1^-|)^{N-\epsilon_0}}\Big),
\end{equation}
where $A_5=C_N\int_{\R^N}U_{0,1}^{2_s^*-1}$ and $\epsilon_0>0$ is a small constant.
\end{lemma}

\begin{proof}
Here we only prove \eqref{a4.4.1} and the proof of \eqref{a4.4.2} is similar. Denote $d_0:=\min\{\frac{|x_2^+-x_1^+|}{4},\frac{|x_1^--x_1^+|}{4}\}$. We have
\begin{align*}
  &\int_{\R^N}U_{x_1^+,\lambda}^{2_s^*-1}U_{x_j^+,\lambda}\\
  &=C^{2^*_s}_N\int_{\R^N}\Big(\frac{\lambda}{1+\lambda^2|y-x^+_1|^2}\Big)^{\frac{N+2s}{2}}\Big(\frac{\lambda}{1+\lambda^2|y-x^+_j|^2}\Big)^{\frac{N-2s}{2}}\\
  &=C^{2^*_s}_N\Big(\int_{B_{{d_0}}(x_1^+)}+\int_{\R^N\backslash B_{{d_0}}(x_1^+)}\Big)\Big(\frac{\lambda}{1+\lambda^2|y-x^+_1|^2}\Big)^{\frac{N+2s}{2}}\Big(\frac{\lambda}{1+\lambda^2|y-x^+_j|^2}\Big)^{\frac{N-2s}{2}}\\
  &=C^{2^*_s}_N\int_{B_{{ \lambda d_0}}(0)}\frac{1}{(1+z^2)^{\frac{N+2s}{2}}}\frac{1}{(1+z^2-2\lambda\langle z,x_j^+-x_1^+\rangle+\lambda^2|x_j^+-x_1^+|^2)^{\frac{N-2s}{2}}}\\
  &\quad +O\Big(\frac{1}{(\lambda|x_j^+-x_1^+|)^{N-\epsilon_0}}\Big)\\
  &=\frac{C^{2^*_s}_N}{(\lambda|x_j^+-x_1^+|)^{N-2s}}\int_{{B_{{\lambda d_0}}(0)}}\frac{1}{(1+z^2)^{\frac{N+2s}{2}}}\Bigg(1-\frac{N-2s}{2}\frac{1+z^2-2\lambda\langle z,x_j^+-x_1^+\rangle}{\lambda^{2}|x_j^+-x_1^+|^2}\\
  &\quad +O\bigg(\Big(\frac{1+z^2-2\lambda\langle z,x_j^+-x_1^+\rangle}{\lambda^2|x_j^+-x_1^+|^2}\Big)^2\bigg)\Bigg)+O\Big(\frac{1}{(\lambda|x_j^+-x_1^+|)^{N-\epsilon_0}}\Big)\\
  &=\frac{A_5}{(\lambda|x_j^+-x_1^+|)^{N-2s}}+O\Big(\frac{1}{(\lambda|x_j^+-x_1^+|)^{N-\epsilon_0}}\Big),
  \end{align*}
where $A_5=C_N\int_{\R^N}U_{0,1}^{2^*_s-1}$ and $\epsilon_0>0$ is a small constant.
\end{proof}

Similarly, we can also obtain, for $j=2,\cdots,N$,
\begin{equation}\label{a4.4.1b}
  \int_{\R^N}U_{x_1^+,\lambda}U_{x_j^+,\lambda}=O\Big(\frac{1}{\lambda^{N-2s}|x_j^+ - x_1^+|^{N-4s}} \Big),
\end{equation}
and for $j=1,\cdots,N$,
\begin{equation}\label{a4.4.2b}
  \int_{\R^N}U_{x_1^+,\lambda}U_{x_j^-,\lambda}=O\Big(\frac{1}{\lambda^{N-2s}|x_j^- - x_1^+|^{N-4s}} \Big).
\end{equation}

\begin{lemma}\label{lemA.4.5a}
We have the expansion, for $k \to \infty$
\begin{equation}\label{a4.4.1a}
  \int_{\R^N}U_{x_1^+,\lambda}^{2_s^*-2}U_{x_j^+,\lambda}\frac{\partial U_{x_1^+,\lambda}}{\partial y_l}=\frac{A_6(x_j^+-x_1^+)_l}{\lambda^{N-2s}|x_j^+-x_1^+|^{N-2s+2}}+O\Big(\frac{1}{(\lambda|x_j^+-x_1^+|)^{N+1-\epsilon_0}}\Big),
\end{equation}
and
\begin{equation}\label{a4.4.2a}
  \int_{\R^N}U_{x_1^+,\lambda}^{2_s^*-2}U_{x_j^-,\lambda}\frac{\partial U_{x_1^+,\lambda}}{\partial y_l}=\frac{A_6(x_j^--x_1^+)_l}{\lambda^{N-2s}|x_j^--x_1^+|^{N-2s+2}}+O\Big(\frac{1}{(\lambda|x_j^+-x_1^+|)^{N+1-\epsilon_0}}\Big),
\end{equation}
where $A_6=\frac{(N-2s)^2}{N+2s}A_5$ and $\epsilon_0>0$ is a small constant.
\end{lemma}

\begin{proof}
 We have
\begin{align*}
  &\int_{\R^N}U_{x_1^+,\lambda}^{2_s^*-2}U_{x_j^+,\lambda}\frac{\partial U_{x_1^+,\lambda}}{\partial y_l}\\
  &=C^{2^*_s}_N\int_{{B_{{\lambda d_0}}(0)}}\frac{(N-2s)^2\lambda z_l}{(1+z^2)^{\frac{N+2s}{2}+1}}\frac{\lambda\langle z,x_j^+-x_1^+\rangle}{(\lambda|x_j^+-x_1^+|)^{N-2s+2}}
  +O\Big(\frac{1}{(\lambda|x_j^+-x_1^+|)^{N+1-\epsilon_0}}\Big)\\
  &=\frac{A_6(x_j^+-x_1^+)_l}{\lambda^{N-2s}|x_j^+-x_1^+|^{N-2s+2}}+O\Big(\frac{1}{(\lambda|x_j^+-x_1^+|)^{N+1-\epsilon_0}}\Big),
\end{align*}
where $A_6=\frac{(N-2s)^2}{N}C_N^{2_s^*}\int_{\R^N}\frac{z^2}{(1+z^2)^{\frac{N+2s}{2}+1}}=\frac{(N-2s)^2}{N+2s}A_5$ and $\epsilon_0>0$ is a small constant.
\end{proof}

Similarly, we can also obtain, for $j=2,\cdots,N$,
\begin{equation}\label{a4.4.1c}
  \int_{\R^N}U_{x_1^+,\lambda}\frac{\partial U_{x_j^+,\lambda}}{\partial y_l}=O\Big(\frac{1}{\lambda^{N-2s}|x_j^+ - x_1^+|^{N-4s+1}} \Big),
\end{equation}
and for $j=1,\cdots,N$,
\begin{equation}\label{a4.4.2c}
  \int_{\R^N}U_{x_1^+,\lambda}\frac{\partial U_{x_j^-,\lambda}}{\partial y_l}=O\Big(\frac{1}{\lambda^{N-2s}|x_j^- - x_1^+|^{N-4s+1}} \Big).
\end{equation}

\begin{lemma}\label{lemA.3}
Suppose that $\tau \in (0,\frac{N-2s}{2})$, $y =(y_1,y_2,\cdots,y_N)$. Then there is a small $\theta>0$, such that
\begin{equation*}
  \int_{\R^N}\frac{Z_{\bar{r},\bar{h},\bar{y}'',\lambda}^{2_s^*-2}}{|y-z|^{N-2s}}
  \sum\limits_{j=1}^{k}\frac{1}{(1+\lambda|z-x_j^+|)^{\frac{N-2s}{2}+\tau}}dz \leq C \sum\limits_{j=1}^{k}\frac{1}{(1+\lambda|y-x_j^+|)^{\frac{N-2s}{2}+\tau+\theta}},
\end{equation*}
and
\begin{equation*}
  \int_{\R^N}\frac{Z_{\bar{r},\bar{h},\bar{y}'',\lambda}^{2_s^*-2}}{|y-z|^{N-2s}}
  \sum\limits_{j=1}^{k}\frac{1}{(1+\lambda|z-x_j^-|)^{\frac{N-2s}{2}+\tau}}dz \leq C \sum\limits_{j=1}^{k}\frac{1}{(1+\lambda|y-x_j^-|)^{\frac{N-2s}{2}+\tau+\theta}}.
\end{equation*}
\end{lemma}

\begin{proof}
Since proof of Lemma \ref{lemA.3} is similar to Lemma 2.2 of \cite{guo2016infinitely}, here we omit it.
\end{proof}

\begin{lemma}\label{lemA.4.6}
For $y \in \Omega_1^+$, there exists a constant $C>0$ such that
\begin{align*}
&\sum\limits_{j=2}^{k}U_{x_j^+,\lambda}+\sum\limits_{j=1}^{k}U_{x_j^-,\lambda}\\
&\leq\frac{C\lambda^{\frac{N-2s}{2}}}{(1+\lambda|y-x_1^+|)^{N-2s-\theta}}\Big(\sum\limits_{j=2}^{k}\frac{1}{(\lambda|x_1^+-x_j^+|)^{\theta}}+\sum\limits_{j=1}^{k}\frac{1}{(\lambda|x_1^+-x_j^-|)^{\theta}}\Big),
\end{align*}
where $\theta \in (1,N-2s)$.
\end{lemma}

\begin{proof}
For $y \in \Omega_1^+$ and $j=2,\cdots,k$, we have
\begin{equation*}
  |y-x_j^+|\geq|x_1^+ - x_j^+| - |y-x_1^+| \geq \frac{1}{4}|x_1^+ - x_j^+|, \,\quad \text{if~}\, |y-x_1^+| \leq\frac{1}{4}|x_1^+ - x_j^+|,
\end{equation*}
and
\begin{equation*}
  |y-x_j^+|\geq|y-x_1^+| \geq \frac{1}{4}|x_1^+ - x_j^+|, \,\quad \text{if~} \, |y-x_1^+| \geq\frac{1}{4}|x_1^+ - x_j^+|.
\end{equation*}
Similarly, we can get
\begin{equation*}
  |y-x_j^-|\geq \frac{1}{4}|x_1^+ - x_j^-|.
\end{equation*}
Hence, there holds
\begin{align*}
&\sum\limits_{j=2}^{k}U_{x_j^+,\lambda}+\sum\limits_{j=1}^{k}U_{x_j^-,\lambda}\\
&\leq \frac{C\lambda^{\frac{N-2s}{2}}}{(1+\lambda|y-x_1^+|)^{N-2s-\theta}}\Big(\sum\limits_{j=2}^{k}\frac{1}{(1+\lambda|y-x_j^+|)^{\theta}}+\sum\limits_{j=1}^{k}\frac{1}{(1+\lambda|y-x_j^-|)^{\theta}}\Big)\\
&\leq\frac{C\lambda^{\frac{N-2s}{2}}}{(1+\lambda|y-x_1^+|)^{N-2s-\theta}}\Big(\sum\limits_{j=2}^{k}\frac{1}{(\lambda|x_1^+-x_j^+|)^{\theta}}+\sum\limits_{j=1}^{k}\frac{1}{(\lambda|x_1^+-x_j^-|)^{\theta}}\Big).
\end{align*}
\end{proof}

\begin{lemma}(Lemma A.3, \cite{guo2020solutions})\label{lemC.1}
Let $\sigma>0.$ For any constant $0<\theta<N$, there exists a constant $C>0$, independent of $\sigma$, such that
\begin{equation*}
\int_{\R^N\backslash B_\sigma(y)}\frac{1}{|y-z|^{N+2s}}\frac{1}{(1+|z|)^\theta}dz\leq C\Big(\frac{1}{(1+|y|)^{\theta+2s}}+\frac{1}{\sigma^{2s}}\frac{1}{(1+|y|)^\theta}\Big).
\end{equation*}
\end{lemma}

\begin{lemma}(Lemma A.5, \cite{guo2020solutions})\label{lemA.7}
Suppose that $(y-x)^2+t^2=\rho^2$, $t>0$. Then there exists a constant $C>0$ such that
\begin{equation*}
  |\tilde{Z}_{x_j^\pm,\lambda}|\leq\frac{C}{\lambda^{\frac{N-2s}{2}}}\frac{1}{(1+|y-x_j^\pm|)^{N-2s}} \text{~and~} |\nabla\tilde{Z}_{x_j^\pm,\lambda}|\leq\frac{C}{\lambda^{\frac{N-2s}{2}}}\frac{1}{(1+|y-x_j^\pm|)^{N-2s+1}}.
\end{equation*}
\end{lemma}

\begin{lemma}\label{lemA.8}
For any $\delta>0$, there exists $\rho=\rho(\delta)\in(2\delta,5\delta)$ such that when $N>4s$,
\begin{equation*}
  \int_{\partial''B_\rho^+}t^{1-2s}|\tilde{\varphi}|^2dS \leq \frac{Ck\|\varphi\|_*^2}{\lambda^\tau} \text{~and~}
  \int_{\partial''B_\rho^+}t^{1-2s}|\nabla\tilde{\varphi}|^2dS \leq \frac{Ck\|\varphi\|_*^2}{\lambda^\tau},
\end{equation*}
where $C>0$ is a constant, depending on $\delta$.
\end{lemma}
\begin{proof}
The proof is similar to Lemma A.6 in \cite{guo2020solutions}.
\end{proof}

\section{Energy expansion}\label{secB}
\renewcommand{\theequation}{B.\arabic{equation}}

\noindent In this section, we give some estimates of the energy expansions for
the first derivatives of $ I(Z_{\bar{r},\bar{h},\bar{y}'',\lambda})$
about $\lambda,\bar{r},\bar{h}$  and $\bar{y}''_{j}(j=4,\cdots,N)$ respectively.

Recall that
\begin{equation*}
  I(u)=\frac{1}{2}\int_{\R^N}\big(|(-\Delta)^{\frac{s}{2}}u|^2 + V(y)u^2 \big)dy- \frac{1}{2_s^*}\int_{\R^N}(u)_+^{2_s^*} dy.
\end{equation*}

\begin{lemma}\label{lemB.0}
If $N>4s+1$, then
\begin{equation}\label{B0.0}
  \begin{split}
 I(Z_{\bar{r},\bar{h},\bar{y}'',\lambda})
&=k\bigg(B_0+\frac{B_1V(\bar{r},\bar{y}'')}{\lambda^{2s}}-\frac{B_2k^{N-2s}}{\lambda^{N-2s}(\sqrt{1-\bar{h}^2})^{N-2s}}\\
&\quad-\frac{B_3k}{\lambda^{N-2s}\bar{h}^{N-2s-1}\sqrt{1-\bar{h}^2}}+O\Big(\frac{1}{\lambda^{2s+\epsilon}}\Big)\bigg),
  \end{split}
\end{equation}
where
\begin{equation*}
  B_0=\frac{2s}{N}\int_{\R^N}U^{2^*_s}_{0,1},~B_1=\int_{\R^N}U^2_{0,1},~B_2=\frac{A_1C_N}{\bar{r}^{N-2s}}\int_{\R^N}U^{2^*_s-1}_{0,1},
\end{equation*}
and $B_3=\frac{A_2}{A_1}B_2$ are positive constants.
\end{lemma}

\begin{proof}
Denote $d_0:=\min\{\frac{|x_2^+-x_1^+|}{4},\frac{|x_1^--x_1^+|}{4}\}$. Using the symmetry, we obtain
\begin{equation}\label{B0.1}
  \begin{split}
  &I(Z_{\bar{r},\bar{h},\bar{y}'',\lambda})=I(Z^*_{\bar{r},\bar{h},\bar{y}'',\lambda})+O\Big(\frac{k}{\lambda^{2s+\epsilon}}\Big)\\
&=\int_{\R^N}\Big(\frac{1}{2}\big(|(-\Delta)^{\frac{s}{2}}Z^*_{\bar{r},\bar{h},\bar{y}'',\lambda}|^2 + V(y)(Z^*_{\bar{r},\bar{h},\bar{y}'',\lambda})^2\big)
- \frac{1}{2_s^*}(Z^*_{\bar{r},\bar{h},\bar{y}'',\lambda})^{2_s^*}\Big)+O\Big(\frac{k}{\lambda^{2s+\epsilon}}\Big)\\
&=k\Bigg(\int_{\R^N}U^{2^*_s}_{x_1^+,\lambda}+\int_{\R^N}U^{2^*_s-1}_{x_1^+,\lambda}\Big(\sum\limits_{j=2}^{k}U_{x_j^+,\lambda}+\sum\limits_{j=1}^{k}U_{x_j^-,\lambda}\Big)
+\int_{\R^N}V(y)U^{2}_{x_1^+,\lambda}\\
&\quad +\int_{\R^N}V(y)U_{x_1^+,\lambda}\Big(\sum\limits_{j=2}^{k}U_{x_j^+,\lambda}+\sum\limits_{j=1}^{k}U_{x_j^-,\lambda}\Big)\Bigg)
- \frac{1}{2_s^*}\int_{\R^N}(Z^{*}_{\bar{r},\bar{h},\bar{y}'',\lambda})^{2_s^*}+O\Big(\frac{k}{\lambda^{2s+\epsilon}}\Big).
  \end{split}
\end{equation}
Next, we calculate the terms one by one. By Lemma \ref{lemA.4.5}, we have
\begin{align}\label{B0.2}
  &\int_{\R^N}U^{2^*_s-1}_{x_1^+,\lambda}\Big(\sum\limits_{j=2}^{k}U_{x_j^+,\lambda}+\sum\limits_{j=1}^{k}U_{x_j^-,\lambda}\Big)\notag\\
 &=\sum\limits_{j=2}^{k}\frac{A_5}{(\lambda|x^+_j-x^+_1|)^{N-2s}}+\sum\limits_{j=1}^{k}\frac{A_5}{(\lambda|x^-_j-x^+_1|)^{N-2s}}\notag\\
 &\quad +O\Big(\sum\limits_{j=2}^{k}\frac{1}{(\lambda|x^+_j-x^+_1|)^{N-\tilde{\sigma}}}+\sum\limits_{j=1}^{k}\frac{1}{(\lambda|x^-_j-x^+_1|)^{N-\tilde{\sigma}}}\Big),
\end{align}
where $0<\tilde{\sigma}<N-4s$ is a small constant.
A direct Taylor expansion gives that
\begin{equation}\label{B0.3}
  \begin{split}
  &\int_{\R^N}V(y)U^{2}_{x_1^+,\lambda}
  =\int_{B_{d_0}(x_1^+)}V(y)U^{2}_{x_1^+,\lambda}+\int_{\R^N\setminus B_{d_0}(x_1^+)}V(y)U^{2}_{x_1^+,\lambda}\\
  &=
  \int_{B_{d_0}(x_1^+)}\Big(V(x_1^+)+\nabla V(x_1^+)\cdot(y-x_1^+)+O(|y-x_1^+|^2)\Big)U^{2}_{x_1^+,\lambda}+O\Big(\frac{1}{(\lambda d_0)^{N-2s}}\Big)\\
  &=\frac{V(\bar{r},\bar{y}'')}{\lambda^{2s}}\Big(\int_{\R^N}U^2_{0,1}+o(1)\Big)+O\Big(\frac{1}{(\lambda d_0)^{N-2s}}\Big).
  \end{split}
\end{equation}

Similarly, Using \eqref{a4.4.1b} and \eqref{a4.4.2b}, we obtain
\begin{equation}\label{B0.4}
  \begin{split}
  &\int_{\R^N}V(y)U_{x_1^+,\lambda}\Big(\sum\limits_{j=2}^{k}U_{x_j^+,\lambda}+\sum\limits_{j=1}^{k}U_{x_j^-,\lambda}\Big)\\
  &=O\bigg(\frac{1}{\lambda^{N-2s}}\Big(\sum\limits_{j=2}^{k}\frac{1}{|x^+_j-x^+_1|^{N-4s}}+\sum\limits_{j=1}^{k}\frac{1}{|x^-_j-x^+_1|^{N-4s}}\Big)\bigg).
  \end{split}
\end{equation}
Using the symmetry, together with Lemma \ref{lemA.4.6}, we have
\begin{equation}\label{B0.5}
  \begin{split}
  & \int_{\R^N}(Z^*_{\bar{r},\bar{h},\bar{y}'',\lambda})_+^{2_s^*}=2k\int_{\Omega_1^+}\Big(\sum\limits_{j=1}^{k}U_{x_j^+,\lambda}+\sum\limits_{j=1}^{k}U_{x_j^-,\lambda}\Big)^{2_s^*}\notag\\
  &=2k\Bigg\{\int_{\Omega_1^+}U^{2_s^*}_{x_1^+,\lambda}+2^*_s U^{2_s^*-1}_{x_1^+,\lambda}\Big(\sum\limits_{j=2}^{k}U_{x_j^+,\lambda}+\sum\limits_{j=1}^{k}U_{x_j^-,\lambda}\Big)\notag\\
  &\quad +\int_{\Omega_1^+}O\bigg(U^{2_s^*-2}_{x_1^+,\lambda}\Big(\sum\limits_{j=2}^{k}U_{x_j^+,\lambda}+\sum\limits_{j=1}^{k}U_{x_j^-,\lambda}\Big)^2+\Big(\sum\limits_{j=2}^{k}U_{x_j^+,\lambda}+\sum\limits_{j=1}^{k}U_{x_j^-,\lambda}\Big)^{2_s^*}\bigg)\Bigg\}\\
&=2k\Bigg\{\int_{\R^N}U^{2_s^*}_{0,1}+2^*_s\int_{\R^N}U^{2^*_s-1}_{x_1^+,\lambda}\Big(\sum\limits_{j=2}^{k}U_{x_j^+,\lambda}+\sum\limits_{j=1}^{k}U_{x_j^-,\lambda}\Big) +O\Big(\frac{1}{\lambda^N}\Big)\notag\\
&\quad  +O\Big(\sum\limits_{j=2}^{k}\frac{1}{(\lambda|x^+_j-x^+_1|)^{N-\tilde{\sigma}}}+\sum\limits_{j=1}^{k}\frac{1}{(\lambda|x^-_j-x^+_1|)^{N-\tilde{\sigma}}}\Big)\Bigg\}.\notag
  \end{split}
\end{equation}
Hence, combining \eqref{B0.1}-\eqref{B0.5}, we have
\begin{align*}
 & I(Z_{\bar{r},\bar{h},\bar{y}'',\lambda})\\
&= k\bigg(B_0+\frac{B_1V(\bar{r},\bar{y}'')}{\lambda^{2s}} -\sum\limits_{j=2}^{k}\frac{A_5}{(\lambda|x_j^+-x_1^+|)^{N-2s}} -\sum\limits_{j=1}^{k}\frac{A_5}{(\lambda|x_j^--x_1^+|)^{N-2s}} \\
&\quad +o\Big(\frac{1}{\lambda^{2s}}\Big)+O\Big(\frac{k^{N-\tilde{\sigma}}}{\lambda^{N-\tilde{\sigma}}}+\frac{k^{N-4s}}{\lambda^{N-2s}}\Big)+O\Big(\frac{1}{\lambda^{N-2s}{\bar{h}}^{N-2s}}\Big)\bigg),
\end{align*}
which, together with Lemma \ref{lemA.4}, gives \eqref{B0.0}.
\end{proof}

\begin{lemma}\label{lemB.1}
If $N>4s+1$, then
\begin{equation}\label{B1.1}
  \begin{split}
  \frac{\partial I(Z_{\bar{r},\bar{h},\bar{y}'',\lambda})}{\partial \lambda}
  &=k\bigg(-\frac{2sB_1V(\bar{r},\bar{y}'')}{\lambda^{2s+1}}+\frac{(N-2s)B_2k^{N-2s}}{\lambda^{N-2s+1}(\sqrt{1-\bar{h}^2})^{N-2s}}\\
    &\quad+\frac{(N-2s)B_3k}{\lambda^{N-2s+1}\bar{h}^{N-2s-1}\sqrt{1-\bar{h}^2}}+O\Big(\frac{1}{\lambda^{2s+1+\epsilon}}\Big)\bigg),
  \end{split}
\end{equation}
where $B_1$, $B_2$ and $B_3$ are the same positive constants in Lemma \ref{lemB.0}.
\end{lemma}

\begin{proof}
By a direct computation, we have
\begin{equation*}
  \begin{split}
&\frac{\partial I(Z^*_{\bar{r},\bar{h},\bar{y}'',\lambda})}{\partial \lambda}\\
&=\int_{\R^N}\Big\{V(y)Z_{\bar{r},\bar{h},\bar{y}'',\lambda}^{*}- \Big((Z_{\bar{r},\bar{h},\bar{y}'',\lambda}^*)^{2_s^*-1}-\sum\limits_{j=1}^{k}U_{x_j^+,\lambda}^{2_s^*-1}-\sum\limits_{j=1}^{k}U_{x_j^-,\lambda}^{2_s^*-1}\Big)\Big\}\frac{\partial Z_{\bar{r},\bar{h},\bar{y}'',\lambda}^{*}}{\partial \lambda}.
  \end{split}
\end{equation*}
By \eqref{a4.4.1b} and \eqref{a4.4.2b}, we can check that
\begin{align*}
&\int_{\R^N}V(y)Z_{\bar{r},\bar{h},\bar{y}'',\lambda}^{*}\frac{\partial Z_{\bar{r},\bar{h},\bar{y}'',\lambda}^{*}}{\partial \lambda}\\
&=2k\Big(\int_{\R^N}V(y)U_{x_1^+,\lambda}\frac{\partial U_{x_1^+,\lambda}}{\partial \lambda} +\int_{\R^N}V(y)\frac{\partial U_{x_1^+,\lambda}}{\partial \lambda}\big(\sum\limits_{j=2}^{k}U_{x_j^+,\lambda}+\sum\limits_{j=1}^{k}U_{x_j^-,\lambda}\big)\Big)\\
&=2k\bigg(V(x_1^+)\int_{\R^N}U_{x_1^+,\lambda}\frac{\partial U_{x_1^+,\lambda}}{\partial \lambda}+\int_{\R^N}(V(y)-V(x_1^+))U_{x_1^+,\lambda}\frac{\partial U_{x_1^+,\lambda}}{\partial \lambda}\\
  &\quad +O\Big(\frac{1}{\lambda^{N-2s+1}}\big(\sum\limits_{j=2}^{k} \frac{1}{|x_j^+-x_1^+|^{N-4s}}
+\sum\limits_{j=1}^{k} \frac{1}{|x_j^--x_1^+|^{N-4s}}\big)\Big)\bigg)\\
&=k\bigg(-\frac{2sB_1 V(\bar{r},\bar{y}'')}{\lambda^{2s+1}} + o\Big(\frac{1}{\lambda^{2s+1}}\Big)+O\Big(\frac{1}{\lambda^{N-2s+1}{d_0}^{N-2s}}\Big)\\
&\quad +O\Big(\frac{1}{\lambda^{N-2s+1}}\big(\sum\limits_{j=2}^{k} \frac{1}{|x_j^+-x_1^+|^{N-4s}}
+\sum\limits_{j=1}^{k} \frac{1}{|x_j^--x_1^+|^{N-4s}}\big)\Big)\bigg).
\end{align*}

By symmetry, together with \eqref{a4.4.1} and \eqref{a4.4.2}, we have
\begin{align*}
&\int_{\R^N}\Big((Z_{\bar{r},\bar{h},\bar{y}'',\lambda}^*)^{2_s^*-1}-\sum\limits_{j=1}^{k}U_{x_j^+,\lambda}^{2_s^*-1}-\sum\limits_{j=1}^{k}U_{x_j^-,\lambda}^{2_s^*-1}\Big)\frac{\partial Z_{\bar{r},\bar{h},\bar{y}'',\lambda}^{*}}{\partial \lambda}\\
&=2k\int_{\Omega^+_1}\Big((Z_{\bar{r},\bar{h},\bar{y}'',\lambda}^*)^{2_s^*-1}-\sum\limits_{j=1}^{k}U_{x_j^+,\lambda}^{2_s^*-1}-\sum\limits_{j=1}^{k}U_{x_j^-,\lambda}^{2_s^*-1}\Big)\frac{\partial Z_{\bar{r},\bar{h},\bar{y}'',\lambda}^*}{\partial \lambda}\\
&=2k\Bigg\{\int_{\Omega^+_1}(2_s^*-1)U_{x_1^+,\lambda}^{2_s^*-2}\Big(\sum\limits_{j=2}^{k}U_{x_j^+,\lambda}+\sum\limits_{j=1}^{k}U_{x_j^-,\lambda}\Big)\frac{\partial U_{x_1^+,\lambda}}{\partial \lambda} \\
&\quad +\int_{\Omega^+_1}(2_s^*-1)U_{x_1^+,\lambda}^{2_s^*-2}\Big(\sum\limits_{j=2}^{k}U_{x_j^+,\lambda}+\sum\limits_{j=1}^{k}U_{x_j^-,\lambda}\Big)\frac{\partial }{\partial \lambda}\Big(\sum\limits_{j=2}^{k}U_{x_j^+,\lambda}+\sum\limits_{j=1}^{k}U_{x_j^-,\lambda}\Big)\\
&\quad+ O\bigg(\lambda^{-1}\int_{\Omega^+_1}U_{x_1^+,\lambda}^{2_s^*-3}\Big(\sum\limits_{j=2}^{k}U_{x_j^+,\lambda}+\sum\limits_{j=1}^{k}U_{x_j^-,\lambda}\Big)^2\bigg)\\
&\quad +O\bigg(\lambda^{-1}\int_{\Omega^+_1}\Big(\sum\limits_{j=2}^{k}U_{x_j^+,\lambda}+\sum\limits_{j=1}^{k}U_{x_j^-,\lambda}\Big)^{2_s^*-1}\bigg)\Bigg\}\\
&=k\bigg(-\sum\limits_{j=2}^{k}\frac{(N-2s)A_5}{\lambda^{N-2s+1}|x_j^+-x_1^+|^{N-2s}} - \sum\limits_{j=1}^{k}\frac{(N-2s)A_5}{\lambda^{N-2s+1}|x_j^--x_1^+|^{N-2s}}\\
&\quad +O\Big(\frac{1}{\lambda^{N-2s+1}}\Big)
+O\Big(\lambda^{-1}\big(\sum\limits_{j=2}^{k} \frac{1}{(\lambda|x_j^+-x_1^+|)^{N-\tilde{\sigma}}}
+\sum\limits_{j=1}^{k} \frac{1}{(\lambda|x_j^--x_1^+|)^{N-\tilde{\sigma}}}\big)\Big)
 \bigg).
\end{align*}
Thus, we obtain that
\begin{align*}
  &\frac{\partial I(Z_{\bar{r},\bar{h},\bar{y}'',\lambda})}{\partial \lambda}=\frac{\partial I(Z^*_{\bar{r},\bar{h},\bar{y}'',\lambda})}{\partial \lambda}+O\Big(\frac{k}{\lambda^{2s+1+\epsilon}}\Big) \\
  &=k\Big(-\frac{2sB_1 V(\bar{r},\bar{y}'')}{\lambda^{2s+1}}
  +\sum\limits_{j=2}^{k}\frac{(N-2s)A_5}{\lambda^{N-2s+1}|x_j^+-x_1^+|^{N-2s}} +\sum\limits_{j=1}^{k}\frac{(N-2s)A_5}{\lambda^{N-2s+1}|x_j^--x_1^+|^{N-2s}}\\
 &\quad
  +o\Big(\frac{1}{\lambda^{2s+1}}\Big)+O\Big(\frac{k^{N-\tilde{\sigma}}}{\lambda^{N+1-\tilde{\sigma}}}+\frac{k^{N-4s}}{\lambda^{N-2s+1}}\Big)
  +O\Big(\frac{1}{\lambda^{N-2s+1}{\bar{h}}^{N-2s}}\Big)\Big),
\end{align*}
which, together with Lemma \ref{lemA.4.4.5}, yields that \eqref{B1.1}. The proof is completed.
\end{proof}

\begin{lemma}\label{lemB.1a}
If $N>4s+1$, then
\begin{equation}\label{B1a.1}
  \begin{split}
&\frac{\partial I(Z_{\bar{r},\bar{h},\bar{y}'',\lambda})}{\partial \bar{h}}\\
&=k\bigg(-\frac{(N-2s)B_2\bar{h}k^{N-2s}}{\lambda^{N-2s}(\sqrt{1-\bar{h}^2})^{N-2s+2}}
    +\frac{(N-2s-1)B_3k}{\lambda^{N-2s}\bar{h}^{N-2s}\sqrt{1-\bar{h}^2}}
+ O\Big(\frac{\bar{h}}{\lambda^{2s+\epsilon}}\Big)\bigg),
  \end{split}
\end{equation}
where $B_2$ and $B_3$ are the same positive constants in Lemma \ref{lemB.0}.
\end{lemma}

\begin{proof}
We have
\begin{align*}
&\frac{\partial I(Z_{\bar{r},\bar{h},\bar{y}'',\lambda}^{*})}{\partial \bar{h}}\\
&=\int_{\R^N}\Big\{V(y)Z_{\bar{r},\bar{h},\bar{y}'',\lambda}^{*}-\Big((Z_{\bar{r},\bar{h},\bar{y}'',\lambda}^*)^{2_s^*-1}-\sum\limits_{j=1}^{k}U_{x_j^+,\lambda}^{2_s^*-1}-\sum\limits_{j=1}^{k}U_{x_j^-,\lambda}^{2_s^*-1}\Big)\Big\}\frac{\partial Z_{\bar{r},\bar{h},\bar{y}'',\lambda}^{*}}{\partial \bar{h}}.
\end{align*}
Noting that
$$\Big\langle x_j^+-x_1^+,\frac{\partial(-x_1^+)}{\partial \bar{h}}\Big\rangle=-2\bar{r}^2\bar{h}\sin^2(\frac{(j-1)\pi}{k}),$$
and
$$\Big\langle x_j^+-x_1^+,\frac{\partial(-x_1^+)}{\partial \bar{h}}\Big\rangle=2\bar{r}^2\bar{h}\Big(1-\sin^2(\frac{(j-1)\pi}{k})\Big),$$
which, together with \eqref{a4.4.1c} and \eqref{a4.4.2c}, gives that
\begin{equation}\label{B1a.2}
  \begin{split}
&\int_{\R^N}V(y)Z_{\bar{r},\bar{h},\bar{y}'',\lambda}^{*}\frac{\partial Z_{\bar{r},\bar{h},\bar{y}'',\lambda}^{*}}{\partial \bar{h}}\\
&=2k\bigg(\int_{\R^N}V(y)U_{x_1^+,\lambda}\frac{\partial U_{x_1^+,\lambda}}{\partial \bar{h}}+\int_{\R^N}V(y)\frac{\partial U_{x_1^+,\lambda}}{\partial \bar{h}}\big(\sum_{j=2}^{k}U_{x_j^+,\lambda}+\sum_{j=1}^{k}U_{x_j^-,\lambda}\big)\bigg)\\
&=O\bigg(\frac{k\bar{h}}{\lambda^{N-2s}}\Big(\sum_{j=2}^{k} \frac{1}{|x_j^+-x_1^+|^{N-4s}}+\sum_{j=1}^{k}\frac{1}{|x_j^--x_1^+|^{N-4s+2}}\Big)\bigg) +O\Big(\frac{k\lambda}{(\lambda d_0)^{N-2s+1}}\Big)\\
&=O\Big(\frac{\bar{h}k^{N-4s+1}}{\lambda^{N-2s}}\Big)+O\Big(\frac{k}{\lambda^{N-2s}d_0^{N-2s+1}}\Big),
  \end{split}
\end{equation}
since we can compute
\begin{align*}
  &\int_{\R^N}V(y)U_{x_1^+,\lambda}\frac{\partial U_{x_1^+,\lambda}}{\partial \bar{h}} \\
  & =\frac{1}{2}\int_{B_{d_0}(x^+_1)}\Big(V(x^+_1)+\nabla V(x^+_1)\cdot(y-x^+_1)+O(|y-x^+_1|^2)\Big)\frac{\partial U^2_{x_1^+,\lambda}}{\partial \bar{h}}\\
  &\quad +O\Big(\int_{{\R^N}\backslash B_{d_0}(x^+_1)}V(y)U_{x_1^+,\lambda}\frac{\partial U_{x_1^+,\lambda}}{\partial \bar{h}}\Big)\\
  &=\frac{1}{2}\int_{B_{d_0}(x^+_1)}\Big(\nabla V(x^+_1)\cdot(y-x^+_1)\Big)\frac{\partial U^2_{x_1^+,\lambda}}{\partial \bar{h}}+O\Big(\frac{1}{\lambda^{N-2s}d_0^{N-2s+1}}\Big)\\
  &=-\frac{1}{2}\int_{{B_{d_0}(x^+_1)}}U^2_{x_1^+,\lambda}\Big(\frac{\partial V(\bar{r},\bar{y}'')}{\partial \bar{r}}\bar{r}\bar{h}-\frac{\partial V(\bar{r},\bar{y}'')}{\partial \bar{r}}\bar{r}\bar{h}\Big)+O\Big(\frac{1}{\lambda^{N-2s}d_0^{N-2s+1}}\Big)\\
  &=O\Big(\frac{1}{\lambda^{N-2s}d_0^{N-2s+1}}\Big).
\end{align*}

By symmetry and Lemma \ref{lemA.4.5a}, we have
\begin{align*}
&\int_{\R^N}\Big((Z_{\bar{r},\bar{h},\bar{y}'',\lambda}^*)^{2_s^*-1}-\sum\limits_{j=1}^{k}U_{x_j^+,\lambda}^{2_s^*-1}-\sum\limits_{j=1}^{k}U_{x_j^-,\lambda}^{2_s^*-1}\Big)\frac{\partial Z_{\bar{r},\bar{h},\bar{y}'',\lambda}^*}{\partial \bar{h}}\\
&=2k\bigg(\int_{\R^N}(2_s^*-1)U_{x_1^+,\lambda}^{2_s^*-2}\Big(\sum\limits_{j=2}^{k}U_{x_j^+,\lambda}+\sum\limits_{j=1}^{k}U_{x_j^-,\lambda}\Big)\frac{\partial U_{x^+_1,\lambda}}{\partial \bar{h}}\\
&\quad +O\Big(\sum\limits_{j=2}^{k}\frac{\bar{h}}{(\lambda|x_j^+-x_1^+|)^{N-\tilde{\sigma}}}+\sum\limits_{j=1}^{k}\frac{\bar{h}}{(\lambda|x_j^--x_1^+|)^{N-\tilde{\sigma}}}\Big)\bigg)\\
&=2k\bigg(-\sum\limits_{j=2}^{k}\frac{2(2_s^*-1) A_6\bar{r}^2\bar{h}}{\lambda^{N-2s}|x_j^+-x_1^+|^{N+2-2s}}\Big(\sin^2\big(\frac{(j-1)\pi}{k}\big)\Big)\\
&\quad +\sum\limits_{j=1}^{k}\frac{2(2_s^*-1) A_6\bar{r}^2\bar{h}}{\lambda^{N-2s}|x_j^--x_1^+|^{N+2-2s}}\Big(1-\sin^2\big(\frac{(j-1)\pi}{k}\big)\Big)\\
&\quad+O\Big(\sum\limits_{j=2}^{k}\frac{\bar{h}}{(\lambda|x_j^+-x_1^+|)^{N-\tilde{\sigma}}}\Big) +O\Big(\sum\limits_{j=1}^{k}\frac{\bar{h}}{(\lambda|x_j^--x_1^+|)^{N-\tilde{\sigma}}}\Big)\bigg)\\
 &=2k\bigg(-\sum\limits_{j=2}^{k}\frac{(N-2s)A_5}{\lambda^{N-2s}|x_j^+-x_1^+|^{N-2s}}\frac{\bar{h}}{2(1-\bar{h}^2)}\\
 &\quad   +\sum\limits_{j=1}^{k}\frac{2(N-2s)A_5\bar{r}^2\bar{h}}{\lambda^{N-2s}|x_j^--x_1^+|^{N+2-2s}}\Big(1-\sin^2\big(\frac{(j-1)\pi}{k}\big)\Big)\\
 &\quad +O\Big(\sum\limits_{j=2}^{k}\frac{\bar{h}}{(\lambda|x_j^+-x_1^+|)^{N-\tilde{\sigma}}}+\sum\limits_{j=1}^{k}\frac{\bar{h}}{(\lambda|x_j^--x_1^+|)^{N-\tilde{\sigma}}}\Big)\bigg)
\end{align*}
which, together with \eqref{B1a.2} and Lemma \ref{lemA.4.4.5}, gives
\begin{align*}
&\frac{\partial I(Z_{\bar{r},\bar{h},\bar{y}'',\lambda})}{\partial \bar{h}}=\frac{\partial I(Z^*_{\bar{r},\bar{h},\bar{y}'',\lambda})}{\partial \bar{h}}+O\Big(\frac{k\bar{h}}{\lambda^{2s+\epsilon}}\Big) \\
&=k\bigg(-\frac{(N-2s)B_2\bar{h}k^{N-2s}}{\lambda^{N-2s}(\sqrt{1-\bar{h}^2})^{N-2s+2}}
    +\frac{(N-2s-1)B_3k}{(\lambda\bar{h})^{N-2s}\sqrt{1-\bar{h}}}
    -\frac{B_3k}{\lambda^{N-2s}\bar{h}^{N-2s-2}(\sqrt{1-\bar{h}})^3} \\
&\quad  +O\Big(\frac{\bar{h}k^{N-4s}}{\lambda^{N-2s}}\Big)
        +O\Big(\frac{\bar{h}k^{N-\tilde{\sigma}}}{\lambda^{^{N-\tilde{\sigma}}}}\Big)
        +O\Big(\frac{1}{\lambda^{N-2s}\bar{h}^{N-2s+1}}\Big)+O\Big(\frac{\bar{h}}{\lambda^{2s+\epsilon}}\Big)\bigg).
\end{align*}
If $\lambda \sim k^{\frac{N-2s}{N-4s}}$ and $\bar{h}\sim k^{-\frac{N-2s-1}{N-2s+1}}$, then the term $\frac{B_3k}{\lambda^{N-2s}\bar{h}^{N-2s-2}(\sqrt{1-\bar{h}})^3}$ can be absorbed in $O\Big(\frac{\bar{h}}{\lambda^{2s+\epsilon}}\Big)$, and we obtain the result.
\end{proof}

Note that $Z_{i,l}^\pm=O(\lambda Z_{x_i^\pm,\lambda})$, $l=2,4,\cdots,N$. Similarly, we can prove the following lemma.

\begin{lemma}\label{lemB.3}
If $N>4s+1$, then we have
\begin{equation*}
  \begin{split}
  \frac{\partial I(Z_{\bar{r},\bar{h},\bar{y}'',\lambda})}{\partial \bar{r}}&=2k\bigg(\frac{B_1}{\lambda^{2s}}\frac{\partial V(\bar{r},\bar{y}'')}{\partial\bar{r}}+\sum\limits_{j=2}^{k}\frac{(N-2s)A_5}{\bar{r}\lambda^{N-2s}|x_j^+-x_1^+|^{N-2s}}\\
 &\quad +\sum\limits_{j=1}^{k}\frac{(N-2s)A_5}{\bar{r}\lambda^{N-2s}|x_j^--x_1^+|^{N-2s}}+O\Big(\frac{1}{\lambda^{s+\epsilon}}\Big)\bigg),
  \end{split}
\end{equation*}
and
\begin{equation*}
  \frac{\partial I(Z_{\bar{r},\bar{h},\bar{y}'',\lambda})}{\partial \bar{y}^{''}_j}=2k\bigg(\frac{B_1}{\lambda^{2s}}\frac{\partial V(\bar{r},\bar{y}'')}{\partial\bar{y}^{''}_j}+O\Big(\frac{1}{\lambda^{s+\epsilon}}\Big)\bigg).
\end{equation*}
\end{lemma}

\section{Local Pohozaev identities for the fractional Laplacian}\label{secC}

\renewcommand{\theequation}{C.\arabic{equation}}
In this section, we establish the local Pohozaev identities for the fractional Laplacian.
\begin{lemma}\label{lem3.0}
There hold that for $i=4,\cdots,N$
\begin{equation}\label{C.1}
  \begin{split}
    &-\int_{\partial''{B_\rho^+}} t^{1-2s}\frac{\partial\tilde{u}_k}{\partial \nu}\frac{\partial\tilde{u}_k}{\partial y_i}+\frac{1}{2}\int_{\partial''{B_\rho^+}} t^{1-2s}|\nabla\tilde{u}_k|^2\nu_i\\
    &=\int_{B_\rho}\Big(-V(r,y'')u_k+(u_k)_+^{2_s^*-1}+\sum\limits_{l=1}^N
c_l\sum\limits_{j=1}^k\big(Z_{x_j^+,\lambda}^{2_s^*-2}Z_{j,l}^+ +Z_{x_j^-,\lambda}^{2_s^*-2}Z_{j,l}^-\big)\Big)\frac{\partial u_k}{\partial y_i},
  \end{split}
\end{equation}
and
\begin{equation}\label{C.2}
  \begin{split}
    &\int_{\partial''{B_\rho^+}} \Big(-t^{1-2s}\langle\nabla\tilde{u}_k,Y\rangle\frac{\partial\tilde{u}_k}{\partial \nu}+\frac{1}{2}t^{1-2s}|\nabla\tilde{u}_k|^2\langle Y,\nu\rangle\Big) +\frac{2s-N}{2}\int_{\partial{B_\rho^+}}t^{1-2s}\frac{\partial\tilde{u}_k}{\partial \nu}\tilde{u}_k\\
    &=\int_{{B_\rho}}\Big(-V(r,y'')u_k+(u_k)_+^{2_s^*-1}+\sum\limits_{l=1}^N
c_l\sum\limits_{j=1}^k\big(Z_{x_j^+,\lambda}^{2_s^*-2}Z_{j,l}^+ +Z_{x_j^-,\lambda}^{2_s^*-2}Z_{j,l}^-\big)\Big)\langle \nabla u_k,y\rangle.
  \end{split}
\end{equation}
\end{lemma}

\begin{proof}
The proof is standard. For the readers' convenience, we still give the proof. Multiplying \eqref{3.1} by $\frac{\partial\tilde{u}_k}{\partial y_i}~(i=4,\cdots,N)$, integrating by parts on the domain ${B_\rho^+}$, we have
\begin{equation}\label{3.3A}
  \begin{split}
    0 &= -\int_{{B_\rho^+}} \text{div}(t^{1-2s}\nabla\tilde{u}_k)\frac{\partial\tilde{u}_k}{\partial y_i}\\
    &=-\int_{\partial{B_\rho^+}}t^{1-2s}\frac{\partial\tilde{u}_k}{\partial \nu}\frac{\partial\tilde{u}_k}{\partial y_i}+\int_{{B_\rho^+}}t^{1-2s}\nabla\tilde{u}_k\cdot\nabla\frac{\partial\tilde{u}_k}{\partial y_i}\\
    &=-\int_{\partial{B_\rho^+}}t^{1-2s}\frac{\partial\tilde{u}_k}{\partial \nu}\frac{\partial\tilde{u}_k}{\partial y_i}+\frac{1}{2}\int_{\partial{B_\rho^+}}t^{1-2s}|\nabla\tilde{u}_k|^2 \nu_i\\
    &=-\int_{\partial''{B_\rho^+}}t^{1-2s}\frac{\partial \tilde{u}_k}{\partial \nu}\frac{\partial \tilde{u}_k}{\partial y_i} +\int_{{B_\rho}}\lim\limits_{t\to0^+}t^{1-2s}\frac{\partial\tilde{u}_k}{\partial t}\frac{\partial\tilde{u}_k}{\partial y_i} +\frac{1}{2}\int_{\partial''{B_\rho^+}}t^{1-2s}|\nabla\tilde{u}_k|^2 \nu_i,
  \end{split}
\end{equation}
where we used $\nu_i=0$ and $\frac{\partial\tilde{u}_k}{\partial \nu}=-\frac{\partial\tilde{u}_k}{\partial t}$ on $\partial'{B_\rho^+}$.
We complete the proof of \eqref{C.1}.

Similarly, multiplying \eqref{3.1} by $\langle\nabla\tilde{u}_k,Y\rangle$, integrating by parts on the domain ${B_\rho^+}$, we have
\begin{equation}\label{3.3B}
  \begin{split}
    0 &= -\int_{{B_\rho^+}} \text{div}(t^{1-2s}\nabla\tilde{u}_k)\langle\nabla\tilde{u}_k,Y\rangle\\
    &=-\int_{\partial{B_\rho^+}}t^{1-2s}\frac{\partial\tilde{u}_k}{\partial \nu}\langle\nabla\tilde{u}_k,Y\rangle +\int_{{B_\rho^+}}t^{1-2s}\nabla\tilde{u}_k\cdot\nabla\langle\nabla\tilde{u}_k,Y\rangle\\
    &=\int_{\partial{B_\rho^+}}t^{1-2s}\Big(-\frac{\partial\tilde{u}_k}{\partial \nu}\langle\nabla\tilde{u}_k,Y\rangle +\frac{1}{2}|\nabla\tilde{u}_k|^2 \langle Y,\nu\rangle\Big) -\frac{N-2s}{2}\int_{{B_\rho^+}}t^{1-2s}|\nabla\tilde{u}_k|^2. \\
  \end{split}
\end{equation}
On the other hand, we have that
\begin{equation}\label{3.3C}
  \begin{split}
    0 &= -\int_{{B_\rho^+}} \text{div}(t^{1-2s}\nabla\tilde{u}_k)\tilde{u}_k=-\int_{\partial{B_\rho^+}}t^{1-2s}\frac{\partial\tilde{u}_k}{\partial \nu}\tilde{u}_k +\int_{{B_\rho^+}}t^{1-2s}|\nabla\tilde{u}_k|^2.
  \end{split}
\end{equation}
Note that $\langle Y,\nu \rangle=0$ on $\partial'{B_\rho^+}$. Inserting \eqref{3.3C} into \eqref{3.3B}, we obtain that
\begin{equation}\label{3.3D}
  \begin{split}
  &\frac{N-2s}{2}\int_{\partial{B_\rho^+}}t^{1-2s}\frac{\partial\tilde{u}_k}{\partial \nu}\tilde{u}_k\\
  &= -\int_{\partial{B_\rho^+}}t^{1-2s}\frac{\partial\tilde{u}_k}{\partial \nu}\langle\nabla\tilde{u}_k,Y\rangle +\frac{1}{2}\int_{\partial{B_\rho^+}}t^{1-2s}|\nabla\tilde{u}_k|^2 \langle Y,\nu\rangle \\
    &=-\int_{\partial''{B_\rho^+}} t^{1-2s}\langle\nabla\tilde{u}_k,Y\rangle\frac{\partial\tilde{u}_k}{\partial \nu}
    +\int_{{B_\rho}}\lim\limits_{t\to0^+}(t^{1-2s}\frac{\partial\tilde{u}_k}{\partial t} )\langle\nabla u_k,y\rangle\\
    &\qquad +\frac{1}{2}\int_{\partial''{B_\rho^+}}t^{1-2s}|\nabla\tilde{u}_k|^2\langle Y,\nu\rangle.
  \end{split}
\end{equation}
We complete the proof of \eqref{C.2}.
\end{proof}

\section*{Acknowledgments}
The authors would like to thank the referee for the useful comments and suggestions and thank Yang Zhou for the helpful discussion with him.
This paper was supported by National Key Research and Development of China (No. 2022YFA1006900) and NSFC grants (No.12071169).

\end{document}